\documentclass[11pt,english,a4paper, leqno]{article}
\usepackage{amssymb,amsfonts,amsmath,amsthm,amscd}
\usepackage{MnSymbol}
\usepackage[T1]{fontenc}
\usepackage[latin9]{inputenc}
\usepackage{geometry}
\usepackage{xcolor}
\usepackage{pdfcolmk}
\usepackage{textcomp}

\usepackage{PSFigure,psfrag}
\usepackage{MnSymbol}
\usepackage[T1]{fontenc}
\usepackage[latin9]{inputenc}
\usepackage{geometry}
\usepackage{xcolor}
\usepackage{pdfcolmk}
\usepackage{textcomp}
\usepackage{graphicx}

\usepackage{esint}
\PassOptionsToPackage{normalem}{ulem}
\usepackage{ulem}
\def\f{\footnotesize}
\def\IZ{{\mathbb Z}}
\def\IR{{\mathbb R}}
\def\IP{{\mathbb P}}
\def\IN{{\mathbb N}}
\def\IL{{\mathbb L}}
\def\cA{{\cal A}}
\def\cB{{\cal B}}
\def\cC{{\cal C}}
\def\cD{{\cal D}}

\def\cF{{\cal F}}
\def\cI{{\cal I}}
\def\cK{{\cal K}}
\def\cL{{\cal L}}
\def\cS{{\cal S}}
\def\cU{{\cal U}}
\def\cV{{\cal V}}

\def\n{\noindent}
\def\dis{\displaystyle}
\def\ov{\overline}
\def\wt{\widetilde}
\def\wh{\widehat}
\def\dsl{\textstyle\sum\limits}
\def\r{\rightarrow}
\def\ve{\varepsilon}
\def\point{{\mbox{\large $.$}}}
\makeatletter

\providecolor{lyxadded}{rgb}{0,0,1}
\providecolor{lyxdeleted}{rgb}{1,0,0}

\numberwithin{equation}{section}
\numberwithin{figure}{section}
\theoremstyle{plain}
\newtheorem{thm}{\protect\theoremname}[section]
  \theoremstyle{plain}
  \newtheorem{lem}[thm]{\protect\lemmaname}
  \theoremstyle{plain}
  \newtheorem{cor}[thm]{\protect\corollaryname}
  \theoremstyle{plain}
  \newtheorem{prop}[thm]{\protect\propositionname}

\newtheorem{remark}[thm]{Remark}

\theoremstyle{remark}

\title{\large{\textbf{SOLIDIFICATION OF POROUS INTERFACES  \\ AND DISCONNECTION}}}

\date{}

\usepackage[english]{babel}

\addtocounter{section}{-1}

\makeatother

\usepackage{babel}
  \providecommand{\corollaryname}{Corollary}
  \providecommand{\lemmaname}{Lemma}
  \providecommand{\propositionname}{Proposition}
  
\providecommand{\theoremname}{Theorem}

\begin{document}
\maketitle \thispagestyle{empty}
\begin{center} \vspace{-0.7cm}  Maximilian Nitzschner and Alain-Sol Sznitman
\end{center}  
\begin{center}
\end{center}
\begin{abstract}
In this article we obtain uniform estimates on the absorption of Brownian motion by porous interfaces surrounding a compact set. An important ingredient is the construction of certain resonance sets, which are hard to avoid for Brownian motion starting in the compact set. As an application of our results, we substantially strengthen the results of \cite{Szni17}, and obtain when $d \ge 3$, large deviation upper bounds on the probability  that simple random walk in $\IZ^d$, or random interlacements in $\IZ^d$, when their vacant set is in a strongly percolative regime, disconnect the discrete blow-up of a regular compact set from the boundary of the discrete blow-up of a box containing the compact set in its interior. Importantly, we make no convexity assumption on the compact set. It is plausible, although open at the moment, that the upper bounds that we derive in this work match in principal order the lower bounds of \cite{LiSzni14} in the case of random interlacements, and of \cite{Li17} for the simple random walk.
\end{abstract}

\vspace{4cm}

\vfill
\noindent
{\small
-------------------------------- \\
Departement Mathematik, ETH Z\"urich, CH-8092 Z\"urich, Switzerland. \\
}

\newpage
\thispagestyle{empty} \mbox{}
\newpage  \pagestyle {plain}

\section{Introduction}

The notion of capacity has many facets. Given a compact subset $A$ of $\IR^d$, $d \ge 3$, its Brownian capacity ${\rm cap}(A)$ can for instance be viewed as a measure of its size for Brownian motion coming from far away, see \cite{PortSton78}, p.~58. But it also equals the infimum of the Dirichlet energies of smooth compactly supported functions equal to $1$ on a neighborhood of $A$, see \cite{Szni98a}, p.~87, and also alternatively the inverse of the infimum of the energies of probability measures supported by $A$, see \cite{Szni98a}, p.~76. 
It is a classical fact of potential theory that any compact set $S$ that separates $A$ from infinity has a Brownian capacity ${\rm cap}(S)$, which is bounded from below by ${\rm cap}(A)$. However, when such an ``interface'' $S$ is replaced by a ``porous deformation'' $\Sigma$, uniform comparisons between ${\rm cap}(A)$ and ${\rm cap}(\Sigma)$ become more delicate, in part due to the possible degenerations of the interfaces and porous interfaces under consideration. When $A$ is convex, the problem can sometimes be circumvented with the help of the projection attached to $A$, but the issue becomes acute when no convexity assumption is made on $A$. The challenge is then to bring into play a notion of porous interfaces, which has meaningful consequences, and is relevant for applications.

In this article we make no convexity assumption on $A$. We introduce a notion of {\it porous interfaces} $\Sigma$ surrounding $A$, and obtain uniform estimates for the absorption by  $\Sigma$ of Brownian motion starting in $A$. These {\it solidification estimates} readily lead to uniform comparisons between ${\rm cap}(A)$ and ${\rm cap}(\Sigma)$. An important ingredient in proving such solidification estimates is the construction of certain {\it resonance sets}, which are hard to avoid for Brownian motion starting in $A$, and where on many scales the local densities of the interior and the exterior of the {\it segmentation} underlying $\Sigma$ remain balanced.

As an application of these estimates, we are able to substantially strengthen the results of \cite{Szni17} and obtain large deviation upper bounds on the probability that simple random walk in $\IZ^d$, or random interlacements in $\IZ^d$, when their vacant set is in a strongly percolative regime, disconnect the discrete blow-up of a regular compact set $A$ from the boundary of the discrete blow-up of a box containing $A$ in its interior. Whereas the results of \cite{Szni17} handled the case when $A$ is itself a box, and the methods of \cite{Szni17} might conceivably have been extended to handle the case of a regular compact convex set $A$, the case treated here, when $A$ is a regular compact set, requires a genuinely new approach to the coarse graining procedure, which is employed. Quite plausibly, the upper bounds that we derive in this work are sharp and match in principal order the large deviation lower bounds obtained in \cite{Li17} and \cite{LiSzni14}.

We will now describe our results in a more precise form. We consider $\IR^d$, mainly when $d \ge 3$ (although $d \ge 2$ throughout Section 1). We consider a non-empty compact subset $A$ of $\IR^d$, and a collection of ``interfaces'' $S = \partial U_0 = \partial U_1$, where $U_0$ is a bounded Borel subset of $\IR^d$ and $U_1$ its complement. To control the ``distance'' of $A$ to $U_1$, we define for $\ell_*$ non-negative integer
\begin{equation}\label{0.1}
\begin{split}
\cU_{\ell_*,A} = & \;\mbox{the collection of bounded Borel subsets $U_0$ for which the}
\\[-0.5ex]
&\; \mbox{local density of $U_1$ for any box centered at a point of $A$ with}
\\[-1ex]
&\; \mbox{side-length $2 \cdot 2^{-\ell}$ is at most $\frac{1}{2}$ when $\ell \ge \ell_*$}
\end{split}
\end{equation}
(this last condition is for instance satisfied when the sup-distance of $A$ to $U_1 = \IR^d \backslash U_0$ exceeds $2^{-\ell_*}$).

\bigskip
At a heuristic level the ``interface'' $S = \partial U_0 = \partial U_1$ can be viewed as a kind of segmentation of the ``porous interfaces'', which we now introduce. For a given $U_0 \in \cU_{\ell_*,A}$, and $\ve > 0$, $\eta \in (0,1)$ respectively measuring the distance from $S$ at which the porous interface is felt, and the strength with which it is felt, we consider in the hard obstacle case
\begin{equation}\label{0.2}
\begin{split}
\cS_{U_0,\ve,\eta} = & \;\mbox{the collection of compact subsets $\Sigma$ of $\IR^d$ such that}
\\
&\; \mbox{$P_x[H_\Sigma < \tau_\ve] \ge \eta$ for all $x \in S$ ($= \partial U_0$)},
\end{split}
\end{equation}

\bigskip
\n
where $P_x$ stands for the Wiener measure starting from $x$, governing the canonical Brownian motion $X$ on $\IR^d$, $H_\Sigma = \inf\{s \ge 0; X_s \in \Sigma\}$ for the entrance time of $X$ in $\Sigma$, and $\tau_\ve = \inf\{s \ge 0; |X_s - X_0|_\infty \ge \ve\}$ for the first time $X$ moves at sup-distance $\ve$ from its starting point.

\psfrag{s}{$S$}
\psfrag{e}{$\varepsilon$}
\psfrag{A}{$A$}
\psfrag{S}{$\Sigma$}
\psfrag{U0}{$U_0$}
\psfrag{U1}{$U_1$}
\psfrag{2ell}{$2^{-\ell_*}$}
\begin{center}
\includegraphics[width=13cm]{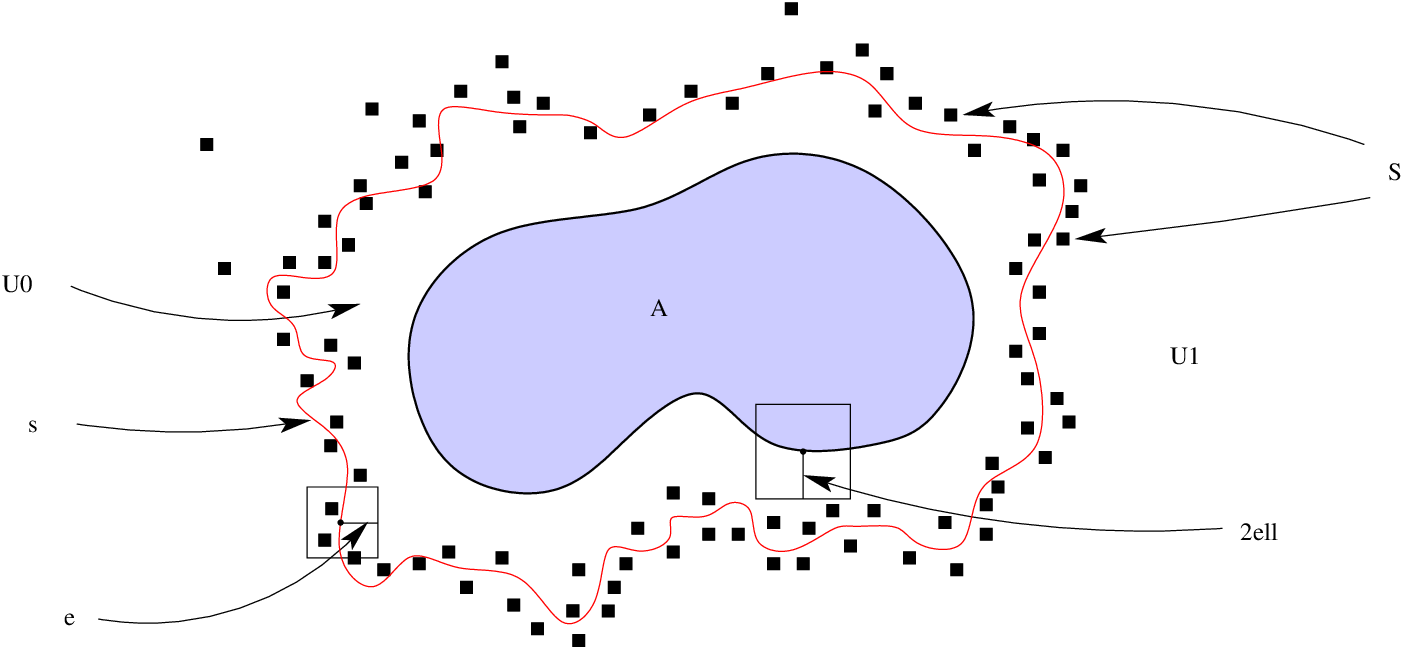}
\end{center}

\bigskip
\begin{center}
Fig.~1: An illustration of a $U_0$ in $\cU_{\ell_*,A}$ and $\Sigma$ in $\cS_{U_0,\ve,\eta}$
\end{center}

\bigskip
In the soft obstacle case, we instead consider

\begin{equation}\label{0.3}
\begin{split}
\cV_{U_0,\ve,\eta} = & \;\mbox{the collection of non-negative, locally bounded},
\\[-0.5ex]
&\; \mbox{measurable functions $V$ on $\IR^d$, such that}
\\
&\mbox{$E_x [\exp \{-\int^{\tau_\ve}_0 V(X_s) \,ds\}] \le 1-\eta$, for all $x \in S$ ($= \partial U_0$)},
\end{split}
\end{equation}

\bigskip\n
where $E_x$ stands for the expectation relative to the measure $P_x$.

\bigskip\medskip
Asymptotic solidification estimates play a central role in this work. They appear in Theorem \ref{theo3.1} and for given $A$ and $\eta$ they provide controls in the limit where $\ve/2^{-\ell_*}$ tends to zero, on the trapping probability of Brownian motion starting in $A$ by the porous interface, uniformly over $\Sigma \in \cS_{U_0,\ve,\eta}$ or $V \in \cV_{U_0,\ve,\eta}$, $U_0 \in \cU_{\ell_*,A}$, and over the starting point $x$ in $A$. Namely, for given $A$ and $\eta$, we show that
\begin{equation}\label{0.4}
\lim\limits_{u \r 0} \;\; \wt{\sup}  \;P_x [H_\Sigma = \infty] = 0
\end{equation}
(where $\wt{\sup}$ stands for the supremum over $\ve \le u\, 2^{-\ell_*}$, $U_0 \in \cU_{\ell_*,A}$, $\Sigma \in \cS_{U_0,\ve,\eta}$, and $x \in A$), and
\begin{equation} \label{0.5}
\lim\limits_{u \r 0} \;\; \wh{\sup} \;E_x \Big[\exp\Big\{- \dis\int^\infty_0 V(X_s)\,ds\Big\}\Big] = 0 
\end{equation}

\smallskip\n
(where $\wh{\sup}$ stands for the supremum over $\ve \le u\, 2^{-\ell_*}$, $U_0 \in \cU_{\ell_*,A}$, $V\in \cV_{U_0,\ve,\eta}$, and $x \in A$), and actually, the quantities under the ``$\lim_{u \r 0}$'' sign in both (\ref{0.4}) and (\ref{0.5}) are maximal when $A = \{0\}$ (see also Remark \ref{rem3.6} 2) for a reformulation of (\ref{0.4}), (\ref{0.5}), using scaling when $A = \{0\}$).

\medskip
As an application of (\ref{0.4}), we show in Corollary \ref{cor3.4} an asymptotic lower bound on capacity that plays a pivotal role in our treatment of the disconnection problems that we consider in Section 4. Namely when ${\rm cap}(A) > 0$, with a similar meaning as below (\ref{0.4}), one has
\begin{equation}\label{0.6}
 \lim\limits_{u \r 0} \;\; \wt{\inf} \; {\rm cap}(\Sigma) / {\rm cap}(A)= 1.
 \end{equation}
 
\medskip \n
 Importantly, no convexity assumption is made on $A$. When $A$ is convex, asymptotic lower bounds on ${\rm cap}(\Sigma)$ are often easier to achieve, as explained in Remark \ref{rem3.6} 3).
 
\bigskip
As a (straightforward) illustration of (\ref{0.5}), we consider the time spent by Brownian motion in the $\ve$-neighborhood of $S$ $(= \partial U_0)$ for the sup-distance, and show in Corollary \ref{cor3.5} that for any $a > 0$ and $\ell_* \ge 0$, 
 \begin{equation}\label{0.7}
\lim\limits_{\ve \r 0} \;\; \sup\limits_{U_0 \in \,  \cU_{\ell_*,A}} \;\;\sup\limits_{x \in A} \;\;  E_x \Big[\exp\Big\{- \mbox{\f $\dis\frac{a}{\ve^2}$} \;\dis\int^\infty_0 1\{d(X_s, S) \le \ve\}\,ds\Big\} \Big] 
= 0.
 \end{equation}
 
\medskip \n
The difficulty in proving (\ref{0.4}) (or (\ref{0.5})) stems from the fact that $U_0$ varies over the whole class $\cU_{\ell_*,A}$ and the interface $S$ as well as the porous interface $\Sigma$ may undergo degenerations and become brittle in certain parts of space (where, for instance, they can behave as a soft potential due to an effect of a ``constant capacity regime'', see Remark \ref{rem3.6} 1)).
 
\medskip
An important ingredient in the proof of Theorem \ref{theo3.1} (cf.~(\ref{0.4}), (\ref{0.5})) is the construction of certain resonance sets, where on well-separated spatial scales the local densities of both $U_0$ and its complement $U_1$ are non-degenerate. Specifically, given $U_0$ and integers $J \ge 1$, $L \ge L(J)$ (see (\ref{1.27})) and $I \ge 1$ the resonance set is (see (\ref{2.5}))
\begin{equation}\label{0.8}
{\rm Res} = \Big\{x \in \IR^d; \dsl_{\ell \in \cA_*} 1\{\wt{\sigma}_\ell(x) \in [\wt{\alpha}, 1 - \wt{\alpha}]\} \ge J\Big\},
\end{equation}

\medskip\n
where $\wt{\alpha} = \frac{1}{3} \,4^{-d}$ is a dimensional constant, $\wt{\sigma}_\ell(x)$ denotes the local density of $U_1$ in the closed box $B(x,4 \cdot 2^{-\ell})$ of center $x$ and side-length $8 \cdot 2^{-\ell}$, and $\ell$ ranges over $\cA_* \subseteq L \IN$, which in essence, see (\ref{2.3}) for the precise definition, consists of the first $I(J+1)$ integers in $L \IN$ that are bigger or equal to $\ell_*$.

\bigskip
The resonance set need not separate $A$ from infinity (see Remark \ref{rem2.4}~1)), but crucially, as shown in Theorem \ref{theo2.1}, it is hard for Brownian motion starting in $A$ to avoid the resonance set when $I$ is large. We even obtain stretched-exponential bounds on the avoidance probability and show in Theorem \ref{theo2.1} that for any $J \ge 1$,
\begin{align}
&\underset{I \r \infty}{\ov{\lim}} \;I^{-1/2^{J-1}} \log \big(\sup\limits_{\ell_* \ge 0} \;\;\sup\limits_{U_0 \in \cU_{\ell_*,A}} \;\; \sup\limits_{L \ge L(J)} \;\; \sup\limits_{x \in A} \;P_x[H_{\rm Res} = \infty]\big) \le - c(J) < 0 \label{0.9}
\\
&\mbox{(with $H_{\rm Res} = \inf\{s \ge 0; X_s \in {\rm Res}\}$ the entrance time of $X$ in the} \nonumber
\\[-0.5ex]
&\;\mbox{resonance set Res)}. \nonumber
\end{align}

\n
Actually, as shown in Theorem \ref{theo2.1}, the quantity under the logarithm is maximal when $A = \{0\}$.

\medskip
Equipped with the crucial estimate (\ref{0.9}), we can use the resonance set in (\ref{0.8}) as a substitute for a Wiener criterion, and infer that the porous interfaces $\Sigma$, or the soft obstacles $V$, whose presence is felt in the $\ve$-vicinity of $S = \partial U_0 = \partial U_1$, have massive trapping power as quantified by (\ref{0.4}) and (\ref{0.5}), see Theorem \ref{theo3.1} and its proof.

\medskip
In Section 4, we apply Corollary \ref{cor3.4} (see (\ref{0.6})), to a disconnection problem for simple random walk on $\IZ^d$, or for random interlacements on $\IZ^d$, when their vacant set $\cV^u$ is in the strongly percolative regime $u \in (0,\ov{u})$, with $\ov{u}$ the critical level from \cite{Szni17}, cf.~(\ref{2.3}) in this reference. We consider a non-empty compact subset $A$ of $\IR^d$ contained in the interior of a box of side-length $2M$ centered at the origin, as well as the discrete blow-up of $A$ and the interior boundary of the discrete blow-up of the above mentioned box:
\begin{equation}\label{0.10}
A_N = (NA) \cap \IZ^d \;\; \mbox{and} \;\; S_N = \{x \in \IZ^d; |x|_\infty = [MN]\}
\end{equation}
(where $[\cdot]$ denotes the integer part).

\bigskip
In the case of random interlacements, we are interested in the disconnection event corresponding to the absence of paths in $\cV^u$ between $A_N$ and $S_N$, denoted by
\begin{equation}\label{0.11}
\cD^u_N = \{A_N \stackrel{\cV^u}{\longleftrightarrow} \hspace{-3ex} \mbox{\f $/$} \quad S_N\}
\end{equation}

\medskip\n
($\cD^u_N$ coincides with the full sample space if $A_N$ happens to be empty). 

\bigskip
When the compact set $A$ is regular in the sense that
\begin{equation}\label{0.12}
{\rm cap}(A) = {\rm cap}(\mathring{A}) \quad \mbox{(with $\mathring{A}$ the interior of $A$)},
\end{equation}

\medskip\n
we show that for $u \in (0,\ov{u})$, that is in the strongly percolative regime for $\cV^u$, one has
\begin{equation}\label{0.13}
\limsup\limits_N \;\;\mbox{\f $\dis\frac{1}{N^{d-2}}$} \; \log \IP [\cD^u_N] \le - \mbox{\f $\dis\frac{1}{d}$} \; (\sqrt{\ov{u}} - \sqrt{u})^2 \,{\rm cap}(A).
\end{equation}
Our results are actually more general, see Theorem \ref{theo4.1} and also Remark \ref{rem4.5} 3).

\bigskip
In the case of the simple random walk, we similarly consider the vacant set $\cV$ that is the complement of the set of sites visited by the canonical walk on $\IZ^d$. We are then interested in the disconnection event corresponding to the absence of paths in $\cV$ between $A_N$ and $S_N$
\begin{equation}\label{0.14}
\cD_N = \{A_N \stackrel{\cV}{\longleftrightarrow} \hspace{-3ex} \mbox{\f $/$} \quad S_N\}
\end{equation}

\medskip\n
(and $\cD_N$ coincides with the full canonical space if $A_N$ happens to be empty).

\medskip
As a rather straightforward consequence of (\ref{0.13}) (this formally corresponds to letting $u \r 0$), we show in Corollary \ref{cor4.4} that when $A$ fulfills the regularity condition (\ref{0.12}), then for any $x$ in $\IZ^d$
\begin{equation}\label{0.15}
\limsup\limits_N \;\; \mbox{\f $\dis\frac{1}{N^{d-2}}$} \; \log P_x^{\IZ^d} [\cD_N] \le - \mbox{\f $\dis\frac{\ov{u}}{d}$} \;{\rm cap}(A)
\end{equation}
(with $P_x^{\IZ^d}$ the canonical law of simple random walk starting at $x$).

\bigskip
Again, our results are more general, see Corollary \ref{cor4.4} and Remark \ref{rem4.5} 3). Actually, it is plausible that the upper bounds (\ref{0.13}), (\ref{0.15}) catch the correct principal exponential decay of the probabilities under consideration, and match the lower bounds respectively derived in \cite{LiSzni14} and \cite{Li17}. This feature rests on the identification (open at the moment) of $\ov{u}$ with the box-to-box critical level $u_{**}$ for the percolation of the vacant set of random interlacements, so that one would then have the equalities $\ov{u} = u_\ast = u_{\ast\ast}$, with $u_*$ the critical level for the percolation of the vacant set of random interlacements. Incidentally, the recent work \cite{DumiRaouTass17} might lead to some progress towards a proof of $u_* = u_{**}$. The coincidence of the three critical levels would then show that under (\ref{0.12}), 
\begin{align*}
&\lim\limits_{N \r \infty} \;  \mbox{\f $\dis\frac{1}{N^{d-2}}$} \; \log \IP [\cD^u_N] = -\mbox{\f $\dis\frac{1}{d}$} \;(\sqrt{u}_* - \sqrt{u})^2 \, {\rm cap} (A), \;\mbox{for $0 < u < u_*$},
\\
\intertext{and}
\\[-4ex]
&\lim\limits_{N \r \infty} \;  \mbox{\f $\dis\frac{1}{N^{d-2}}$} \; \log P_x^{\IZ^d} [\cD_N] = - \mbox{\f $\dis\frac{u_*}{d}$} \; {\rm cap}(A), \;\mbox{for $x$ in $\IZ^d$}.
\end{align*}

\bigskip
The lower bounds in \cite{LiSzni14} for the case of random interlacements, and \cite{Li17} for the case of simple random walk are based on the change of probability method. They respectively involve probability measures $\wt{\IP}_N$ (in the case of random interlacements) and $\wt{P}_N$ (in the case of simple random walk) implementing certain ``strategies'' to produce disconnection. If the critical values $\ov{u} \le u_* \le u_{**}$ actually coincide, the results of the present work show that these strategies are (nearly) optimal and thus hold special significance. We briefly sketch what these strategies are.

\medskip
In the case of random interlacements, the measures $\wt{\IP}_N$ from \cite{LiSzni14} correspond to so-called {\it tilted interlacements} that are slowly space-modulated interlacements at a level (slowly varying over space) equal to $f^2_N(x) = (\sqrt{u} + (\sqrt{u}_{**} - \sqrt{u}) \, h(\frac{x}{N}))^2$, $x \in \IZ^d$, where $h$ on $\IR^d$ is the solution of the equilibrium problem
\begin{equation}\label{0.16}
\left\{ \begin{array}{l}
\mbox{$\Delta h = 0$ outside $\wt{A}$}~,
\\[1ex]
\mbox{$h = 1$ on $\wt{A}$ and $h$ tends to zero at infinity},
\end{array}\right.
\end{equation}

\n
with $\wt{A}$ a slight thickening of $A$. Informally, the tilted interlacements create a ``fence'' around $A_N$ on which they locally behave as interlacements at level $u_{**}$ (one actually chooses $u_{**} + \delta$ in place of $u_{**}$ in the formula for $f_N$, with $\delta$ a small positive number). They produce locally on this fence a strongly non-percolative regime for the vacant set, and typically disconnect $A_N$ from $S_N$, when $N$ is large. Informally, the tilted interlacements correspond to a Poisson cloud of bilateral trajectories, which are governed by the generator $\wt{L}_N \,g(x) = \frac{1}{2d} \sum_{|x' - x| = 1} \frac{f_N(x')}{f_N(x)} \,(g(x') - g(x))$, instead of the discrete Laplacian (corresponding to the case when $f_N$ is a constant function).

\bigskip
In the case of the simple random walk, the measures $\wt{P}_N$ from \cite{Li17} correspond to {\it tilted walks} that in essence behave as a walk with generator $\ov{L}_N \,g(x) =$ \linebreak $\frac{1}{2d} \sum_{|x' - x| = 1} \frac{h_N(x')}{h_N(x)} \,(g(x') - g(x))$ up to a deterministic time $T_N$ and afterwards move as a simple random walk, where $h_N(x) = h(\frac{x}{N})$, with $h$ as in (\ref{0.16}) and $T_N$ is chosen such that the expected time spent by the tilted walk at a point $x \in A_N$ amounts to $(u_{**} + \delta) \,h^2_N(x) = u_{**} + \delta$ (with the choice of $h$ in (\ref{0.16})). Again, this creates a ``fence'' around $A_N$, where the vacant set left by the tilted walk at time $T_N$ is locally in a strongly non-percolative regime, and typically disconnects $A_N$ from $S_N$, for large $N$ (one actually chooses a compactly supported approximation of $h$ in (\ref{0.16}) for the definition of the tilted walk).

\bigskip
Let us emphasize that no convexity assumption is made on $A$ in the derivation of the upper bounds (\ref{0.13}), (\ref{0.15}). This represents an important progress on \cite{Szni17}, where $A$ was assumed to be a box. We use a different approach to the coarse graining of the disconnection event $\cD^u_N$. The notion of porous interfaces and Corollary \ref{cor3.4} (cf.~(\ref{0.6})) play a pivotal role in the derivation of (\ref{0.13}). The scale $\ve$ in (\ref{0.6}) roughly corresponds to $\frac{\wh{L}_0}{N}$, cf.~(\ref{4.19}), where $\wh{L}_0$ is a scale slightly smaller than $N$, on which we perform a ``segmentation'' of an interface of ``blocking boxes'' of a much smaller scale $L_0$, slightly bigger than $(N \log N)^{\frac{1}{d-1}}$, cf.~(\ref{4.19}), which is present when the disconnection event occurs. In a sense we follow a refined version of the strategy for the tracking of interfaces underpinning Section 2 of \cite{DembSzni06}. After removal of a ``bad event'' of negligible probability from the disconnection event $\cD^u_N$, we obtain the effective disconnection event $\wt{\cD}^u_N$, cf.~(\ref{4.41}).  We partition $\wt{\cD}^u_N$ into a not too large collection of events $\cD_{N,\kappa}$ (where $\kappa$ runs over a set of $\exp\{ o(N^{d-2})\}$ elements). This partition embodies the coarse graining, cf.~(\ref{4.45}). In essence, each event $\cD_{N,\kappa}$ corresponds to the selection of discrete ``blocking boxes'' of side-length $L_0$ having ``substantial presence'' in each of the selected boxes of side-length $2\wh{L}_0$ (see Figure 2). If $C$ denotes the union of all selected boxes of side-length $L_0$, we have an exponential estimate on the probability of each $\cD_{N,\kappa}$ in terms of ${\rm cap}_{\IZ^d}(C)$, cf. (\ref{4.14}). With the help of Proposition \ref{propA.1} in the Appendix, we can in essence replace ${\rm cap}_{\IZ^d}(C)$ by $\frac{1}{d}N^{d-2} {\rm cap}(\Sigma)$, where $\Sigma$ (the porous interface) corresponds to the solid $\IR^d$-filling of $\frac{C}{N}$. At this point crucially Corollary \ref{cor3.4} enables us to obtain an asymptotic uniform lower bound of ${\rm cap}(\Sigma)$ in terms of ${\rm cap}(A)$. In our application of Corollary \ref{cor3.4}, it should be underlined that both the interface $S$ and the porous interface $\Sigma$ are ``constructs of the coarse graining procedure'', cf.~(\ref{4.48}), (\ref{4.49}) (in particular $S$ is by no means a priori given). We refer to Section 4 below (\ref{4.26}) for a more detailed outline of the proof of Theorem \ref{theo4.1} (cf.~(\ref{0.13})).

\bigskip
\psfrag{AN}{$A_N$}
\psfrag{SN}{$S_N$}
\psfrag{L0}{\small{$2\widehat{L}_0$}}
\psfrag{LaN}{\small{$\widehat{L}_0$ slightly smaller than $N$}}
\psfrag{Lsk}{\small{$L_0$ slightly bigger than $(N \log N)^{\frac{1}{d-1}}$}}
\includegraphics[width=8cm]{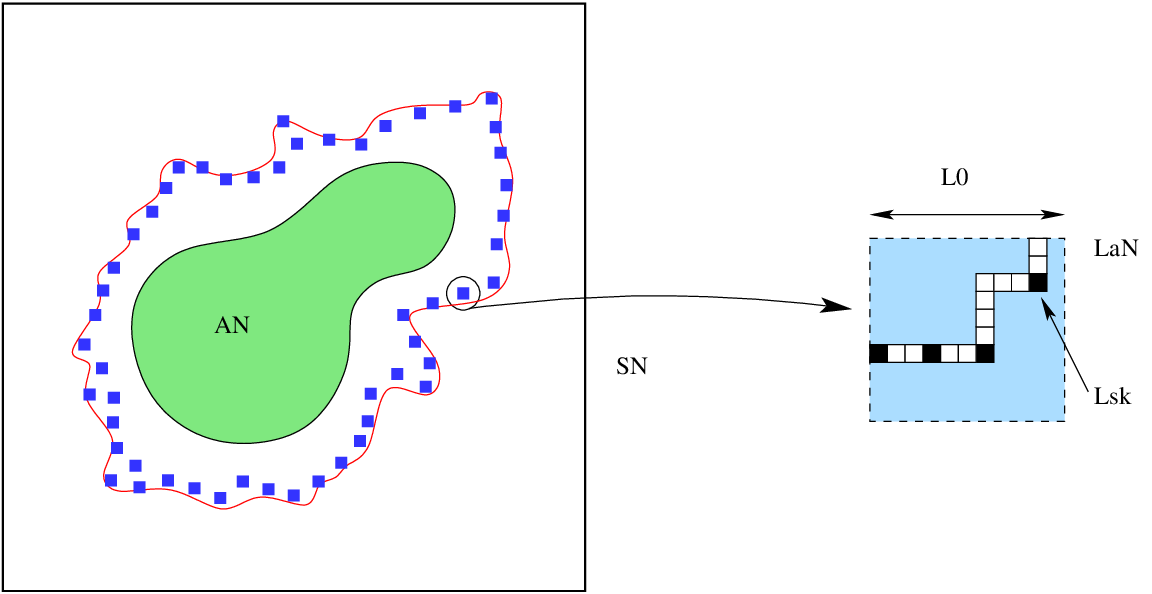}

\bigskip
\begin{center}
\begin{tabular}{lp{12cm}}
Fig.~2: & An illustration of some features entering the definition of the event $\cD_{N,\kappa}$ with the selected boxes of side-length $2\wh{L}_0$ on the 
left-hand side, and the blow-up of one such box with the selected boxes of side-length $L_0$ (in black), on the right-hand side.
 \end{tabular}
\end{center}

\bigskip
Plausibly, the methods of this article might be pertinent in the context of level-set percolation of the Gaussian free field and lead to an improvement of the results of \cite{Szni15}, see Remark \ref{rem4.5} 1). One may also wonder whether in the spirit of the Wulff shape theory for the existence of a large finite cluster at the origin in supercritical Bernoulli percolation, cf.~\cite{Cerf00} and the references therein, some insight might be gained concerning the behavior of a large finite connected component at the origin for the vacant set of random interlacements $\cV^u$ in the strongly percolative regime $u \in (0,\ov{u})$.

\medskip
Let us now describe the organization of this article. Section 1 collects several results concerning the local density functions. In particular, Proposition \ref{prop1.4} provides in essence a lower bound for the probability that Brownian motion enters the resonance set when starting at a good point. In Section 2, the main Theorem \ref{theo2.1} shows that when starting in $A$, Brownian motion can hardly avoid the resonance set, cf.~(\ref{0.9}). Lemma \ref{lem2.2} contains an important induction step for the proof of Theorem \ref{theo2.1}. In Section 3, we introduce the notion of porous interfaces and show the central solidification estimates in Theorem \ref{theo3.1}, see also (\ref{0.4}), (\ref{0.5}). We then provide a first application with the capacity lower bound (\ref{0.6}) in Corollary \ref{cor3.4} and a second (and quicker) application to the estimate (\ref{0.7}) in Corollary \ref{cor3.5}. In section 4, we derive with the help of Corollary \ref{cor3.4} the large deviation upper bounds (\ref{0.13}), (\ref{0.15}) on the probability of the disconnection events (\ref{0.11}), (\ref{0.14}) in Theorem \ref{theo4.1} and Corollary \ref{cor4.4}. Additional estimates appear in Remark \ref{rem4.5} 3). In the Appendix, we provide in Proposition \ref{propA.1} an asymptotic comparison between the simple random walk capacity of arbitrary finite unions of well-separated large boxes and the Brownian capacity of the solid filling of these boxes.

\medskip
Finally, let us explain our convention concerning constants. We denote by $c,c',\wt{c}$ positive constants changing from place to place that simply depend on $d$. Numbered constants such as $c_0,c_1,c_2,\dots$ refer to the value corresponding to their first appearance in the text. Dependence of constants on additional parameters appears in the notation.

\section{Local density functions}

In this section we develop several controls concerning certain local density functions that will help us track the presence of interfaces. The main result is contained in Proposition \ref{prop1.4}. In essence, it provides a lower bound on the probability that Brownian motion starting at a point where the local density on a certain scale is well-balanced, visits before moving at a distance comparable to that scale a point where on several well-separated smaller scales the local densities are well-balanced. This result will be of direct use in the proof of Theorem \ref{theo2.1} in the next section. We begin with some notation.

\medskip
Throughout this section we assume that $d \ge 2$. We denote by $|x|_\infty$ and $|x|_1$ the respective sup-norm and $\ell_1$-norm of $x \in \IR^d$. For $r \ge 0$, we let $B(x,r)$ stand for the closed ball in sup-norm with center $x$ and radius $r$. When $A$ is a subset of $\IR^d$, we write $\ov{A}$ for its closure, $\mathring{A}$ for the interior, and $\partial A = \ov{A} \backslash \mathring{A}$ for its boundary. We let $d(x,A)= \inf \{|x-y|_\infty$; $y \in A\}$ denote the sup-norm distance of $x$ to $A$. When $A$ is a Borel subset of $\IR^d$, we sometimes write $|A|$ for its Lebesgue measure.

\medskip
We denote by $(X_t)_{t \ge 0}$ the canonical $\IR^d$-valued process on $C(\IR_+,\IR^d)$, the space of continuous $\IR^d$-valued functions on $\IR_+$, which we endow with the canonical $\sigma$-algebra $\cF$ and the canonical right-continuous filtration denoted by $(\cF_t)_{t \ge 0}$. We let $(\theta_t)_{t \ge 0}$ stand for the canonical shift (so that $\theta_t(w)(\cdot) = w(t + \cdot)$, for $w \in C(\IR_+, \IR^d)$  and $t \ge 0$). We denote by $P_x$ the Wiener measure starting from $x \in \IR^d$, and by $E_x$ the corresponding expectation. So, under $P_x$ the process $(X_t)_{t \ge 0}$ is the canonical Brownian motion on $\IR^d$ starting from $x$. When $A$ is a closed subset of $\IR^d$, we write $H_A = \inf\{s \ge 0, X_s \in A\}$ for the entrance time of $X$ in $A$, $\wt{H}_A = \inf\{s > 0$; $X_s \in A\}$ for the hitting time of $A$, and when $U$ is an open subset of $\IR^d$, we write $T_U = \inf\{s \ge 0$; $X_s \notin U\}$ ($= H_{U^c})$ for the exit time of $U$. They all are $(\cF_t)$-stopping times.

\medskip
We consider
\begin{equation}\label{1.1}
\mbox{$U_0 \not= \emptyset$, a bounded Borel subset of $\IR^d$ and $U_1 = \IR^d \backslash U_0$, its complement}.
\end{equation}
The boundary
\begin{equation}\label{1.2}
S = \partial U_0 = \partial U_1
\end{equation}

\n
is a compact subset of $\IR^d$. In Section 3 the porous interface corresponding to $\Sigma$ in the hard obstacle case, or to $V$ in the soft potential case, will be ``palpably'' present in the vicinity of every point of $S$. At a heuristic level, $S$ will be some kind of ``segmentation'' for the porous interface.

\medskip
In this section we are mainly concerned with the development of various controls on certain local density functions (in dyadic scales), which we now introduce. For $x \in \IR^d$ and $\ell \in \IZ$, we set 
\begin{equation}\label{1.3}
\left\{ \begin{array}{l}
\wh{\sigma}_\ell(x) = |B(x,2^{-\ell}) \cap U_1| \, / \, |B(x,2^{-\ell})|,
\\[1ex]
\wt{\sigma}_\ell(x) = |B(x,4 \cdot 2^{-\ell}) \cap U_1| \, / \, |B(x,4 \cdot 2^{-\ell})| = \wh{\sigma}_{\ell - 2}(x)\,.
\end{array}\right.
\end{equation}

\medskip\n
Eventually, our main interest will lie in the case $\ell \ge 0$, i.e.~in the local scale picture, but for the time being in this section, we do not impose any restriction on the sign of $\ell$. As a shorthand for the average of a locally integrable function $f$ on a sup-norm ball $B(x,2^{-\ell})$, with $x \in \IR^d$, $\ell \in \IZ$, we write
\begin{equation}\label{1.4}
\strokedint_{B(x,2^{-\ell})} f(y) \, dy = \dis\frac{1}{|B(x,2^{-\ell})|} \; \dis\int_{B(x,2^{-\ell})} f(y)\,dy
\end{equation}

\medskip\n
and we introduce the normalized Lebesgue measure restricted to $B(x,2^{-\ell})$
\begin{equation}\label{1.5}
\mu_{x,\ell}(dy) = \dis\frac{1}{|B(x,2^{-\ell})|} \; 1_{B(x,2^{-\ell})}(y) \, dy\,.
\end{equation}

\medskip\n
We first collect some facts concerning the Lipschitz character of the local densities $\wh{\sigma}_\ell(\cdot)$ and $\wt{\sigma}_\ell(\cdot)$, and we relate $\wh{\sigma}_{\ell}(\cdot)$ to the average of $\wh{\sigma}_{\ell'}$ on $B(x,2^{-\ell})$ when $\ell ' > \ell$.

\begin{lem}\label{lem1.1}
For $x,y \in \IR^d$, $\ell \in \IZ$, one has
\begin{align}
| \wh{\sigma}_\ell (x + y) - \wh{\sigma}_\ell(x)| & \le 2^{\ell} |y|_1, \label{1.6}
\\[1ex]
| \wt{\sigma}_\ell (x + y) - \wt{\sigma}_\ell(x)| & \le 2^{\ell -2 } |y|_1. \label{1.7}
\end{align}

\medskip\n
Moreover, for $\ell ' > \ell$ in $\IZ$ and $x \in \IR^d$ one has
\begin{equation}\label{1.8}
|\wh{\sigma}_\ell(x) - \strokedint_{B(x,2^{-\ell})} \wh{\sigma}_{\ell '}(y) \, dy| \le c_0 \,2^{\ell - \ell'}, \;\mbox{where} \;\, c_0 = d \, 2^{d-1}.
\end{equation}
\end{lem}

\begin{proof}
We begin with the proof of (\ref{1.6}), and note that the claim follows from the fact that for $y$ collinear to a vector of the canonical basis of $\IR^d$,
\begin{equation}\label{1.9}
\begin{split}
|\wh{\sigma}_\ell (x + y) - \wh{\sigma}_\ell(x)| & \le \dis\frac{1}{|B(x,2^{-\ell})|}\; |B(x,2^{-\ell}) \,\Delta B(x + y, 2^{-\ell})|
\\[1ex]
&\quad \;\; \mbox{(with $\Delta$ the symmetric difference)}
\\[2ex]
& \le \dis\frac{2 |y|_1 \,(2\cdot 2^{-\ell})^{d-1}}{(2 \cdot 2^{-\ell})^d} = 2^\ell |y|_1\,.
\end{split}
\end{equation}

\medskip\n
The claim (\ref{1.6}) readily follows. The claim (\ref{1.7}) follows as well since $\wt{\sigma}_\ell = \wh{\sigma}_{\ell - 2}$, cf.~(\ref{1.3}). We now turn to the proof of (\ref{1.8}). We consider $\ell ' > \ell$ in $\IZ$ and note that for $x \in \IZ^d$,
\begin{equation}\label{1.10}
\begin{split}
\strokedint_{B(x,2^{-\ell})} \; \wh{\sigma}_{\ell '}(y) \, dy  = &\; |B(0,2^{-\ell})|^{-1} |B(0,2^{-\ell '})|^{-1}
\\[-0.5ex]
&\; \dis\int 1\{|y -x|_\infty \le 2^{-\ell}\} \,1 \{|z - y|_\infty \le 2^{-\ell '}\} \, 1_{U_1}(z)\,dy \,dz
\\ 
=&\; | B(0,2^{-\ell})|^{-1} \dis\int 1_{U_1}(z) \,\psi^x_{\ell,\ell '}(z) \, dz,
\end{split}
\end{equation}

\medskip\n
where $\psi^x_{\ell, \ell '}$ stands for the continuous compactly supported function
\begin{equation}\label{1.11}
\psi^x_{\ell, \ell '}(z)  = |B(0,2^{-\ell '})|^{-1} \dis\int 1\{|y - x|_\infty \le 2^{-\ell} \}\, 1\{|z - y|_\infty \le 2^{-\ell '}\} \, dy,
\end{equation}
so that
\begin{equation}\label{1.12}
\begin{split}
0 \le \psi^x_{\ell, \ell '}  \le 1, \; \mbox{and $\psi^x_{\ell, \ell '}$}  = &\;1\; \mbox{on $B(x,2^{-\ell} - 2^{-\ell '}),$}
\\
 \psi^x_{\ell, \ell '}  =&\; 0 \;\mbox{on} \; \IR^d \backslash B(x,2^{-\ell} + 2^{-\ell '}). 
\end{split}
\end{equation}
Coming back to (\ref{1.10}), we find that
\begin{equation}\label{1.13}
\begin{split}
\strokedint_{B(x,2^{-\ell})} \wh{\sigma}_{\ell '}(y) \, dy & \ge |B(x,2^{-\ell})|^{-1} \dis\int_{B(x,2^{-\ell} - 2^{- \ell '})} 1_{U_1}(z) \, dz
\\
& \ge \wh{\sigma}_\ell(x) - |B(x,2^{-\ell}) \backslash B(x,2^{-\ell} - 2^{-\ell '})| / |B(x,2^{-\ell})|
\\[1ex]
& = \wh{\sigma}_\ell(x) - \big(1 - (1-2^{\ell - \ell'})^d\big) = \wh{\sigma}_\ell(x) - \dis\int^1_{1 - 2^{\ell - \ell '}} du^{d-1} du
\\ 
& \ge \wh{\sigma}_\ell (x) - d \, 2^{\ell - \ell '}.
\end{split}
\end{equation}
In a similar fashion, we have
\begin{equation}\label{1.14}
\begin{split}
\strokedint_{B(x,2^{-\ell})} \wh{\sigma}_{\ell '}(y) \, dy & \le  |B(x,2^{-\ell})|^{-1} \dis\int_{B(x,2^{-\ell} + 2^{- \ell '})} 1_{U_1}(z) \, dz
\\
& \le \wh{\sigma}_\ell(x) + |B(x,2^{-\ell} + 2^{-\ell '}) \backslash B(x,2^{-\ell})| / |B(x,2^{-\ell})|
\\ 
& = \wh{\sigma}_\ell(x) + (1 + 2^{\ell - \ell'})^d - 1 = \wh{\sigma}_\ell(x) + \dis\int_1^{1 + 2^{\ell - \ell '}} du^{d-1} du 
\\ 
& \le \wh{\sigma}_\ell (x)+ d \, 2^{d- 1} 2^{\ell - \ell '}.
\end{split}
\end{equation}

\medskip\n
Collecting (\ref{1.13}) and (\ref{1.14}), we find (\ref{1.8}). This concludes the proof of \hbox{Lemma \ref{lem1.1}.}
\end{proof}

The next lemma will show that when $\ell ' > \ell$ and $\wh{\sigma}_{\ell '}$ has an average $\beta '$ in a box $B(x,2^{-\ell})$, then, either $\wh{\sigma}_{\ell '}$ takes with sizeable measure in the box $B(x,2^{-\ell})$ values bigger and values smaller than $\beta '$ by a certain amount, or that $\wh{\sigma}_{\ell '}$ takes with sizeable measure in the box $B(x, 2^{-\ell})$ values close to $\beta'$. This lemma will later be useful when showing in Proposition \ref{prop1.3} that Brownian motion starting at $x$ has a sizeable probability to reach points where $\wh{\sigma}_{\ell '}$ is close to $\beta '$ before exiting $B(x,2^{-\ell})$.

\begin{lem}\label{lem1.2}
For $x \in \IR^d$ and $\ell ' > \ell$ in $\IZ$, set
\begin{equation}\label{1.15}
\beta' =  \strokedint_{B(x, 2^{-\ell})} \wh{\sigma}_{\ell '} (y) \,dy.
\end{equation}

\medskip\n
Then, for $0 \le \delta \le \beta' \wedge (1-\beta')$ at least one of i) or ii) below holds:
\begin{equation}\label{1.16}
\left\{ \begin{array}{rl}
{\rm i)} & \mu_{x,\ell} (\{ \wh{\sigma}_{\ell '} > \beta' + \delta \}) \ge \mbox{\f $\dis\frac{\delta}{2}$}  \;\;\mbox{and} \;\;  \mu_{x,\ell} (\{\wh{\sigma}_{\ell'} < \beta' - \delta\}) \ge \mbox{\f $\dis\frac{\delta}{2}$} 
\\[1ex]
{\rm ii)} &  \mu_{x,\ell}  (\{ \beta' - \delta \le \wh{\sigma}_{\ell'} \le \beta' + \delta\}) \ge \mbox{\f $\dis\frac{1}{4}$}  - \mbox{\f $\dis\frac{\delta}{2}$} 
\end{array}\right.
\end{equation}
(we refer to (\ref{1.5}) for the definition of $\mu_{x,\ell})$.
\end{lem}

\begin{proof}
We introduce on some auxiliary probability space governed by the probability $P$, with corresponding expectation denoted by $E[\cdot ]$, a $[0,1]$-valued random variable $Y$ with same law as $\wh{\sigma}_{\ell'}$ under $\mu_{x,\ell}$, so that $\beta' = E[Y]$. We let $\ov{F}(u) = P[Y > u]$ for $u \in \IR$, and note that $\ov{F}(u) = 0$ for $u \ge 1$, and $\beta' = E[Y] = \int^{1}_0 \ov{F}(u) du$. Given $0 \le \delta \le \beta' \wedge (1- \beta')$, either
\begin{equation}\label{1.17}
\left\{ \begin{array}{ll}
{\rm a)}  \quad & \ov{F}(\beta' + \delta) \ge \mbox{\f $\dis\frac{\delta}{2}$}   
\\
\mbox{or}
\\
{\rm b)}  \quad &  \ov{F}  (\beta' + \delta) < \mbox{\f $\dis\frac{\delta}{2}$}.
\end{array}\right.
\end{equation}
In case b), we also find that
\begin{equation*}
\dis\int^{\beta' + \delta}_0 \ov{F}(u)\,du + \mbox{\f $\dis\frac{\delta}{2}$}  \ge \dis\int^{\beta' + \delta}_0 \ov{F}(u) \, du + \ov{F}(\beta' + \delta) \ge \dis\int^{\beta' + \delta}_0 \ov{F}(u)\,du + \dis\int^1_{\beta' + \delta} \ov{F}(u) \, du = \beta', 
\end{equation*}
so that
\begin{equation*}
\dis\int^{\beta' + \delta}_0 \ov{F}(u)\,du \ge \beta' -\mbox{\f $\dis\frac{\delta}{2}$} \,.
\end{equation*}

\medskip\n
At the same time, since $\ov{F} \le 1$, $\int_0^{\beta' - \delta} \ov{F}(u) \, du \le \beta' - \delta$, so that
\begin{equation*}
\dis\int^{\beta' + \delta}_{\beta' - \delta} \ov{F}(u)\,du \ge \beta' - \mbox{\f $\dis\frac{\delta}{2}$}  - (\beta' - \delta) = \mbox{\f $\dis\frac{\delta}{2}$} \,.
\end{equation*}

\medskip\n
Dividing both members by $2 \delta$, we find, since $\ov{F}$ is non-increasing, that $\ov{F}(\beta' - \delta) \ge \frac{1}{4}$. Since we are in case b), we also have $\ov{F} (\beta' + \delta) < \frac{\delta}{2}$. It now follows that in case (\ref{1.17}) b) we additionally have
\begin{equation}\label{1.18}
P[\beta' - \delta \le Y \le \beta' + \delta] \ge\mbox{\f $\dis\frac{1}{4}$} -\mbox{\f $\dis\frac{\delta}{2}$} \;.
\end{equation}

\medskip\n
We now consider $Z = 1 - Y$, so that $0 \le Z \le 1$, $E[Z] = 1 - \beta'$. We further note that $0 \le \delta \le \beta' \wedge (1 - \beta') = E[Z] \wedge (1 - E[Z])$. We can then introduce $\ov{F}_Z(u) = P[Z > u] = P[Y < 1 - u]$ for $u \in \IR$. If we now let $Z$ take the role of $Y$ in the alternative (\ref{1.17}), we see that either 
\begin{equation}\label{1.19}
\left\{ \begin{array}{rl}
{\rm c)} & \quad \ov{F}_Z (1- \beta' + \delta)\;(= P[Y < \beta' - \delta]) \ge  \mbox{\f $\dis\frac{\delta}{2}$}  \;\;\mbox{or}
\\[3ex]
{\rm d)} & \quad  \ov{F}_Z (1- \beta' + \delta)\;(= P[Y < \beta' - \delta])< \mbox{\f $\dis\frac{\delta}{2}$} \,.
\end{array}\right.
\end{equation}

\medskip\n
Additionally, we know that in case d)
\begin{equation*}
P[1 - \beta' - \delta \le Z \le 1 - \beta' + \delta] \;(= P[\beta' - \delta \le Y \le \beta' + \delta]) \ge \mbox{\f $\dis\frac{1}{4}$}  - \mbox{\f $\dis\frac{\delta}{2}$}  \;.
\end{equation*}

\medskip\n
Thus, when either b) in (\ref{1.17}) or d) in (\ref{1.19}) holds we obtain (\ref{1.16}) ii). So, when (\ref{1.16}) ii) does not hold, necessarily both a) in (\ref{1.17}) and c) in (\ref{1.19}) hold. This yields (\ref{1.16}) i) and concludes the proof of Lemma \ref{lem1.2}.
\end{proof}

\medskip
We will now see that when $\ell' > \ell$ and $\wh{\sigma}_{\ell'}$ has an average $\beta'$ over $B(x,2^{-\ell})$, which is not too close to $0$ or $1$, Brownian motion starting at $x$ has a non-degenerate probability of entering a region where $\wh{\sigma}_{\ell'}$ takes values close to $\beta'$ before exiting $B(x,2^{-\ell})$. For $r \ge 0$, we introduce the first time when $X$ moves at $|\cdot |_\infty$-distance $r$ from its starting point (an $(\cF_t)$-stopping time):
\begin{equation}\label{1.20}
\tau_r = \inf \{s \ge 0; \,|X_s - X_0|_\infty \ge r\}.
\end{equation}
We can now state

\begin{prop}\label{prop1.3}
For $x \in \IR^d$, $\ell' > \ell$ in $\IZ$, we set $\beta' = \strokedint_{B(x,2^{-\ell})} 
\, \wh{\sigma}_{\ell'}(y) \, dy$. Then, we have
\begin{equation}\label{1.21}
|\wh{\sigma}_\ell(x) - \beta' | \le c_0 \,2^{\ell - \ell'} \;\mbox{(with $c_0$ as in (\ref{1.8}))},
\end{equation}
and for $0 \le \delta \le \beta' \wedge (1 - \beta') \wedge \frac{1}{4}$,
\begin{equation}\label{1.22}
P_x[H_{\{\wh{\sigma}_{\ell'} \in [\beta' - \delta, \beta' + \delta]\}} <  \tau_{2^{- \ell}}] \ge c_1(\delta)
\end{equation}

\medskip\n
(where $H_{\{\wh{\sigma}_{\ell'} \in [\beta' - \delta, \beta' + \delta]\}}$ denotes the entrance time of $X$ in the closed set $\{\wh{\sigma}_{\ell'} \in [\beta' - \delta, \beta' + \delta]\}$).
\end{prop}

\begin{proof}
We first observe that (\ref{1.21}) is a restatement of (\ref{1.8}) and we only need to prove (\ref{1.22}). We use the alternative (\ref{1.16}) from Lemma \ref{lem1.2}. If (\ref{1.16}) i) holds, we use translation invariance and scaling to choose $\delta' \in (0,1)$ solely depending on $\delta$, such that $\mu_{x,\ell}(\{y: |y - x|_\infty > 2^{-\ell} (1 - \delta')\}) \le \frac{\delta}{4}$, so that
\begin{equation}\label{1.23}
\left\{ \begin{array}{l}
 \mu_{x,\ell}\big(\{\wh{\sigma}_{\ell'} >\beta' + \delta\} \cap B(x,2^{-\ell} (1 - \delta')\big)\big) \ge \mbox{\f $\dis\frac{\delta}{4}$} \;\;\mbox{and}
\\[3ex]
 \mu_{x,\ell}\big(\{\wh{\sigma}_{\ell'} < \beta' - \delta\} \cap B(x,2^{-\ell} (1 - \delta')\big)\big) \ge \mbox{\f $\dis\frac{\delta}{4}$}\,.
\end{array}\right.
\end{equation}

\medskip\n
Letting $U$ stand for the interior of $B(x,2^{-\ell})$ (i.e.~$U = \{y: |y - x|_\infty < 2^{-\ell}\}$), we find by translation invariance, scaling, and standard estimates for Brownian motion killed outside $\mathring{B}(0,1)$ that
\begin{equation}\label{1.24}
\left\{ \begin{array}{l}
P_x\big[H_{\{\wh{\sigma}_{\ell'} \ge \beta' + \delta\} \cap B(x,2^{-\ell} (1 - \delta'))} < T_U\big] \ge c(\delta) \;\;\mbox{and}
\\[3ex]
P_y\big[H_{\{\wh{\sigma}_{\ell'} \le \beta' - \delta\} \cap B(x,2^{-\ell} (1 - \delta'))} < T_U\big] \ge c(\delta) \;\mbox{for all}
\\[1ex]
\mbox{$y \in B(x,2^{-\ell} (1 - \delta'))$} \,.
\end{array}\right.
\end{equation}

\medskip\n
It then follows from the strong Markov property that Brownian motion starting at $x$ enters 
$\{\wh{\sigma}_{\ell'} \ge \beta' + \delta\}$ and then $\{\wh{\sigma}_{\ell'} \le \beta' - \delta\}$ before exiting $U$ with probability at least $c'(\delta) (= c(\delta)^2$). By the continuity of $\wh{\sigma}_{\ell'}$, any such trajectory of Brownian motion enters $\{\wh{\sigma}_{\ell'} = \beta'\}$ before exiting $U$.  So, when (\ref{1.16}) i) holds, we find that
\begin{equation}\label{1.25}
P_x[H_{\{\wh{\sigma}_{\ell'} = \beta'\}} < \tau_{2^{- \ell}}] \ge c'(\delta).
\end{equation}
On the other hand, when (\ref{1.16}) ii) holds, then
\begin{equation*}
\mu_{x,\ell} (\{\wh{\sigma}_{\ell'} \in [\beta' - \delta, \beta' + \delta]\}) \ge \mbox{\f $\dis\frac{1}{4}$} -  \mbox{\f $\dis\frac{\delta}{2}$} \ge \mbox{\f $\dis\frac{1}{8}$}\;,
\end{equation*}

\medskip\n
and it follows from standard estimates on killed Brownian motion (as above) that
\begin{equation}\label{1.26}
P_x[H_{\{\wh{\sigma}_{\ell'} \in [\beta' + \delta, \beta' + \delta]\}}  <  \tau_{2^{- \ell}}] \ge c.
\end{equation}

\medskip\n
Collecting (\ref{1.25}) and (\ref{1.26}) we find (\ref{1.22}). This proves Proposition \ref{prop1.3}.\end{proof}

\medskip
We will now consider decreasing scales $2^{-\ell_0} > 2^{- \ell_1} > \dots > 2^{-\ell_J}$, which are well-separated, and see that when starting at $x$ in $\IR^d$ such that $\wh{\sigma}_{\ell_0}(x) = \frac{1}{2}$, there is a non-degenerate probability for Brownian motion to reach before moving at sup-distance $\frac{3}{2} \,2^{-\ell_0}$ the set of points where all local densities $\wt{\sigma}_{\ell_j}(\cdot)$, $0 \le j \le J$, lie within the fixed interval $[\wt{\alpha}, 1 - \wt{\alpha}]$, with $\wt{\alpha}$ the constant from (\ref{1.35}) below. More precisely, we consider $J \ge 1$, $c_0  ( = d 2^{d-1})$ as in (\ref{1.8}), and we define
\begin{equation}\label{1.27}
\mbox{$L(J) =$ the smallest integer $L \ge 5$ such that $c_0 \, 2^{-L} \le (200 J)^{-1}$}.
\end{equation}

\medskip\n
We also consider a sequence of decreasing spatial scales $2^{-\ell_0} > 2^{- \ell_1} > \dots > 2^{-\ell_J}$, which are well-separated in the sense that
\begin{equation}\label{1.28}
\mbox{$\ell_{j+1} \ge \ell_j + L(J)$, for $0 \le j < J$}.
\end{equation}
Further, we introduce the increasing sequence of compact sub-intervals of $(0,1)$:
\begin{equation}\label{1.29}
I_j = \Big[\mbox{\f $\dis\frac{1}{2}$} - \mbox{\f $\dis\frac{j}{100J}$}, \; \mbox{\f $\dis\frac{1}{2}$} + \mbox{\f $\dis\frac{j}{100 J}$}\Big], \; 0 \le j \le J,
\end{equation}

\medskip\n
as well as the non-decreasing sequence of stopping times
\begin{equation}\label{1.30}
\gamma_0 = H_{\{\wh{\sigma}_{\ell_0} \in I_0\}} \;\; \mbox{and for} \;\; 0 \le j < J, \;\gamma_{j+1} = \gamma_j + H_{\{\wh{\sigma}_{\ell_{j+1}} \in I_{j+1}\}} \circ \theta_{\gamma_j}.
\end{equation}

\medskip\n
The next proposition will enter the proof of the main Theorem \ref{theo2.1} in Section 2. We recall (\ref{1.3}) and (\ref{1.20}) for notation.

\begin{prop}\label{prop1.4}
Assume that $J \ge 1$, and that $\ell_j, 0 \le j \le J$, satisfy (\ref{1.28}). Denote by $\cC$ the event 
\begin{equation}\label{1.31}
\cC = \{ \gamma_0  = 0\} \cap \bigcap\limits_{0 \le j < J} \{\gamma_{j+1} < \gamma_j + \tau_{2^{- \ell_j}} \circ \theta_{\gamma_j}\}.
\end{equation}

\medskip\n
Then, for $x$ in $\IR^d$ such that $\wh{\sigma}_{\ell_0} (x) = \frac{1}{2}$, one has
\begin{equation}\label{1.32}
P_x[\cC] \ge c_2(J).
\end{equation}
Moreover, on the event $\cC$, one has
\begin{align}
&\sup \{|X_s - X_{\gamma_j}|_\infty; \; \gamma_j \le s \le \gamma_J\} \le \dis\frac{3}{2} \;2^{-\ell_j}, \; \mbox{for all $0 \le j < J$, and}\label{1.33}
\\[1ex]
&\wt{\sigma}_{\ell_j} (X_{\gamma_J}) \in [\wt{\alpha}, 1 - \wt{\alpha}], \; \mbox{for all $0 \le j \le J$, where} \label{1.34}
\\[1ex]
& \wt{\alpha} = \dis\frac{1}{3} \;4^{-d}. \label{1.35}
\end{align}
\end{prop}

\begin{proof}
We first prove (\ref{1.32}). We will use induction and Proposition \ref{prop1.3}. We choose (in the notation of Proposition \ref{prop1.3})
\begin{equation}\label{1.36}
\delta = (200 J)^{-1} \big( \le \mbox{\f $\dis\frac{1}{4}$}\big).
\end{equation}

\n
We want to show by induction on $0 \le j \le J$ (with $c_1(\delta)$ as in (\ref{1.22})) that
\begin{equation}\label{1.37}
\begin{array}{l}
P_x[\cC_j] \ge c_1(\delta)^j, \;\mbox{with} 
\\
\cC_j = \{\gamma_0 = 0, \gamma_1 < \gamma_0 + \tau_{2^{- \ell_0}} \circ \theta_{\gamma_0}, \dots, \gamma_j < \gamma_{j-1} + \tau_{2^{- \ell_{j-1}}} \circ \theta_{\gamma_{j-1}}\}.
\end{array}
\end{equation}

\medskip\n
When $j = 0$, since $\wh{\sigma}_{\ell_0}(x) = \frac{1}{2}$, we find that $P_x[\gamma_0 = 0] = 1$ and (\ref{1.37}) is true. Assume now that for some $0 \le j < J$, (\ref{1.37}) is satisfied. We now define $\beta'_{j+1} = \strokedint_{B(X_{\gamma_j}, 2^{-\ell_j})}     \wh{\sigma}_{\ell_{j+1}}(y)\, dy$ and note that on $\cC_j$ due to (\ref{1.30}), one has $\wh{\sigma}_{\ell_j}(X_{\gamma_j}) \in I_j$. Moreover, by (\ref{1.21}) and (\ref{1.28}), we have
\begin{equation}\label{1.38}
\begin{array}{l}
|\wh{\sigma}_{\ell_j} (X_{\gamma_j}) - \beta'_{j+1}| \le c_0 \,2^{-L(J)} \stackrel{(\ref{1.27})}{\le} (200 J)^{-1} (= \delta), \;\mbox{so that} 
\\[1ex]
\beta'_{j+1}  \in \Big[ \mbox{\f $\dis\frac{1}{2}$} -  \mbox{\f $\dis\frac{j}{100 J}$} - \mbox{\f $\dis\frac{1}{200 J}$} , \; \mbox{\f $\dis\frac{1}{2}$} + \mbox{\f $\dis\frac{j}{100 J}$} + \mbox{\f $\dis\frac{1}{200 J}$}\Big] \; \mbox{and}
\\[2ex]
\beta'_{j+1} +  \Big[ - \mbox{\f $\dis\frac{1}{200 J}$} , \mbox{\f $\dis\frac{1}{200 J}$}\Big] \subseteq I_{j+1}.
\end{array}
\end{equation}

\bigskip\n
We can now apply the strong Markov property at time $\gamma_j$ and find that
\begin{equation}\label{1.39}
\begin{split}
P_x [\cC_{j+1}] & = P_x[\cC_j \cap \{ \gamma_{j+1} < \gamma_j + \tau_{2^{- \ell_j}} \circ \theta_{\gamma_j}\}]
\\
& = E_x [\cC_j, P_{X_{\gamma_j}}[H_{\{\wh{\sigma}_{\ell_{j+1}} \in I_{j+1}\}} < \tau_{2 ^{- \ell_j}}\}] \stackrel{(\ref{1.22}), {\rm induction}}{\ge} c_1(\delta)^{j+1}.
\end{split}
\end{equation}

\medskip\n
This completes the proof by induction of (\ref{1.37}) and hence of (\ref{1.32}) with $c_2(J) = c_1(\delta)^J$. As for (\ref{1.33}), we note that on $\cC$, for any $0 \le j < J$, one has
\begin{equation}\label{1.40}
\begin{split}
\sup\{|X_s - X_{\gamma_j}|_\infty; \; \gamma_j \le s \le \gamma_J\} &\le 2^{-\ell_j} + 2^{-\ell_{j+1}} + \dots + 2^{-\ell_{J-1}}
\\
& \!\!\!\!\!\stackrel{(\ref{1.28})}{\le} 2^{-\ell_j} \dsl_{m \ge 0} 2^{-m L(J)} < \mbox{\f $\dis\frac{3}{2}$} \;2^{-\ell_j},
\end{split}
\end{equation}

\medskip\n
since $L(J) \ge 5$, so that $1 - 2^{-L(J)} > \frac{2}{3}$. This proves (\ref{1.33}).

\bigskip
We then turn to the proof of (\ref{1.34}). We note that by (\ref{1.33}), on the event $\cC$, one has for any $0 \le j \le J$, $B(X_{\gamma_j}, 2^{-\ell_j}) \subseteq B(X_{\gamma_J}, 4 \cdot 2^{- \ell_j})$, and $\wh{\sigma}_{\ell_j} (X_{\gamma_j}) \in I_j \subseteq [\frac{1}{3}, \frac{2}{3}]$, so that $\wt{\sigma}_{\ell_j}(X_{\gamma_J}) \in [\frac{1}{3} \;4^{-d}, 1 - \frac{1}{3} \;4^{-d}] = [\wt{\alpha}, 1 - \wt{\alpha}]$, and (\ref{1.34}) follows. This concludes the proof of Proposition \ref{prop1.4}.
\end{proof}

\section{Resonances}

In this section we introduce certain resonance sets where at least $J$ among a larger collection of local densities $\wt{\sigma}_\ell$ of $U_1$, corresponding to well-separated spatial scales, simultaneously take non-degenerate values in $[\wt{\alpha}, 1 - \wt{\alpha}]$,  cf.~(\ref{1.35}). Our main object is Theorem \ref{theo2.1} below. It shows that for Brownian motion starting at a location where all but finitely many of the local densities $\wh{\sigma}_\ell$ are at most $\frac{1}{2}$, it is hard to avoid such a resonance set. An important induction step for the proof of Theorem \ref{theo2.1} is contained in Lemma \ref{lem2.2}. We will then use Theorem \ref{theo2.1} as a main tool in the next section when considering porous interfaces which are markedly felt in the vicinity of $S = \partial U_0 = \partial U_1$. The approach we use remains pertinent in the two-dimensional case, but the formulation of a relevant version of Theorem \ref{theo2.1} when $d=2$ involves some modification of the set-up (with some killing of Brownian motion). For conciseness and because the application in Section 4 only pertains to the case $d \ge  3$, we assume from now on that
\begin{equation}\label{2.1}
d \ge 3.
\end{equation}

\n
As in the previous section, we consider $U_0$ a non-empty, bounded, Borel subset of $\IR^d$ and the associated local density functions $\wh{\sigma}_\ell, \wt{\sigma}_\ell$, cf.~(\ref{1.3}). We now want to introduce the resonance set. To this end, we consider $\ell_* \ge 0$ (so that $2^{-\ell_*}$ will bound from above the scales under consideration, and in some sense bound from below the distance of the starting point of Brownian motion to $U_1$), $J \ge 1$ will control the strength of the resonance, $L \ge L(J)$ (with $L(J)$ as in (\ref{1.27})), will govern the separation of scales, and $I \ge 1$, will control the number of scales inspected. We then define (with $\IN = \{0,1,2,\dots\}$)
\begin{align}
\ell_0 &= \inf\{ \ell \in (J+1) \,L \IN; \; \ell \ge \ell_*\}, \label{2.2}
\\[1ex]
\cA_* & = \{ \ell \in  L \IN; \; \ell_0 \le \ell < \ell_0 + I(J+1) \,L\} \;\;\mbox{(so that $|\cA_*| = I(J+1)$)},\label{2.3}
\\[1ex]
\cA  & = \{ \ell \in  (J+1)\, L \IN; \; \ell_0 \le \ell < \ell_0 + I(J+1) \,L\} \;\;\mbox{(so that $|\cA| = I$)}.\label{2.4}
\end{align}

\medskip\n
The resonance set is then defined (with $\wt{\alpha}$ as in (\ref{1.35})) as:
\begin{equation}\label{2.5}
{\rm Res} = \{x \in \IR^d; \;\dsl_{\ell \in \cA_*} 1\{\wt{\sigma}_\ell (x) \in [\wt{\alpha}, 1 - \wt{\alpha}]\} \ge J\},
\end{equation}

\n
and we sometimes write ${\rm Res} (U_0, \ell_*,J, L, I)$, if we want to recall the parameters entering its definition. Note that the functions $\wt{\sigma}_\ell$, $\ell \ge 0$, are continuous, cf.~Lemma \ref{lem1.1}, and $\cA_*$ finite, so that the resonance set Res is a (possibly empty) compact subset of $\IR^d$.

\medskip
To describe the collection of subsets $U_0$ under consideration in the bounds we wish to derive in Theorem \ref{theo2.1}, we consider some non-empty compact subset $A$ in $\IR^d$ and introduce for $\ell_* \ge 0$ as above, 
\begin{equation}\label{2.6}
\begin{split}
\cU_{\ell_*,A} = &\;\mbox{the collection of bounded Borel subsets $U_0$ of $\IR^d$ such that}
\\
&\; \mbox{$\wh{\sigma}_\ell(x) \le \frac{1}{2}$, for all $x \in A$ and $\ell \ge \ell_*$},
\end{split}
\end{equation}
as well as

\vspace{-4ex}
\begin{equation}\label{2.7}
\cU_{\ell_*} = \cU_{\ell_*,A = \{0\}}.
\end{equation}

\medskip\n
We are now ready to state the main result of this section. It provides stretched exponential bounds in $I$ on the probability that Brownian motion starting in $A$ avoids the resonance set if $U_0 \in \cU_{\ell_*,A}$, $L \ge L(J)$ and $I$ is large. Incidentally, let us point out that the resonance set need not ``block'' $A$ in a topological sense: $A$ may well lie in the unbounded component of the complement of the resonance set, see Remark \ref{rem2.4}.

\begin{thm}\label{theo2.1}
For $J, I \ge 1$, $A$ non-empty compact subset of $\IR^d$, define (with {\rm Res} as in (\ref{2.5}))
\begin{equation}\label{2.8}
\begin{split}
\Phi_{J,I,A} & = \sup\limits_{\ell_* \ge 0} \; \,\sup\limits_{U_0 \in \cU_{\ell_*,A}}\,\, \sup\limits_{L \ge L(J)} \;\;\sup\limits_{x \in A} \;P_x[H_{{\rm Res}} = \infty], \;\mbox{and}
\\
\Phi_{J,I} & = \Phi_{J,I,A = \{0\}}.
\end{split}
\end{equation}

\medskip\n
Then, the case $A = \{0\}$ is maximal in the sense that
\begin{equation}\label{2.9}
\Phi_{J,I,A} \le \Phi_{J,I},
\end{equation}

\medskip\n
and as $I \r \infty$, one has the stretched exponential bound:
\begin{equation}\label{2.10}
\limsup\limits_I \;I^{-1/2^{J-1}} \log \Phi_{J,I} \le \log \big(1 - c_2(J)\big) (< 0).
\end{equation}
\end{thm}

\begin{proof}
We first prove (\ref{2.9}). To this end, we simply note that $U_0 \in \cU_{\ell_*,A}$ and $x \in A$ implies that $U_0 - x \in \cU_{\ell_*}$ and ${\rm Res} (U_0 - x, \ell_*, J, L, I) = {\rm Res} (U_0, \ell_*, J, L, I) - x$. The claim (\ref{2.9}) simply follows by translation invariance of Brownian motion.

\medskip
We now turn to the proof of (\ref{2.10}). We consider $\ell_* \ge 0$, $J \ge 1$, $L \ge L(J)$, $I \ge 1$, as well as some $U_0 \in \cU_{\ell_*}$. We then introduce the notion of an $I$-family, which is constituted of stopping times $S_i$, $0 \le i \le I$, of a random finite subset $\cL$ of $(J+1) \, L \IN$, and of random integer valued functions $\wh{\ell}_i, 1 \le i \le I$, such that

\begin{equation}\label{2.11}
\left\{ \begin{array}{rl}
{\rm i)} & \mbox{$0 \le S_0 \le S_1 \le \dots \le S_I$ are $P_0$-a.s. finite $(\cF_t)$-stopping times,}
\\[1.5ex]
{\rm ii)} & \mbox{$\cL$ is an $\cF_{S_0}$-measurable finite subset of $(J+1) \,L\IN$, with $| \cL | \ge I$},
\\[1.5ex]
{\rm iii)} &\mbox{$\wh{\ell}_i, 1 \le i \le I$, are respectively $\cF_{S_i}$-measurable, distinct and}
\\
&\mbox{$\cL$-valued},
\\[1.5ex]
{\rm iv)} & \mbox{$P_0$-a.s., $\wh{\sigma}_{\wh{\ell}_i}(X_{S_i}) = \mbox{\f $\dis\frac{1}{2}$}$, for $1 \le i \le I$}.
\end{array}\right.
\end{equation}

\medskip\n
We recall that $\wh{\sigma}_\ell (0) \le \frac{1}{2}$ for $\ell \ge \ell_*$, in particular for all $\ell \in \cA$ in (\ref{2.4}). As a result, $P_0$-a.s., for all $\ell \in \cA$, the continuous non-negative functions $s \ge 0 \rightmapsto \wh{\sigma}_\ell(X_s)$ start at a value smaller or equal to $\frac{1}{2}$ and reach at some point the value $\frac{1}{2}$ (they eventually become equal to $1$ for large $s$). There is however in general no prescribed order in which these ``crossings'' happen, see Remark \ref{rem2.4}. An example of an $I$-family to keep in mind thus corresponds to the choice
\begin{equation}\label{2.12}
\left\{ \begin{split}
\cL & = \cA \;\;\mbox{(see (\ref{2.4}))},
\\[1ex]
S_0 & = 0,
\\[1ex]
S_1 & = \inf\big\{s \ge 0; \wh{\sigma}_\ell (X_s) = \mbox{\f $\dis\frac{1}{2}$} , \;\mbox{for some $\ell \in \cL\big\}$},
\\[2ex]
\wh{\ell}_1 & = \max\big\{\ell \in \cL; \,\wh{\sigma}_\ell (X_{S_1}) = \mbox{\f $\dis\frac{1}{2}$}\big\}, \;\mbox{if $S_1 < \infty$, and}
\\[-0.5ex]
&\quad \; \mbox{$\wh{\ell}_1 = \max \{ \ell \in \cL\}$, if $S_1 = \infty$},
\\[2ex]
S_2 & = \inf\big\{s \ge S_1; \,\wh{\sigma}_\ell (X_s) = \mbox{\f $\dis\frac{1}{2}$}, \;\mbox{for some $\ell \in \cL \backslash \{\wh{\ell}_1\}\big\}$},
\\[2ex]
\wh{\ell}_2 & = \mbox{$\max\big\{\ell \in \cL \backslash \{\wh{\ell}_1\}, \; \wh{\sigma}_\ell(X_{S_2}) =\mbox{\f $\dis\frac{1}{2}$}\}$, if $S_2 < \infty$, and}
\\[-0.5ex]
&\quad \; \mbox{$\wh{\ell}_2 = \max \{ \ell \in \cL \backslash \{\wh{\ell}_1\}\big\}$,   if $S_2 = \infty$,}
\\[-1ex]
\vdots
\\
S_I & = \mbox{$\inf\big\{ s \ge S_{I - 1}; \;\wh{\sigma}_\ell(X_s) = \mbox{\f $\dis\frac{1}{2}$}$, for $\ell \in \cL \backslash \{ \wh{\ell}_1,\dots, \wh{\ell}_{I - 1}\}\big\}$},
\\[2ex]
\wh{\ell}_I & = \mbox{the unique element of $\cL \backslash \{ \wh{\ell}_1, \dots, \wh{\ell}_{I-1}\big\}$ (when $S_I = \infty$ note that}
\\[-0.5ex]
&\quad \; \mbox{$\max\big\{\ell \in \cL \backslash \{ \wh{\ell}_1,\dots, \wh{\ell}_{I-1}\}\big\}$ coincides with the unique element of}
\\[-0.5ex]
&\quad \; \cL \backslash \{ \wh{\ell}_1,\dots, \wh{\ell}_{I-1}\}).
\end{split}\right.
\end{equation}

\bigskip\n
The formulas for $\wh{\ell}_i$ when $S_i = \infty$ are merely here for completeness and pertain to $P_0$-negligible events.

\medskip
Given an $I$-family as in (\ref{2.11}), we also define the stopping times
\begin{equation}\label{2.13}
\begin{array}{l}
T_i = \inf\{s \ge S_i;  \;|X_s - X_{S_i}|_\infty \ge 2 \cdot 2^{-\wh{\ell}_i}\} 
\\
\mbox{(understood as $+ \infty$, if $S_i = \infty$), for $1 \le i \le I$},
\end{array}
\end{equation}

\n
as well as the $\cF_{S_0}$-measurable subsets of ``intermediate labels'' and of ``labels''
\begin{equation}\label{2.14}
\cL_{\rm int} = \{\ell + j L; \; \ell \in \cL, \; 1 \le j \le J\} \; \mbox{and $\cL_* = \cL \cup \cL_{\rm int}$}.
\end{equation}
Note that $\cL \subseteq (J+1) \,L \IN$, so that
\begin{equation}\label{2.15}
\cL \cap \cL_{\rm int} = \emptyset,
\end{equation}
and that in the notation of (\ref{2.3}), (\ref{2.4})
\begin{equation}\label{2.16}
\mbox{when $\cL = \cA$, then $\cL_* = \cA_*$}.
\end{equation}
Further we introduce for $1 \le k \le J$ the $(\cL_*,k)$-resonance set
\begin{equation}\label{2.17}
{\rm Res}_{\cL_*,k} = \Big\{x \in \IR^d; \, \dsl_{\ell \in \cL_*} 1\{\wt{\sigma}_\ell (x) \in [\wt{\alpha}, 1 - \wt{\alpha}]\} \ge k\Big\}
\end{equation}

\n
(so when $\cL = \cA$ as in (\ref{2.4}), and $k = J$, we recover the resonance set Res$(U_0,\ell_*,J,L,I)$ from (\ref{2.5})).

\medskip
We now introduce an important quantity on which we will derive upper bounds by induction on $k$ in the crucial Lemma \ref{lem2.2} below. Namely, for $1 \le k \le J$ and $I \ge 1$, we set
\begin{equation}\label{2.18}
\Gamma^{(J)}_k (I) = \sup P_0 [\inf\{s \ge S_0; X_s \in {\rm Res}_{\cL_*,k}\} > \max \limits_{1 \le i \le I} T_i],
\end{equation}

\n
where the supremum is taken over all $I$-families (a non-empty collection by (\ref{2.12})), and we set by convention
\begin{equation}\label{2.19}
\mbox{$\Gamma_k^{(J)}(I) = 1$, when $1 \le k \le J$ and $I \le 0$}.
\end{equation}

\n
We will also drop the superscript $(J)$ from the notation when this causes no confusion. 

At this point, the Reader may first read the statement of the next lemma, skip its proof, and directly proceed above (\ref{2.34}) to see how the proof of Theorem \ref{theo2.1} is completed. We now have (where we fix $\ell_* \ge 0$, $J \ge 1$, $L \ge L(J)$, $U_0 \in \cU_{\ell_*}$):

\begin{lem}\label{lem2.2}
\begin{equation}\label{2.20}
\mbox{$\Gamma_1^{(J)}(I) = 0$, for all $I \ge 1$},
\end{equation}

\n
and for $1 \le k  < J$, $I \ge 1$, setting $\Delta = [\sqrt{I}] \;\;( \ge 1)$, 
\begin{equation}\label{2.21}
\Gamma_{k+1}^{(J)}(I) \le \big(1 - c_2 (J)\big)^{\sqrt{I} - 1} + I^{1 + \frac{k-1}{2}} \Gamma_k^{(J)} (\Delta - k+1).
\end{equation}
\end{lem}

\begin{proof}
We first prove (\ref{2.20}). We consider $I \ge 1$ and some $I$-family, cf.~(\ref{2.11}). Then, $P_0$-a.s., $\wh{\sigma}_{\wh{\ell}_1}(X_{S_1}) = \frac{1}{2}$, so that $U_1$ and $U_0$ have relative volume $\frac{1}{2}$ in $B(X_{S_1},2^{-\wh{\ell}_1}$). It thus follows that $P_0$-a.s., $\wt{\sigma}_{\wh{\ell}_1}(X_{S_1})$ lies in $[\frac{1}{2} \cdot 4^{-d}, 1 - \frac{1}{2} \cdot 4^{-d}] \subseteq [\wt{\alpha}, 1 - \wt{\alpha}]$, so that $X_{S_1} \in {\rm Res}_{\cL_*,1}$. Hence, the probability in the right-hand side of (\ref{2.18}) equals $0$, when $k = 1$ and (\ref{2.20}) follows. 

\medskip
We now turn to the proof of (\ref{2.21}). We introduce the notation
\begin{equation}\label{2.22}
m_\Delta = \Big[\mbox{\f $\dis\frac{I - 1}{\Delta}$}\Big], \;\mbox{so that $1 + m_\Delta \, \Delta \le I < 1 + (m_\Delta + 1) \, \Delta$}.
\end{equation}

\medskip\n
We first assume that $I \ge 2$, (the case $I = 1$ will be straightforward to handle) and note that, as we explain below,
\begin{equation}\label{2.23}
I - \Delta \ge 1
\end{equation}

\n
(indeed for $I = 2$, $\Delta = 1$, for $I = 3$, $\Delta = 1$, and for $I \ge 4$, $I - \Delta \ge I - \sqrt{I} = \sqrt{I} (\sqrt{I} - 1) \ge \sqrt{I} \ge 2)$.

\medskip
We thus consider $1 \le k < J$ as well as $I \ge 2$, and want to prove (\ref{2.21}). We consider an $I$-family and write
\begin{equation}\label{2.24}
\begin{array}{l}
P_0 [\inf\{s \ge S_0; X_s \in {\rm Res}_{\cL_*, k + 1}\} > \max\limits_{1 \le i \le I} \,T_i] \le a_1 + a_2, \;\mbox{where}
\\[3ex]
\mbox{$a_1 = P_0[T_i < S_{i+\Delta}$ for all $1 \le i \le I - \Delta$, and}
\\[1ex]
\mbox{$\inf\{s \ge S_0$; $X_s \in {\rm Res}_{\cL_*,k+1}\} > \max\limits_{1 \le i \le I} \,T_i]$},
\\[4ex]
\mbox{$a_2 = P_0[T_i \ge S_{i+\Delta}$ for some $i \in \{1,\dots,I-\Delta\}$, and}
\\[1ex]
\mbox{$\inf\{s \ge S_0$; $X_s \in {\rm Res}_{\cL_*,k+1}\} > \max\limits_{1 \le i \le I} \,T_i]$}.
\end{array}
\end{equation}

\medskip\n
We first bound $a_1$. As a shorthand we write $i_\Delta = 1 + m_\Delta \,\Delta$, and find that 
\begin{equation*}
a_1 \le P_0[S_1 < T_1 < S_{1 + \Delta}  < T_{1 + \Delta} < \dots < S_{i_\Delta} < T_{i_\Delta} < \inf\{s \ge S_0; X_s \in {\rm Res}_{\cL_*,k+1}\}].
\end{equation*}

\n
We can use the strong Markov property at time $S_{i_\Delta}$ and find that the last expression is smaller or equal to
\begin{equation}\label{2.25}
\begin{split}
E_0\big[&S_1 < T_1 < \dots < S_{i_\Delta} < \inf\{s \ge S_0; X_s \in {\rm Res}_{\cL_*, k+1}\},
\\ 
&\wt{P}_{X_{S_{i_\Delta}}}[\inf\{s \ge 0; \,|\wt{X}_s - \wt{X}_0|_\infty \ge 2 \cdot 2^{-\wh{\ell}_{i_\Delta}}\} <
\\
& \inf\{s \ge 0; \, \wt{X}_s \in {\rm Res}_{\cL_*,k+1}\}]\big],
\end{split}
\end{equation}

\n
where $\wt{X}_\point$ denotes the canonical Brownian motion under the Wiener measure $\wt{P}_{X_{S_{i_\Delta}}}$ starting from $X_{S_{i_\Delta}}$, and the respectively $\cF_{S_0}$-measurable map $\cL_*$ (with an at most countable set of possible values), and the $\cF_{S_{i_\Delta}}$-measurable map $\wh{\ell}_{i_\Delta}$ are not integrated under the measure $\wt{P}_{X_{S_{i_\Delta}}}$.

\medskip
If we now choose $x = X_{S_{i_\Delta}}$ in (\ref{1.32}) (and $\ell_0 = \wh{\ell}_{i_\Delta} \in \cL$, by (\ref{2.11}) iii)), then Proposition \ref{prop1.4} and the fact that $P_0$-a.s., $\wh{\sigma}_{\wh{\ell}_{i_\Delta}}(X_{S_{i_\Delta}}) = \frac{1}{2}$ (by (\ref{2.11}) iv)), imply that $P_0$-a.s.,
\begin{equation}\label{2.26}
\begin{array}{l}
\wt{P}_{X_{s_{i_\Delta}}} [\inf\{s \ge 0; \,|\wt{X}_s - \wt{X}_0|_\infty \ge 2 \cdot 2^{-\wh{\ell}_{i_\Delta}}\} >
\\
\inf\{ s \ge 0; \, \wt{X}_s \in {\rm Res}_{\cL_*,k+1}\}] \ge c_2(J).
\end{array}
\end{equation}
Hence, the expression in (\ref{2.25}) is at most
\begin{equation}\label{2.27}
P_0 [S_1 < T_1 < \dots < T_{i_\Delta - \Delta} < \inf\{s \ge 0,\, X_s \in {\rm Res}_{\cL_*,k+1}\}] (1 - c_2(J)).
\end{equation}
By induction, we then find that
\begin{equation}\label{2.28}
\begin{split}
a_1  \le &\; \big(1-c_2(J)\big)^{m_\Delta} \,P_0[S_1 < T_1 < \inf\{s \ge S_0; \, X_s \in {\rm Res}_{\cL_*,k+1}\}]   
\\
\le &\; \big(1 - c_2(J)\big)^{m_\Delta + 1} \le  \big(1-c_2(J)\big)^{\sqrt{I} - 1}, \;\mbox{since} \
\\
&\; m_\Delta + 1 \stackrel{(\ref{2.22})}{>} (I - 1) \,\Delta^{-1} \ge (I-1) / \sqrt{I} \ge \sqrt{I} - 1.
\end{split}
\end{equation}

\medskip\n
We now bound $a_2$ (cf.~(\ref{2.24})). To this end, we first note that
\begin{equation}\label{2.29}
\begin{split}
a_2  = & \;P_0[T_i \ge S_{i+\Delta} \;\;\mbox{for some} \;\; i \in \{1,\dots,I-\Delta\}, 
\\
&\;\inf\{s\ge S_0; \,X_s \in {\rm Res}_{\cL_*,k+1}\} > \max\limits_{1 \le i \le I} \,T_i]
\\[1ex]
\le &\; I \max\limits_{1 \le i_0 \le I-\Delta} P_0[S_{i_0 + \Delta} \le T_{i_0} \;\;\mbox{and} 
\\
&\;\inf\{s \ge S_0; X_s \in {\rm Res}_{\cL_*,k+1}\} > \max\limits_{1 \le i \le I} \,T_i].
\end{split}
\end{equation}

\n
Denote by $F$ the event inside the probability in the last member of (\ref{2.29}). As we now explain, $P_0$-a.s. on $F$, there are at most $(k-1)$ integer values of $i \in (i_0,i_0 + \Delta]$ such that $T_{i_0} \le T_i$. Indeed, otherwise we can find $k$ values of $i$ in $(i_0,i_0 + \Delta]$ such that with $P_0$-positive measure on $F$, $|X_{T_{i_0}} - X_{S_i}|_\infty \le 2 \cdot 2^{-\wh{\ell}_i}$. Including $i_0$ to this list yields $k+1$ values in $[i_0,i_0 + \Delta]$ such that with $P_0$-positive measure on $F$, $\wh{\sigma}_{\wh{\ell}_i} (X_{S_i}) = \frac{1}{2}$ and $|X_{T_{i_0}} - X_{S_i}|_\infty \le 2 \cdot 2^{-\wh{\ell}_i}$, and hence $\wt{\sigma}_{\wh{\ell}_i}(X_{T_{i_0}}) \in [\frac{1}{2} \, \cdot \,4^{-d}, 1 - \frac{1}{2} \cdot 4^{-d}] \subseteq [\wt{\alpha}, 1 - \wt{\alpha}]$ for these $k+1$ values. This would force that with $P_0$-positive measure on $F$, $X_{T_{i_0}} \in {\rm Res}_{\cL_*, k+1}$, which is impossible by the very definition of $F$.

\bigskip
Assume for the time being that $\Delta\ge k$, and denote by $(i_0,i_0 + \Delta]'$ an arbitrary subset of $(i_0,i_0 + \Delta]$ with $\Delta - (k-1)$ elements. The number of possible choices of such a subset is at most $\Delta^{k-1}$. It now follows from the remark in the paragraph above and from (\ref{2.29}) that
\begin{equation}\label{2.30}
\begin{split}
a_2 \le I \,\Delta^{k-1} \max\limits_{1 \le i_0 \le I-\Delta} \;\max\limits_{(i_0,i_0 + \Delta]'} \;P_0 \big[& S_{i_0} \le S_i \le T_i \le T_{i_0} < 
\\[-0.5ex]
&\inf\{ s \ge S_0, X_s \in {\rm Res}_{\cL_*,k+1}\}, \;\mbox{for all} 
\\
&i \in (i_0,i_0 + \Delta]'\, \big]
\end{split}
\end{equation}

\medskip\n
(and $\max_{(i_0,i_0 + \Delta]'}$ denotes the maximum over all possible subsets $(i_0,i_0 + \Delta]'$ of $(i_0,i_0 + \Delta]$ with $\Delta - (k-1)$ elements).

\medskip
We then set
\begin{equation}\label{2.31}
\begin{split}
\cL' = \cL \backslash \{ \wh{\ell}_{i_0}\} &\;\; \mbox{(an $\cF_{S_{i_0}}$-measurable subset of $(J+1) \,L \IN$, with}
\\
&\;\; | \cL'| = |\cL| - 1 \ge I - 1 \stackrel{(\ref{2.23})}{\ge} \Delta \ge \Delta - (k-1) \ge 1).
\end{split}
\end{equation}

\n
Now $\cL',S_{i_0},S_i,i \in (i_0,i_0 + \Delta]'$, $\wh{\ell}_i,i \in (i_0,i_0 + \Delta]'$ yield (up to the deterministic increasing relabeling of $[i_0,i_0 + \Delta]'$ into $[0,\Delta -(k-1)])$ a $\Delta - (k-1)$-family, cf.~(\ref{2.11}). In addition, on the event under the probability in (\ref{2.30}), one has for all $s \in [S_{i_0}, T_{i_0}]$, $|X_s - X_{S_{i_0}}|_\infty \le 2 \cdot 2^{-\wh{\ell}_{i_0}}$, so that $P_0$-a.s. on this event, $\wh{\sigma}_{\wh{\ell}_{i_0}}(X_{S_{i_0}}) = \frac{1}{2}$ and $\wt{\sigma}_{\wh{\ell}_{i_0}}(X_s) \in [\wt{\alpha}, 1 - \wt{\alpha}]$ for all $s \in [S_{i_0},T_{i_0}]$.

\medskip
It follows that $P_0$-a.s. on the event under the probability in (\ref{2.30}), $T_{i_0} < \inf\{s \ge S_0; \,X_s \in {\rm Res}_{\cL'_*,k}\}$ (where $\cL'_*$ is defined as in (\ref{2.14}) with $\cL'$ now playing the role of $\cL$). This shows that
\begin{equation}\label{2.32}
\begin{split}
a_2 \le &\; I  \Delta^{k-1} \max\limits_{1 \le i_0 \le I-\Delta}\; \max\limits_{(i_0,i_0 + \Delta]'} P_0 \big[\inf\{s \ge S_{i_0}; X_s \in {\rm Res}_{\cL'_*,k}\} 
\\
&\quad \;\,  >\, \max\limits_{ (i_0,i_0 + \Delta]'} T_i\big] \le  I \Delta^{k-1} \Gamma_k\big(\Delta - (k-1)\big)
\end{split}
\end{equation}

\medskip\n
(and as mentioned below (\ref{2.19}) we dropped the superscript $(J)$ from the notation).

\medskip
Since $\Gamma_k(m) = 1$, for $m \le 0$, by convention, see (\ref{2.19}), one can remove the assumption $\Delta \ge k$ and find that (recall $\Delta = [\sqrt{I}]$)
\begin{equation}\label{2.33}
a_2 \le I \Delta^{k-1}\, \Gamma_k\big(\Delta - (k-1)\big) \le I^{1 + \frac{k-1}{2}} \Gamma_k\big(\Delta - (k-1)\big).
\end{equation}

\medskip\n
Adding the bounds (\ref{2.28}) and (\ref{2.33}), we find, coming back to (\ref{2.24}), that
\begin{equation*}
\begin{array}{l}
P_0\big[\inf\{s \ge S_0; \,X_s \in {\rm Res}_{\cL_*,k+1}\} > \max\limits_{1 \le i \le I} T_i\big] \le 
\\
\big(1-c_2(J)\big)^{\sqrt{I} -1} + I^{1 + \frac{k-1}{2}} \Gamma_k (\Delta - k+1).
\end{array}
\end{equation*}

\medskip\n
Taking the supremum over all possible $I$-families yields (\ref{2.21}) when $I \ge 2$. However, when $I = 1$, the right-hand side of (\ref{2.21}) is at least $1$ due to the first term in the right-hand side, and (\ref{2.21}) holds as well. This completes the proof of (\ref{2.21}) and hence of Lemma \ref{lem2.2}.
\end{proof}

\medskip
We now resume the proof of (\ref{2.10}) of Theorem \ref{theo2.1}. We introduce the quantity (for given $J \ge 1$)
\begin{equation}\label{2.34}
\wt{\Gamma}_k^{(J)}(I) = \left\{ \begin{array}{l}
\sup\limits_{\ell_* \ge 0} \; \;\sup\limits_{U_0 \in \cU_{\ell_*}, L \ge L(J)} \,\Gamma_k^{(J)}(I), \; \mbox{for $1 \le k \le J$ and $I \ge 1$},
\\
\\[-1ex]
1, \;\mbox{for $1 \le k \le J$ and $I \le 0$}.
\end{array}\right.
\end{equation}

\medskip\n
From the definition of $\Gamma_k^{(J)}(I)$, see (\ref{2.18}), (\ref{2.17}), and from the $I$-family (\ref{2.12}), as well as (\ref{2.16}), we see that in the notation of (\ref{2.8}), for $J, I \ge 1$, one has
\begin{equation}\label{2.35}
\Phi_{J,I} = \sup\limits_{\ell_* \ge 0} \;\;\sup\limits_{U_0 \in \cU_{\ell_*}} \;\sup\limits_{L \ge L(J)} \,P_0[H_{{\rm Res}} = \infty] \le \wt{\Gamma}_J^{(J)} (I).
\end{equation}

\medskip\n
On the other hand, by Lemma \ref{lem2.2} and (\ref{2.34}), we see that
\begin{equation}\label{2.36}
\left\{ \begin{array}{l}
\wt{\Gamma}_1^{(J)}(I) = 0, \;\;\mbox{for all $I \ge 1$, and}
\\[1ex]
\wt{\Gamma}^{(J)}_{k+1}(I) \le \big(1 - c_2(J)\big)^{\sqrt{I} - 1} \! + I^{1 + \frac{k-1}{2}} \;\wt{\Gamma}_k^{(J)} (\Delta - k+1), 
\\[1.5ex]
\mbox{for $1 \le k \le  J$ and $I \ge 1$.}
\end{array}\right.
\end{equation}

\medskip\n
We will now prove by induction on $k$ that for $1 \le k \le J$,
\begin{equation}\label{2.37}
\limsup\limits_I I^{-1/2^{k-1}} \log \wt{\Gamma}_k^{(J)} (I) \le \log \big(1 - c_2(J)\big).
\end{equation}

\medskip\n
Choosing $k = J$ will complete the proof of (\ref{2.10}) in view of (\ref{2.35}). 

\medskip
By the first line of (\ref{2.36}), the claim (\ref{2.37}) is immediate when $k = 1$. Then, we assume that (\ref{2.37}) holds for some $1\le k < J$. We have
\begin{equation}\label{2.38}
 \begin{array}{l}
\limsup\limits_I \;I^{-1/2^k} \log \wt{\Gamma}^{(J)}_{k+1} (I) \stackrel{(\ref{2.36})}{\le}
\\[1ex]
\limsup\limits_I \;I^{-1/2^k} \;\max\{(\sqrt{I} - 1) \log\big(1 - c_2(J)\big), \Big(1 + \mbox{\f $\dis\frac{k-1}{2}$}\Big) \log I +
\\
 \log \wt{\Gamma}^{(J)}_k ([\sqrt{I}] - k+1)\} \le  \max\Big\{\log \big(1 - c_2(J)\big),
 \\[1ex]
\limsup\limits_I I^{-1/2^k} ([\sqrt{I}] - k+1)^{1/2^{k-1}} \dis\frac{\log \wt{\Gamma}_k^{(J)} ([\sqrt{I}] - k+1)}{([\sqrt{I}] - k+1)^{1/2^{k-1}}}\Big\} \stackrel{\rm induction}{\le}
 \\
 \log\big(1 - c_2(J)\big).
\end{array}
\end{equation}

\medskip\n
This proves (\ref{2.37}) and as explained above concludes the proof of Theorem \ref{theo2.1}.
\end{proof}

\begin{remark}\label{rem2.4} \rm 1) We describe a simple example showing that the compact set $A$ may lie in the unbounded component of the complement of the resonance set defined in (\ref{2.5}) (in particular, the resonance set need not ``block'' $A$ in a topological sense).

\medskip
We consider  $J \ge 2$, $L \ge L(J)$ $(\ge 5)$, and set $\ell_i = \ell_0 + i L$, for $i \ge 0$, where $\ell_0$ $= \ell_*$ so that (\ref{2.2}) holds. We choose $A = \{0\}$, denote by $e_1$ the first vector of the canonical basis of $\IR^d$, and define
\begin{align*}
U_0 & = \bigcup\limits_{k \ge 0} W_k, \; \mbox{where}
\\
W_0 & = B(0,4 \cdot 2^{-\ell_0}) \; \mbox{and for $k \ge 1$, $W_k = W_{k-1} \cup (x_{k-1} + 2^{-\ell_k} V)$}
\end{align*}

\n
with $x_{k-1}$ the rightmost point of $W_{k-1}$ on $\IR_+ \,e_1$ and $V = [0,8 \cdot 2^L] \times [-4,4]^{d-1}$. That is, $W_k$ is obtained by piling on $W_{k-1}$ in the positive $e_1$-direction a small thin rectangular parallelepiped of length $8 \cdot 2^{-\ell_{k-1}}$ in the $e_1$-direction and side-length $8 \cdot 2^{-\ell_k}$ in the other directions. As a shorthand notation we set $\delta_k = 8 \cdot 2^{-\ell_k}$, for $k \ge 0$.

\medskip
Note that $U_0$ is bounded and $d(0,U_1) = 4 \cdot 2^{-\ell_0}$, so $U_0 \in \cU_{\ell_* = \ell_0}$ in the notation of (\ref{2.7}). As we now explain, one also has
\begin{equation}\label{2.39}
\mbox{$\dsl_{k \ge 0} 1\{\wt{\sigma}_{\ell_k} \big(w(t)\big) \in [\wt{\alpha},1 - \wt{\alpha}]\} \le 1$ for all $t \ge 0$, if $w(t) = t \, e_1$ for $t \ge 0$},
\end{equation}
and consequently, for any $I \ge 1$,
\begin{equation}\label{2.40}
w(t) \notin {\rm Res}(U_0, \ell_*, J, L, I) \; \mbox{for all $t \ge 0$}.
\end{equation}

\medskip\n
To prove (\ref{2.39}), we set for $k \ge 0$, $s_k = x_k \cdot e_1 - \frac{1}{2} \, \delta_k$, $t_k = x_k \cdot e_1 + \frac{1}{2} \, \delta_k$ (recall that $\delta_k = 8 \cdot 2^{-\ell_k}$). Then, the function $t \ge 0 \rightmapsto \wt{\sigma}_{\ell_k} (w(t))$ is continuous and non-decreasing. For $t \le s_k$, we have $d(w(t), U_1) \ge 4 \cdot 2^{-\ell_k}$, so that $\wt{\sigma}_{\ell_k}(w(t)) = 0$, and for $t \ge t_k$, we have $d(w(t),W_k) \ge 4 \cdot 2^{-\ell_k}$, so that $1 - \wt{\sigma}_{\ell_k}(w(t)) \le 2^{-L(d-1)}$, which is smaller than $\wt{\alpha} = \frac{1}{3} \cdot 4^{-d}$ (because $2^{(L-2)(d-1)} > 12$, since $L \ge 5$ and $d \ge 3$). It follows that
\begin{equation}\label{2.41}
\mbox{for any $k \ge 0$, $\{t \ge 0; \, \wt{\sigma}_{\ell_k}(w(t)) \in [\wt{\alpha},1 - \wt{\alpha}]\} \subseteq [s_k, t_k]$}\,.
\end{equation}

\medskip\n
On the other hand, for any $ k \ge 0$, $t_k < s_{k+1}$ and the intervals $[s_k,t_k]$, $k \ge 0$, are pairwise disjoint. The claims (\ref{2.39}) and (\ref{2.40}) follow.

\bigskip\n
2) Here is another simple example showing that the order in which the level sets 
$\{\wt{\sigma}_\ell \in [\wt{\alpha}, 1 - \wt{\alpha}]\}$ are crossed along a Brownian path is in general random. We keep the same notation as in 1) above but now define (see Figure 3)
\begin{align*}
U_0 = &  W_2 \cup \wt{W}_1 \backslash T_2\,,
\\[-2ex]
\intertext{where} 
\\[-6ex]
\wt{W}_1 & \mbox{$= [-x_1 \cdot e_1, \; -x_0 \cdot e_1] \times \big[-\mbox{\f $\dis \frac{\delta_1}{2}$}, \mbox{\f $\dis\frac{\delta_1}{2}$}\big]^{d-1}$ and} 
\\[1ex]
T_2 = &\; (- \infty, -x_1 \cdot e_1 + \delta_1) \times \Big(-\mbox{\f $\dis\frac{\delta_2}{2}$}, \mbox{\f $\dis\frac{\delta_2}{2}$}\Big)^{d-1}. 
\end{align*}

\bigskip
\psfrag{0}{$0$}
\psfrag{l0}{$8 \cdot 2^{-\ell_0}$}
\psfrag{l1}{$8 \cdot 2^{-\ell_1}$}
\psfrag{ll1}{\tiny{$8 \cdot 2^{-\ell_1}$}}
\psfrag{l2}{$8 \cdot 2^{-\ell_2}$}
\psfrag{U0}{$U_0$}
\begin{center}
 \includegraphics[width=13cm]{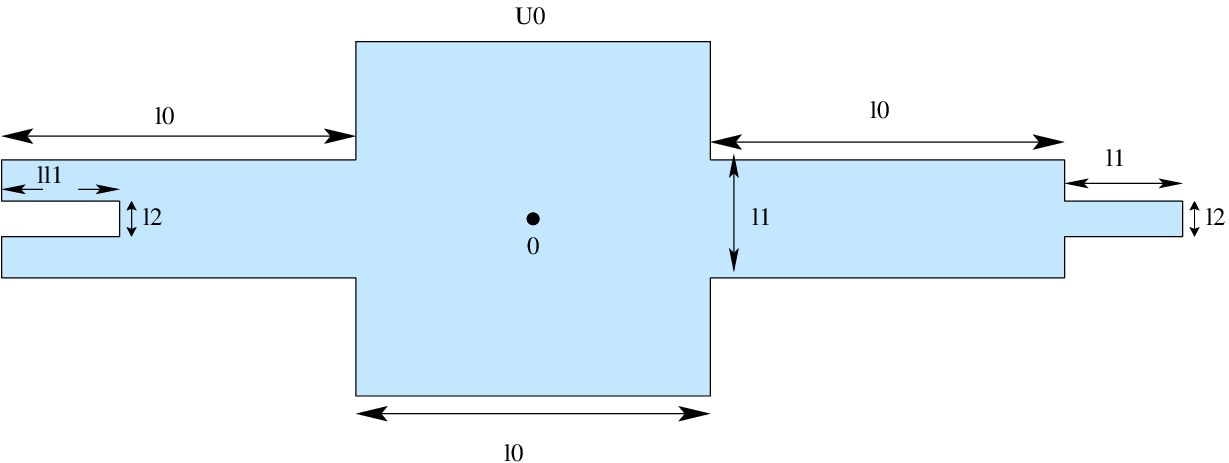}
\end{center}

\begin{center}
Fig.~3: An illustration of $U_0$
\end{center}

\bigskip
In this case similar calculations as in 1) show that the trajectory $w(t) = t e_1$, $t \ge 0$, first reaches the level set $\{\wt{\sigma}_{\ell_0} \in [\wt{\alpha}, 1 - \wt{\alpha}]\}$, then the level set $\{\wt{\sigma}_{\ell_1} \in [\wt{\alpha}, 1 - \wt{\alpha}]\}$, and then the level set $\{\wt{\sigma}_{\ell_2} \in [\wt{\alpha}, 1 - \wt{\alpha}]\}$. On the other hand, $\wt{w}(t) = -w(t)$ first crosses the level set $\{\wt{\sigma}_{\ell_0} \in [\wt{\alpha}, 1 - \wt{\alpha}]\}$, then the level set $\{\wt{\sigma}_{\ell_2} \in [\wt{\alpha}, 1 - \wt{\alpha}]\}$, and then the level set $\{\wt{\sigma}_{\ell_1} \in [\wt{\alpha}, 1 - \wt{\alpha}]\}$. This order remains the same for trajectories in small tubular neighborhoods around $w$ and $\wt{w}$. These tubular neighborhoods have positive measure for $P_0$ (the Wiener measure) and the order in which the level sets $\{\wt{\sigma}_\ell \in [\wt{\alpha}, 1 - \wt{\alpha}]\}$, $\ell \in \{\ell_0,\ell_1,\ell_2\}$ are crossed along a Brownian path is thus random. Similar considerations also apply when one instead considers the order in which the level sets $\{\wh{\sigma}_\ell = \frac{1}{2}\}$ are crossed (recall that $\wh{\sigma}_\ell = \wt{\sigma}_{\ell + 2}$, cf.~(\ref{1.3})). \hfill $\square$
\end{remark}

\section{Solidification of porous interfaces}

We now apply the results of the last section to the study of porous interfaces. The porous interfaces under consideration will be felt within a small distance denoted by $\ve$ from each point of $S = \partial U_0$, where $U_0 \in \cU_{\ell_*,A}$, cf.~(\ref{2.6}), with a strength measured by $\eta$ (we recall that $2^{-\ell_*}$ controls in a suitable sense the distance of $A$ to $U_1 = \IR^d \backslash U_0$, cf.~(\ref{2.6})). The porous interfaces will correspond to hard obstacles or to soft obstacles, and the main Theorem \ref{theo3.1} of this section provides uniform controls on the trapping probability of Brownian motion starting in $A$ by the porous interface when the ratio $\ve/2^{-\ell_*}$ goes to zero. We then derive an asymptotic lower bound on the capacity of the porous interface in Corollary \ref{cor3.4} that will play a crucial role in the next section. We also provide an application in the soft obstacle case in Corollary \ref{cor3.5}. In a heuristic fashion the ``interfaces'' $S = \partial U_0$ can be thought of as some sort of ``segmentation'' of the porous interfaces. One should emphasize that in the classes over which they vary, the interfaces and porous interfaces may undergo degenerations. In certain regions of space they may become brittle and have little trapping power, see for instance Remark \ref{rem3.6} 1). Throughout this section we assume $d \ge 3$.

\medskip
We first need some notation. We consider $U_0$ as in (\ref{1.1}), a non-empty bounded Borel subset of $\IR^d$ and $S = \partial U_0 = \partial U_1$ (with $U_1 = \IR^d \backslash U_0)$ as in (\ref{1.2}). In the hard obstacle case, given $\ve > 0$ and $\eta \in (0,1)$, the porous interfaces will vary in the class
\begin{equation}\label{3.1}
\begin{split}
\cS_{U_0,\ve,\eta}  = &\;\mbox{the class of compact subsets $\Sigma$ of $\IR^d$ such that}
\\
&\;\mbox{$P_x [H_\Sigma <  \tau_\ve] \ge \eta$, for all $x \in \partial U_0$},
\end{split}
\end{equation}

\medskip\n
with $H_\Sigma$ the entrance time of Brownian motion in $\Sigma$ and $\tau_\ve$ the first time it moves at sup-distance $\ve$ from its starting point, cf.~(\ref{1.20}). In the soft obstacle case, the porous interfaces will instead vary in the class
\begin{equation}\label{3.2}
\begin{split}
\cV_{U_0,\ve,\eta}  = &\;\mbox{the class of locally bounded measurable functions $V \ge 0$ on $\IR^d$}
\\
&\;\mbox{such that $E_x \big[\exp\big\{- \int_0^{\tau_\ve} V(X_s) \,ds\big\}\big] \le 1 - \eta$, for all $x \in \partial U_0$}.
\end{split}
\end{equation}

\medskip\n
The next theorem is the main result of this section. It provides in the limit $\ve/2^{-\ell_*}$ going to zero uniform controls on the killing of Brownian motion starting in $A$, when $U_0 \in \cU_{\ell_*,A}$ (with $\ell_* \ge 0$), in the presence of a porous interface corresponding to $\Sigma \in \cS_{U_0,\ve,\eta}$ in the hard obstacle case, or to $V \in \cV_{U_0,\ve,\eta}$ in the soft obstacle case.

\medskip
\begin{thm}\label{theo3.1} (Solidification of porous interfaces)

\medskip\n
Consider a non-empty compact subset $A$ of $\IR^d$ and $\eta \in (0,1)$. Then, in the hard obstacle case, one has
\begin{equation}\label{3.3}
\lim\limits_{u \r 0} \; \;\sup\limits_{\ve \le u \,2^{-\ell_*}} \;\; \sup\limits_{U_0 \in \cU_{\ell_*,A}} \;\; \sup\limits_{\Sigma \in \cS_{U_0,\ve,\eta}} \; \sup\limits_{x \in A} \;P_x[H_\Sigma = \infty] = 0,
\end{equation}

\medskip\n
and the expression under $\lim_{u \r 0}\, \sup_{\ve \le u \,2^{-\ell_*}}$ is maximal for the choice $A = \{0\}$.

\medskip
Likewise, in the soft obstacle case, one has
\begin{equation}\label{3.4}
\lim\limits_{u \r 0} \; \sup\limits_{\ve \le u \,2^{-\ell_*}} \; \sup\limits_{U_0 \in \cU_{\ell_*,A}} \; \sup\limits_{V \in \cV_{U_0,\ve,\eta}} \; \sup\limits_{x \in A} \;E_x\Big[\exp\Big\{ - \dis\int^\infty_0 V(X_s) \,ds\Big\}\Big] = 0,
\end{equation}

\n
and the expression under $\lim_{u \r 0}\, \sup_{\ve \le u \,2^{-\ell_*}}$ is maximal for the choice $A = \{0\}$.
\end{thm}

(When $A = \{0\}$, scaling can also be applied to reformulate (\ref{3.3}) and (\ref{3.4}), see Remark \ref{rem3.6} 2).)

\begin{proof}
Note that in the hard obstacle case, cf.~(\ref{3.1}), when $x \in A$, $U_0 \in \cU_{\ell_*,A}$, $\Sigma \in \cS_{U_0,\ve,A}$, then $U_0 - x \in \cU_{\ell_*}$ ($= \cU_{\ell_*,\{0\}}$, see (\ref{2.7}) for notation), $\Sigma - x \in \cS_{U_0 - x,\ve,\eta}$ and $P_x[H_\Sigma = \infty] = P_0[H_{\Sigma - x} = \infty]$, so the maximality statement for the case $A = \{0\}$ stated below (\ref{3.3}) follows. In the soft obstacles case, cf.~(\ref{3.2}), when $x \in A$, $U_0 \in \cU_{\ell_*,A}$ as above, and $V \in \cV_{U_0,\ve,\eta}$, then $V( \cdot + x) \in \cV_{U_0 - x,\ve,\eta}$ and 
\begin{equation*}
E_x \Big[\exp\Big\{- \dis\int^\infty_0 V(X_s) \,ds\Big\}\Big] = E_0 \Big[\exp\Big\{- \dis\int^\infty_0 V(X_s + x) \,ds\Big\}\Big],
\end{equation*}

\n
whence the maximality of the case $A = \{0\}$ stated below (\ref{3.4}).

\medskip
We now prove (\ref{3.3}) in the case $A = \{0\}$ (the general case follows by the maximality property explained above). The following lemma will be useful (we recall (\ref{1.3}), (\ref{1.35}) and above (\ref{1.1}) for notation). Incidentally, the Reader may possibly first skip its proof and proceed above (\ref{3.8}) to see how the proof of (\ref{3.3}) follows.

\begin{lem}\label{lem3.2}
For $\Sigma \in \cS_{U_0,\ve,\eta}$, $\ell \ge 0$ with $\ve \le \frac{1}{4} \;2^{-\ell}$, $x_0$, $y \in \IR^d$ such that $\wt{\sigma}_\ell (x_0) \in [\wt{\alpha}, 1 - \wt{\alpha}]$ and $|y - x_0|_\infty \le \frac{1}{4} \;2^{-\ell}$, one has
\begin{equation}\label{3.5}
P_y[H_\Sigma < T_{\mathring{B}(x_0,5 \cdot 2^{-\ell})}] \ge c_3 (\eta) \;(> 0)\,.
\end{equation}
\end{lem}

\begin{proof}
Note that $|y - x_0|_\infty \le \frac{1}{4} \;2^{-\ell}$ and $\wt{\sigma}_\ell(x_0) = |B(x_0, 4 \cdot 2^{-\ell}) \cap U_1 | / |B(x_0, 4 \cdot 2^{-\ell})| \in [\wt{\alpha}, 1 - \wt{\alpha}]$, so that using classical properties of the Dirichlet heat kernel, see for instance \cite{Szni98a}, p.~13,~18, as well as scaling and translation invariance, we find that
\begin{equation*}
P_y [X_{4^{-\ell}} \in U_0, \;X_{2 \cdot 4^{-\ell}} \in U_1 \;\;\mbox{and}\;\;  2\cdot 4^{-\ell} < T_{\mathring{B}(x_0, \frac{9}{2} \cdot 2^{-\ell})}] \ge c.
\end{equation*}

\medskip\n
The event under the probability above is contained in $\{H_S < T_{\mathring{B}(x_0,\frac{9}{2} \cdot 2^{-\ell})}\}$ (indeed the continuous trajectory $X_\point$ encounters $U_0$ und $U_1$ during the time interval $[0,2 \cdot 4^{-\ell}]$, and hence meets $S = \ov{U}_0 \cap \ov{U}_1$ during the same time interval). So we have
\begin{equation}\label{3.6}
P_y [H_S < T_{\mathring{B}(x_0, \frac{9}{2} \cdot 2^{-\ell})}] \ge c.
\end{equation}
Then, by the strong Markov property and (\ref{3.1}), we see that since $\ve \le \frac{1}{4} \;2^{-\ell}$,
\begin{equation}\label{3.7}
\begin{array}{l}
P_y[H_\Sigma \circ \theta_{H_S} + H_S < T_{\mathring{B}(x_0, 5 \cdot 2^{-\ell})}]  \ge 
\\
E_y\big[H_S < T_{\mathring{B}(x_0, \frac{9}{2} \cdot 2^{-\ell})}, P_{X_{H_S}}[H_\Sigma < \tau_\ve]\big] \stackrel{(\ref{3.1}),(\ref{3.6})}{\ge}   c\, \eta \stackrel{\rm def}{=} c_3(\eta).
\end{array}
\end{equation}
This completes the proof of Lemma \ref{lem3.2}. 
\end{proof}

\medskip
We can now resume the proof of (\ref{3.3}) (when $A = \{0\}$). We pick $J \ge 1$, $L \ge L(J)$ (see (\ref{1.27})), and for $I \ge 1$, $\ell_* \ge 0$ and $U_0 \in \cU_{\ell_*}$, we write Res for the resonance set Res$(U_0,\ell_*,J,L,I)$, see (\ref{2.5}). We recall the notation $\cA_*$ from (\ref{2.3}), so $\cA_* \subseteq L\IN$ satisfies
\begin{equation}\label{3.8}
\begin{split}
|\cA_*| & = I(J + 1), \;\min \cA_* \ge \ell_*, \;\max \cA_* \le \ell_* + (I+1) (J+1)\,L, \;\mbox{and}
\\
{\rm Res} & = \{x \in \IR^d, \dsl_{\ell \in \cA_*} 1\{\wt{\sigma}_\ell(x) \in [\wt{\alpha},1 - \wt{\alpha}]\} \ge J\}.
\end{split}
\end{equation}

\n
We then consider $u > 0$ and $0 \le \ve \le u \,2^{-\ell_*}$. As we now explain:
\begin{equation}\label{3.9}
\begin{array}{l}
\mbox{when $u < \mbox{\f $\dis\frac{1}{4}$} \;2^{-(I + 1) (J+1)\,L}$, then for any $x_0 \in {\rm Res}$, $\Sigma \in \cS_{U_0,\ve,\eta}$, one has}
\\
P_{x_0}[H_\Sigma > T_{\mathring{B}(x_0, 5 \cdot 2^{- \min \cA_*})}] \le \big(1 - c_3 (\eta)\big)^J.
\end{array}
\end{equation}
 
\n
Indeed, one has $\ve \le u\,2^{-\ell_*} \le  \frac{1}{4}\; 2^{-\max \cA_*}$ by (\ref{3.8}). One applies the strong Markov property at the successive times of exit of the balls 
 $\mathring{B} (x_0, 5 \cdot 2^{-\ell})$, when $\ell \in \cA_*$, and notes that when $\ell' > \ell$ in $\cA_*$, then $\ell' \ge \ell + L$ $(\ge \ell + 5$, see (\ref{1.27})) so that $5 \cdot 2^{-\ell'} \le \frac{1}{4} \cdot 2^{-\ell}$. One can then repeatedly apply (\ref{3.5}) and obtain that  \begin{equation}\label{3.10}
 P_{x_0} [H_\Sigma > T_{\mathring{B}(x_0, 5 \cdot 2^{- \min \cA_*})}] \le \big(1 - c_3(\eta)\big)^{\sum\limits_{\ell \in \cA_*} 1 \{\wt{\sigma}_\ell (x_0) \in [\wt{\alpha}, 1 - \wt{\alpha}]\}},
 \end{equation}
and (\ref{3.9}) follows since $x_0 \in {\rm Res}$.

\medskip
Thus, for $u < \frac{1}{4} \;2^{-(I+1)(J+1)L}$ and $\ve \le u \,2^{-\ell_*}$, we find that for $U_0 \in \cU_{\ell_*}$ and $\Sigma \in \cS_{U_0, \ve,\eta}$
\begin{equation}\label{3.11}
\begin{array}{l}
P_0[H_\Sigma = \infty] \le P_0[H_{\rm Res} = \infty] + E_0\big[H_{\rm Res} < \infty, P_{X_{H_{\rm Res}}} [H_\Sigma = \infty]\big] \stackrel{(\ref{3.9})}{\le}
\\
P_0[H_{\rm Res} = \infty] + \big(1 - c_3(\eta)\big)^J \stackrel{(\ref{2.8})}{\le} \Phi_{J,I} + \big(1 - c_3(\eta)\big)^J.
\end{array}
\end{equation} 

\medskip\n
If we now take the supremum over $\Sigma \in \cS_{U_0, \ve,\eta}$, $U_0 \in \cU_{\ell_*}$ and $\ve \le u\,2^{-\ell_*}$, we find after letting $u$ tend to zero that
\begin{equation}\label{3.12}
\limsup\limits_{u \r 0} \; \sup\limits_{\ve \le u\,2^{-\ell_*}} \; \sup\limits_{U_0 \in \cU_{\ell_*}} \; \sup\limits_{\Sigma \in \cS_{U_0,\ve,\eta}} P_0 [H_\Sigma = \infty] \le \Phi_{J,I} + \big(1 - c_3(\eta)\big)^J\,.
\end{equation}

\medskip\n
Letting $I$ tend to infinity, the first term in the right member of (\ref{3.12}) goes to zero by (\ref{2.10}) of Theorem \ref{theo2.1}. We can then let $J$ tend to infinity and obtain (\ref{3.3}) when $A = \{0\}$ (and hence in the general case).

\medskip
Let us briefly explain how one obtains (\ref{3.4}). With analogous arguments as for Lemma \ref{lem3.2}, one shows instead with the help of (\ref{3.2}):

\begin{lem}\label{lem3.3}
For $V \in \cV_{U_0,\ve,\eta}$, $\ell \ge 0$ with $\ve \le \frac{1}{4} \;2^{-\ell}$, $x_0,y \in \IR^d$ such that $\wt{\sigma}_\ell(x_0) \in [\wt{\alpha},1 - \wt{\alpha}]$ and $|y - x_0|_\infty \le \frac{1}{4}   \cdot2^{-\ell}$, one has
\begin{equation}\label{3.13}
E_y \Big[\exp\Big\{ - \mbox{\f $\dis\int_0^{T_{\mathring{B}(x_0, 5 \cdot 2^{- \ell})}}$}V(X_s) \,ds\Big\}\Big] \le 1 - c_4(\eta).
\end{equation} 
\end{lem}

From this lemma and the strong Markov property, one deduces in place of (\ref{3.9}) that when $u < \frac{1}{4}\;2^{-(I+1)(J+1)L}$, and $\ve \le u \,2^{-\ell_*}$, then for any $U_0 \in \cU_{\ell_*}$, $x_0 \in {\rm Res}$, $V \in \cV_{U_0,\ve,\eta}$, one has
\begin{equation}\label{3.14}
E_{x_0} \Big[\exp\Big\{ - \mbox{\f $\dis\int_0^{\tau_{5 \cdot 2^{-\min \cA_*}}}$} V(X_s) \,ds\Big\}\Big] \le \big(1 - c_4(\eta)\big)^J.
\end{equation} 

\n
One then concludes as below (\ref{3.11}). This completes the proof of Theorem \ref{theo3.1}.
\end{proof}

We can now state a corollary of Theorem \ref{theo3.1} that will play an important role in the next section in our treatment of certain disconnection problems for random interlacements and the simple random walk on $\IZ^d$. We denote by ${\rm cap}(F)$ the Brownian capacity of a compact subset or a bounded open subset $F$ of $\IR^d$, cf.~\cite{PortSton78}, p.~57,~58.

\begin{cor}\label{cor3.4} (Capacity lower bound)

\medskip\n
Consider a compact subset $A$ of $\IR^d$ with positive capacity and $\eta \in (0,1)$. Then, one has
\begin{equation}\label{3.15}
\lim_{u \r 0} \;\; \inf\limits_{\ve \le u \,2^{-\ell_*}} \;\;\inf\limits_{U_0 \in \cU_{\ell_*,A}} \;\; 
\inf\limits_{\Sigma \in \cS_{U_0,\ve,\eta}}\, {\rm cap}(\Sigma) / {\rm cap}(A) = 1.
\end{equation} 

\n
Moreover, for any $\eta \in (0,1)$ one has
\begin{equation}\label{3.15a}
\lim_{u \r 0} \;\; \inf\limits_{\ve \le u \,2^{-\ell_*}} \;\;\inf\limits_{A} \;\;\inf\limits_{U_0 \in \cU_{\ell_*,A}} \;\; 
\inf\limits_{\Sigma \in \cS_{U_0,\ve,\eta}}\, {\rm cap}(\Sigma) / {\rm cap}(A) = 1
\end{equation} 
($A$ varies over the collection of compact sets of positive capacity in the infimum).
\end{cor}

\begin{proof}
We begin with (\ref{3.15}). Note that the quantity under the $\lim_{u \r 0}$ is non-increasing in $u$, so the limit exists. We first prove that it is at least $1$. Denote by $g(\cdot, \cdot)$ the Brownian Green function on $\IR^d$ (which is known to be symmetric), and for $C$ compact subset of $\IR^d$, let $e_C$ and $h_C$ respectively stand for the equilibrium measure and the equilibrium potential of $C$. Then (see for instance \cite{PortSton78}, p.~58 or Chapter 2 \S 3 and \S 4 of \cite{Szni98a}), one has
\begin{equation}\label{3.16}
h_C(x) = P_x[\wt{H}_C < \infty] = \dis\int g(x,y) \,e_C(dy), \; \mbox{for all $x \in \IR^d$},
\end{equation} 

\n
and $h_C = 1$ on $C$ except on a set of zero capacity. Moreover, $e_C$ is supported by $C$, has total mass ${\rm cap}(C)$, and does not charge sets of zero capacity.

\medskip
Now for $A$ as above (\ref{3.15}), for $u >0$, $\ell_* \ge 0$, $0 < \ve \le u\, 2^{-\ell_*}$, $U_0 \in \cU_{\ell_*,A}$, and $\Sigma \in \cS_{U_0,\ve,\eta}$, we have
\begin{equation}\label{3.17}
\begin{array}{lcl}
{\rm cap}(\Sigma) &\!\!\!\!\!\! \ge &\!\!\!\!\!\! \dis\int h_A(y) \;e_\Sigma(dy) \underset{\rm symmetry}{\stackrel{(\ref{3.16})}{=}} \dis\iint g(x,y) \,e_A(dx) \, e_\Sigma(dy)
\\[2ex]
&\!\!\!\!\!\!  \stackrel{(\ref{3.16})}{=}  &\!\!\!\!\!\! \dis\int h_\Sigma(x) \,e_A(dx) = \dis\int P_x [H_\Sigma < \infty] \,e_A(dx)
\\[2ex]
&\!\!\!\!\!\! \ge & \!\!\!\!\!\! \inf\limits_{x \in A} P_x [H_\Sigma < \infty] \;{\rm cap}(A),
\end{array}
\end{equation} 

\n
where the second equality in the second line of (\ref{3.17}) follows from the fact that the set of irregular points $x$ for which $P_x[\wt{H}_\Sigma < \infty] \not= P_x[H_\Sigma < \infty]$, has zero capacity and hence null $e_A$-measure.

\medskip
Thus, ${\rm cap}(\Sigma)/{\rm cap}(A) \ge \inf_{x \in A} \, P_x[H_\Sigma < \infty]$, and taking the infimum in the right member over $\Sigma \in \cS_{U_0,\ve,\eta}$, $U_0 \in \cU_{\ell_*,A}$, $\ve \le u\, 2^{-\ell_*}$, and letting $u$ tend to zero, we see that the limit in (\ref{3.15}) is at least $1$ by (\ref{3.3}) of Theorem \ref{theo3.1}.

\medskip
To show that it equals $1$, denote by $A_{\ell_*}$ the set of points at sup-distance at most $2^{-\ell_*}$ from $A$. Then for $\ve \le u\,2^{-\ell_*}$, choosing $U_0 = A_{\ell_*} = \Sigma$, we see that $U_0 \in \cU_{\ell_*,A}$ and $\Sigma \in \cS_{U_0,\ve,\eta}$. In addition, ${\rm cap}(A_{\ell_*}) \downarrow {\rm cap}(A)$ as $\ell_* \r \infty$, cf.~\cite{PortSton78}, p.~60. Letting $\ve$ go to $0$ and $\ell_*$ to $\infty$ in such a way that $\ve\, 2^{\ell_*}$ tends to $0$, the limit in (\ref{3.15}) is at most $1$. 

\medskip
We now turn to the proof of (\ref{3.15a}). A similar monotonicity argument as in the proof of (\ref{3.15}) shows the existence of the limit. This limit is at most equal to the limit in (\ref{3.15}) (with $A$ an arbitrary compact set of positive capacity) and thus is at most $1$. The limit is also bigger or equal to $1$ as we now explain. Indeed, as above one has ${\rm cap}(\Sigma)/{\rm cap}(A) \ge \inf_{x \in A} \, P_x[H_\Sigma < \infty]$, and it now follows from (\ref{3.3}) and the maximality of the case  $A= \{0\}$ that the limit in (\ref{3.15a}) is at least $1$. This concludes the proof of (\ref{3.15a}), and hence of Corollary \ref{cor3.4}.
\end{proof}

We also state an immediate consequence of Theorem \ref{theo3.1} in the case of soft obstacles concerning the time spent by Brownian motion in the $\ve$-neighborhood of $S = (\partial U_0 = \partial U_1)$, for any $U_0 \in \cU_{\ell_*,A}$. We recall that $d(x,S)$ stands for the sup-distance of $x$ to $S$ (see the beginning of Section 1).

\begin{cor}\label{cor3.5}
For any non-empty compact set $A$ in $\IR^d$, $\ell_* \ge 0$ and $a > 0$,
\begin{equation}\label{3.18}
\lim\limits_{\ve \r 0} \;\; \sup\limits_{U_0 \in \cU_{\ell_*,A}} \;\; \sup\limits_{x \in A} \; E_x\Big[\exp\Big\{- \mbox{\f $\dis\frac{a}{\ve^2}$} \; \dis\int^\infty_0 1\{d (X_s,S) \le \ve\} \,ds\Big\}\Big] = 0.
\end{equation}
\end{cor}

\begin{proof}
We write $V(y) = \frac{a}{\ve^2} \,1\{d(y,S) \le \ve\}$, and note that in the notation of (\ref{3.2}) $V \in \cV_{U_0,\ve,\eta}$, where $1 - \eta = E_0 [\exp\{-a \, \tau_1\}]$ (see (\ref{1.20}) for notation). The claim now follows from (\ref{3.4}) of Theorem \ref{theo3.1}.
\end{proof}

\begin{remark}\label{rem3.6} \rm  1) We briefly sketch here an example showing that the interfaces and porous interfaces under consideration may undergo degenerations in certain parts of space, where they may become brittle and have little trapping power (see Fig.~4 below).

\medskip
We consider $d \ge 4$, $0 \le \ve < 1$, $r(\ve) = n(\ve) \,\ve$, where $n(\ve) = [\ve^{-\frac{2}{d-1}}]$, so that
\begin{equation}\label{3.19}
\mbox{$r(\ve) \sim \ve^{\frac{d-3}{d-1}}$ and $ r(\ve) / \ve \r \infty$, as $\ve \r 0$}.
\end{equation}

\n
We let $\wt{\IL}_\ve$ denote the cable graph consisting of solid segments between neighboring sites in $r(\ve) \, \IZ^d$. We also consider the discrete set $\IL_\ve$ of sites placed at regular spacings $\ve$ along each solid edge of $\wt{\IL}_\ve$. We endow $\IL_\ve$ with its natural graph structure and set $\Gamma_\ve = \IL_\ve \cap B(0,100)$.

\medskip
We consider $0 \le a < \frac{1}{2}$ (fixed) and assume that $\Sigma$ (the ``porous interface'') is such that $\Sigma \cap B(0,100)$ coincides with the restriction of $\bigcup_{x \in \Gamma_\ve} B_{x,a \ve}$ to $B(0,100)$ (where $B_{z,u}$ stands for the closed Euclidean ball of center $z$ and radius $u$). One can arrange $\Sigma$ so that $\Sigma \in \cS_{U_0,\ve,\eta}$ for a suitable $U_0$ and $\eta$ (depending on $a$). For instance, one can consider some spanning tree of $\Gamma_\ve$ rooted at a point on the inner boundary of $\Gamma_\ve$ in $\IL_\ve$, and let $U_0$ behave in the $\ve^2$-neighborhood of $B(0,100)$ as a ``cactus-like'' structure thickening this spanning tree with thin tubes of radius $\ve^2$ with a stem hanging out at the root.

\psfrag{re}{$r(\varepsilon)$}
\psfrag{e}{$\varepsilon$}
\psfrag{x}{$x$}
\begin{center}
 \includegraphics[width=8cm]{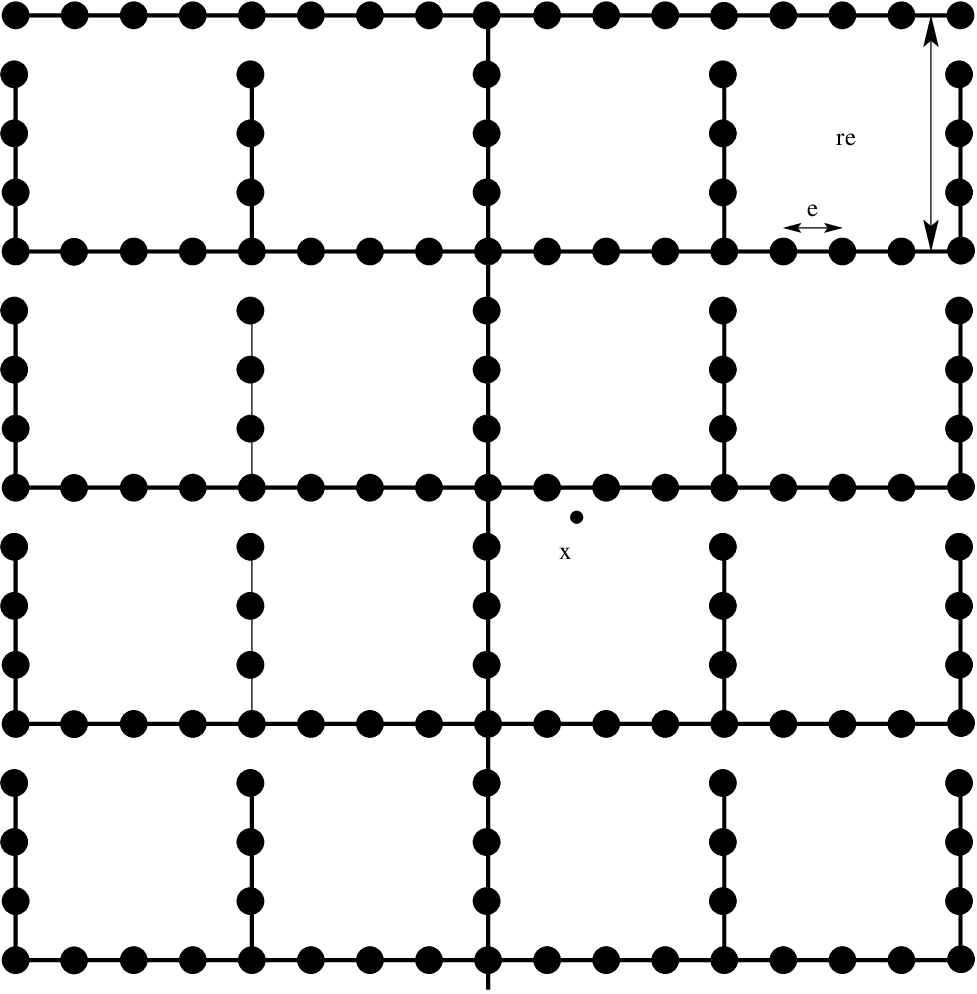}
\end{center}

\begin{center}
\begin{tabular}{lp{12cm}}
Fig.~4: & An illustration of a part of $U_0$ and $\Sigma$, where $\Sigma$ consists of the 
union of the black balls (with radius $a \ve$).  The black lines 
correspond to thin tubes of radius $\ve^2$  (their union constitutes the ``cactus-like'' structure), and the point $x$ belongs to $M_\ve$. 
 \end{tabular}
\end{center}

\medskip
We then consider $M_\ve$, the set of points at Euclidean distance $\ve$ from $\wt{\IL}_\ve$ and sup-distance at most $\frac{r(\ve)}{4}$ from the middle-point of a solid segment of $\wt{\IL}_\ve$ contained in $B(0,10)$. As we now explain, for some constants $c>0$ and $0<c' <1$,
\begin{equation}\label{3.20}
\lim\limits_{\ov{\ve \r 0}}\;\; \inf\limits_{x \in M_\ve} \; P_x[H_\Sigma > \tau_c] \ge c',
\end{equation}

\n
and $\Sigma$ has little trapping power on Brownian motion starting in $M_\ve$. 

\medskip
To see (\ref{3.20}), denote by $D_\ve$ the set of points at Euclidean distance at least $r(\ve)/10$ from $\wt{\IL}_\ve$. As we now explain, one has a constant $c''>0$ such that
\begin{equation}\label{3.21}
\lim\limits_{\ov{\ve \r 0}}\;\; \inf\limits_{x \in M_\ve} \; P_x[H_\Sigma > H_{D_\ve \cap B(0,20)}] \ge c''.
\end{equation}

\n
Indeed, to each $x$ in $M_\ve$, one can attach a unique solid segment of $\wt{\IL}_\ve$, and taking as origin the middle of the segment, one can decompose Brownian motion into a $1$-dimensional motion parallel to the segment and a $(d-1)$-dimensional motion transversal to the segment. One can make sure that the $(d-1)$-dimensional component reaches radius $r(\ve)/10$ before time $C\,r(\ve)^2$ and reaching radius $\frac{\ve}{2}$ with a non-degenerate probability. One can also force the $1$-dimensional component not to move by more than $r(\ve)/8$ from its starting point up to time $C\,r(\ve)^2$, with non-degenerate probability. The lower bound (\ref{3.21}) now follows by independence.

\medskip
Then, for $y \in B(0,20) \cap D_\ve, \; v \le 1$ and $a \,\ve < v$, we have
\begin{equation}\label{3.22}
\begin{array}{l}
P_y [H_\Sigma  < \tau_v] \le \dsl_{z \in \IL_\ve \cap B(y,2v)} P_y [H_{B_{z,a\ve}} < \infty] =
\\[3ex]
\mbox{(with similar notation as below (\ref{3.15a}))}
\\[1ex]
\dsl_{z \in \IL_\ve \cap B(y,2v)} \dis\int g(y,y') \,e_{B_{z,a \ve}} (dy') \le \dsl_{k \in \IZ^d \backslash \{0\} \atop |k r(\ve)|_\infty \le c'v} \; \dis\frac{c}{|k r(\ve)|^{d-2}} \; \dis\frac{r(\ve)}{\ve} \; a^{d-2} \,\ve^{d-2} \le
\\[1ex]
c''\;\dis\frac{\ve^{d-3}}{r(\ve)^{d-3}} \;\Big(\mbox{\f $\dis\frac{v}{r(\ve)}$}\Big)^2 a^{d-2} \stackrel{(\ref{3.19})}{\le} \ov{c} \,v^2 \,a^{d-2} \le \mbox{\f $\dis\frac{1}{2}$} \; \;\mbox{if we choose $v = \wt{c}$ small}.
\end{array}
\end{equation}

\medskip\n
The strong Markov property combined with (\ref{3.21}) and (\ref{3.22}) now yields (\ref{3.20}). Although we will not need it, let us mention that a more detailed analysis of the above example along the lines of the constant capacity regime of small obstacles (see for instance \cite{BaxtJain87}, \cite{CiorMura82}, and also \cite{Szni98a}, p.~116-120) would reveal that $c$ can be chosen equal to $1$ in (\ref{3.20}), but that one does indeed feel $\Sigma$ when starting in $B(0,10)$, in the sense that $\ov{\lim}_{\ve \r 0} \; \sup_{|x|_\infty \le 10} P_x [H_\Sigma > \tau_1] \le 1 - c(a)$.

\bigskip\n
2) In Theorem \ref{theo3.1}, when $A = \{0\}$, scaling can be applied to reformulate (\ref{3.3}) and (\ref{3.4}). Indeed, for any $\ell_* \ge 0$, $U_0 \in \cU_{\ell_*}$ is equivalent to $2^{\ell_*} U_0 \in \cU_0$. Moreover, given $U_0 \in \cU_{\ell_*}$, for $\ve > 0$, $\eta \in (0,1)$, then $\Sigma \in \cS_{U_0,\ve,\eta}$ is equivalent to $2^{\ell_*} \Sigma \in \cS_{2^{\ell_*}U_0,\ve 2^{\ell_*},\eta}$ and $P_0[H_\Sigma < \infty] = P_0 [H_{2^{\ell_*}\Sigma} < \infty]$. Likewise, $V \in \cV_{U_0,\ve,\eta}$ is equivalent to $4^{-\ell_*} V(2^{-\ell_*} \, \cdot) \in \cV_{2^{\ell_*} U_0,\ve 2^{\ell_*},\eta}$, and
\begin{equation*}
E_0 \Big[\exp\Big\{- \dis\int_0^{\infty} V(X_s) \,ds\Big\}\Big] = E_0\Big[\exp\Big\{- \dis\int_0^\infty 4^{-\ell_*} V(2^{-\ell_*} X_s) \, ds \Big\}\Big].
\end{equation*}

\medskip\n
As a result, when $A = \{0\}$, (\ref{3.3}) can be restated as
\begin{equation}\label{3.23}
\lim\limits_{\ve \r 0} \;\; \sup\limits_{U_0 \in \cU_{\ell_* = 0}} \;\; \sup\limits_{\Sigma \in \cS_{U_0,\ve,\eta}} P_0 [ H_\Sigma = \infty] = 0,
\end{equation}
and (\ref{3.4}) can be restated as
\begin{equation}\label{3.24}
\lim\limits_{\ve \r 0} \;\; \sup\limits_{U_0 \in \cU_{\ell_* = 0}} \;\; \sup\limits_{V \in \cV_{U_0,\ve,\eta}} E_0\Big[\exp\Big\{ - \dis\int^\infty_0 V(X_s) \,ds \Big\}\Big] = 0. 
\end{equation}

\smallskip\n
3) In the context of Corollary \ref{cor3.4} the situation is typically simpler when $A$ is a convex set. If $\pi_A$ denotes the projection on $A$, then $\pi_A$ is a contraction for the Euclidean distance. If $\Sigma$ is a porous interface and $\Sigma' = \pi_A(\Sigma)$, then one knows that ${\rm cap}(\Sigma') \le {\rm cap}(\Sigma)$, see \cite{Land72}, p.~58, or \cite{Matt95}, p.~126. In good cases $\Sigma'$ satisfies a Wiener criterion (see for instance \cite{Szni98a}, p.~72), which quantifies that $\Sigma'$ is felt on many scales when starting from $A$, and permits one to show that $\inf_{x \in \partial A} P_x[H_{\Sigma'} < \infty]$ is close to $1$. One can then infer from this fact a lower bound on ${\rm cap}(\Sigma)$.  \hfill $\square$
\end{remark}

\section{Disconnection}

We will now apply the results of the previous section, specifically Corollary \ref{cor3.4}, to derive large deviation upper bounds on the probability that random interlacements in $\IZ^d$, when their vacant set is in a strongly percolative regime, or the simple random walk disconnect a large macroscopic body from the boundary of a large box of comparable size, which contains this body, cf.~Theorem \ref{theo4.1}, Corollary \ref{cor4.4}, and Remark \ref{rem4.5} 3). The macroscopic body in question will correspond to the discrete blow-up of a compact set $A$ in $\IR^d$, with non-empty interior, see (\ref{4.4}), and also Remark \ref{rem4.5} 3) for a variant of this set-up. The main novelty compared to \cite{Szni17}, where $A$ was itself a box, and the method could plausibly have been adapted to the case of a regular compact convex set $A$, is that here we do not require any convexity assumption on $A$. The interfaces $S$ and the porous interfaces $\Sigma$ will arise in the context of a coarse graining procedure, which underlies the large deviation upper bounds that we derive, see (\ref{4.48}), (\ref{4.49}).

\medskip
We briefly introduce some notation concerning random interlacements, and refer to the end of Section 1 of \cite{Szni17} and the references therein for more details. Throughout we assume $d \ge 3$. The random interlacements $\cI^u, u \ge 0$, and the corresponding vacant sets $ \cV^u = \IZ^d \backslash \cI^u$ are defined on a certain probability space $(\Omega, \cA, \IP)$. In essence, $\cI^u$ corresponds to the trace left on $\IZ^d$ by a certain Poisson point process of doubly infinite trajectories modulo time-shift that tend to infinity at positive and negative infinite times, with intensity proportional to $u$. As $u$ grows, $\cV^u$ becomes thinner and it is by now well-known (see for instance \cite{CernTeix12}, \cite{DrewRathSapo14c}) that there is a critical level $u_* \in (0,\infty)$ such that 
\begin{equation}\label{4.1}
\begin{array}{l}
\mbox{for $u < u_*$, $\IP$-a.s., $\cV^u$ has an infinite component},
\\[1ex]
\mbox{for $u > u_*$, $\IP$-a.s., all connected components of $\cV^u$ are finite}. 
\end{array}
\end{equation}

\n
One can further introduce critical values
\begin{align}
&0 < \ov{u} \le u_* \le u_{**} < \infty, \;\mbox{where}\label{4.2}
\\[1ex]
&u_{**} = \inf\{u \ge 0; \, \liminf\limits_L \;\IP[B_{\IZ^d}(0,L) \stackrel{\cV^u}{\longleftrightarrow}  
\partial B_{\IZ^d}(0,2L)] = 0\} \label{4.3}
\end{align}

\n
(the event under the probability corresponds to the existence of a nearest neighbor path in $\cV^u$ between $B_{\IZ^d}(0,L) \stackrel{\rm def}{=} B(0,L) \cap \IZ^d$ and the exterior boundary of $B_{\IZ^d}(0,2L)$).

\medskip
We refer to (2.3) of \cite{Szni17} for the precise definition of $\ov{u}$. It is known to be positive thanks to Theorem 1.1. of \cite{DrewRathSapo14a} (and also \cite{Teix11}, when $d \ge 5$). The ranges $u > u_{**}$ and $0 < u < \ov{u}$ respectively correspond to strongly non-percolative and strongly percolative regimes of the vacant set $\cV^u$. It is plausible but open at the moment that actually $\ov{u} = u_* = u_{**}$.

\medskip
We are interested here in the strongly percolative regime $0 < u < \ov{u}$ for the vacant set $\cV^u$. We consider
\begin{equation}\label{4.4}
\mbox{a compact subset $A$ of $\IR^d$},
\end{equation}
and we assume that $M > 0$ is such that
\begin{equation}\label{4.5}
A \subseteq \mathring{B}(0,M).
\end{equation}
Given $N \ge 1$, the discrete blow-up of $A$ is defined as
\begin{equation}\label{4.6}
A_N = (NA) \cap \IZ^d,
\end{equation}
and we write
\begin{equation}\label{4.7}
S_N = \{x \in \IZ^d; \, |x|_\infty = [MN]\}.
\end{equation}

\medskip\n
We note that for large $N$, i.e.~$N \ge N_0(A,M)$,
\begin{equation}\label{4.8}
A_N \subseteq B_{\IZ^d} (0,MN) \backslash S_N.
\end{equation}
Our interest lies in the disconnection event
\begin{equation}\label{4.9}
\cD^u_N = \{A_N \stackrel{\cV^u}{\longleftrightarrow}  \hspace{-3ex} \mbox{\f $/$} \quad S_N\}
\end{equation}

\medskip\n
(corresponding to the absence of a nearest neighbor path in $\cV^u$ between $A_N$ and $S_N$), and we tacitly assume from now on that $N \ge N_0(A,M)$. The main result of this section is the following asymptotic upper bound.

\begin{thm}\label{theo4.1}
Assume that $A$ is a compact subset of $\IR^d$ and $M > 0$ satisfies (\ref{4.5}). Then, for $0 < u < \ov{u}$, one has (with $\mathring{A}$ the interior of $A$)
\begin{equation}\label{4.10}
\limsup\limits_{N} \;\; \mbox{\f $\dis\frac{1}{N^{d-2}}$} \;\log \IP [\cD^u_N] \le  -\mbox{\f $\dis\frac{1}{d}$} \;(\sqrt{\ov{u}} - \sqrt{u})^2 {\rm cap} (\mathring{A}).
\end{equation}
\end{thm}

\medskip
Of course, when the compact set $A$ is regular in the sense that ${\rm cap}(A) = {\rm cap}(\mathring{A})$, one can replace ${\rm cap}(\mathring{A})$ by ${\rm cap}(A)$ in the right member of (\ref{4.10}). Theorem \ref{theo4.1} has a direct application that yields a similar asymptotic upper bound on the probability that simple random walk disconnects $A_N$ from $S_N$, see Corollary \ref{cor4.4} at the end of this section.

\medskip
Before starting the proof of Theorem \ref{theo4.1}, we introduce some further notation and recall three results from \cite{Szni17}. We consider $0 < u < \ov{u}$, as well as
\begin{equation}\label{4.11}
\mbox{$\alpha > \beta > \gamma$ in $(u,\ov{u})$, and $\wt{\ve} \in (0,1)$ such that $\wt{\ve} \Big(\sqrt{\mbox{\f $\dis\frac{\ov{u}}{u}$}} - 1\Big) < 1$}
\end{equation}

\n
(the parameter $\wt{\ve}$ corresponds to $\ve$ in \cite{Szni17}, but we use here a different notation to avoid a confusion with the parameter $\ve$ of Section 3). We consider from now on an integer $K \ge c(\alpha,\beta,\gamma,\wt{\ve})$ ($\ge 100$, this constant corresponds as below (6.4) of \cite{Szni17} to $c_4(\alpha,\beta,\gamma) \vee c_5(\wt{\ve}) \,\vee c_8(\alpha,\beta,\gamma)$, in the notation of Theorem 2.3, Proposition 3.1 and Theorem 5.1 of \cite{Szni17}). We also consider an integer $L \ge 1$, and for any $z \in \IL \stackrel{\rm def}{=} L \,\IZ^d$, we set
\begin{equation}\label{4.12}
\begin{split}
B_z = & \; z + [0,L)^d \cap \IZ^d \subseteq D_z =  
\\
&\;z+ [-3L,4L)^d \cap \IZ^d \subseteq U_z = z +  [-KL + 1,KL - 1)^d \cap \IZ^d.
\end{split}
\end{equation}

\n
We refer to (2.11) - (2.13) of \cite{Szni17} for the notion of a good$(\alpha, \beta,\gamma)$-box $B_z$ (which is otherwise bad$(\alpha,\beta,\gamma)$). The details of the definition will not be important here. In essence, one looks at the excursions of the interlacements between $D_z$ and the complement of $U_z$. One can order them in a natural fashion, and for a good$(\alpha,\beta,\gamma)$-box $B_z$, the complement of the first $\alpha \,{\rm cap}_{\IZ^d}(D_z)$ excursions leaves in $B_z$ a connected set with sup-norm diameter at least $L/10$, which is connected to similar components in neighboring boxes of $B_z$ via paths in $D_z$ avoiding the first $\beta \, {\rm cap}_{\IZ^d}(D_z)$ excursions (and ${\rm cap}_{\IZ^d}(\cdot)$ stands for the capacity attached to the simple random walk on $\IZ^d$). In addition, the first $\beta \,{\rm cap}_{\IZ^d}(D_z)$ excursions spend a substantial ``local time'' on the (inner) boundary of $D_z$, which is at least $\gamma \,{\rm cap}_{\IZ^d}(D_z)$.

\medskip
We also need the notation $N_u(D_z)$, see (\ref{2.14}) of \cite{Szni17}, which refers to the number of excursions from $D_z$ to $\partial U_z$ (the exterior boundary of $U_z$ in $\IZ^d)$, contained in the interlacement trajectories up to level $u$.

\medskip
We now recall three facts from \cite{Szni17}. First, a connectivity statement, cf.~Lemma 6.1 of \cite{Szni17}:
\begin{equation}\label{4.13}
\begin{array}{l}
\mbox{if $B_{z_i}$, $0 \le i \le n$, is a sequence of neighboring $L$-boxes}
\\
\mbox{(i.e.~the $z_i, 0 \le i \le n$, form a nearest-neighbor path in $\IL$), which are}
\\
\mbox{good$(\alpha,\beta,\gamma)$, and for $i = 0,\dots,n$, $N_u(D_{z_i}) < \beta \,{\rm cap}_{\IZ^d}(D_{z_i})$, then there}
\\
\mbox{exists a path in $(\bigcup^n_{i=0} D_{z_i}) \cap \cV^u$ starting in $B_{z_0}$ and ending in $B_{z_n}$}.
\end{array}
\end{equation}

\n
Second, an exponential bound, cf.~Theorem 3.2 of \cite{Szni17}, namely, if $\cC$ is a non-empty finite subset of $\IL$ with points at mutual sup-distance at least $\ov{K} L$, where $\ov{K} = 2 K+3$, then
\begin{equation}\label{4.14}
\begin{array}{l}
\mbox{$\IP \big[\bigcap\limits_{z \in \cC} \{B_z$ is good$(\alpha,\beta,\gamma)$ and $N_u(D_z) \ge \beta \,{\rm cap}_{\IZ^d}(D_z)\}\big] \le$}
\\[1ex]
\exp\Big\{ - \Big(\sqrt{\gamma} - \mbox{\f $\dis\frac{\sqrt{u}}{1 - \wt{\ve} \,\big(\sqrt{\frac{\ov{u}}{u}} - 1\big)}$} \Big) (\sqrt{\gamma} - \sqrt{u}) \;{\rm cap}_{\IZ^d} (C)\Big\},
\end{array}
\end{equation}

\n
where $C = \bigcup_{z \in \cC} B_z$ (and ${\rm cap}_{\IZ^d}(D_z)$ does not depend on $z$ by translation invariance). The third fact, cf.~(5.1) of \cite{Szni17} (with the choice $\Gamma = 1$), as well as (5.6) and Theorem 5.1 of \cite{Szni17}, is a super-exponential bound
(its proof in \cite{Szni17} uses decoupling via the soft local time technique in the form of Section 2 of \cite{ComeGallPopoVach13}). Namely, there exists a positive function $\rho(L)$ (depending on $\alpha,\beta,\gamma,K$) with $\lim_L \rho(L) = 0$, such that setting
\begin{align}
&N_L = L^{d-1} / \log L, \; \mbox{when $L > 1$}, \label{4.15}
\\[1ex]
&\lim\limits_{L \r \infty} \; N_L^{-(d-2)} \log \IP \big[\mbox{there are at least $\rho(L) (\frac{N_L}{L})^{d-1}$ columns in}\label{4.16}
\\[-0.5ex]
&\hspace{3.4cm} \mbox{$[-N_L,N_L]^d$ in the direction $e$ containing} \nonumber
\\
&\hspace{3.4cm} \mbox{a bad$(\alpha,\beta,\gamma)$-box$\big] = -\infty$}, \nonumber
\end{align}

\medskip\n
where for $e$ vector of the canonical basis of $\IR^d$, a column in $[-N_L,N_L]^d$ in the direction $e$ refers to the collection of $L$-boxes $B_z$ intersecting $[-N_L,N_L]^d$ with same projection on the discrete hyperplane $\{x \in \IZ^d; x \cdot e = 0\}$, and (\ref{4.16}) holds for all vectors $e$ of the canonical basis.

\medskip
It is convenient to introduce the non-increasing function
\begin{equation}\label{4.17}
\begin{split}
\rho_*(u) =& \sup\{ \rho(L'); \,L' \ge [u]\}, \;\mbox{for $u \ge 0$, so that}
\\[1ex]
\rho_*(L) \ge &\,\mbox{$\rho(L)$ for integer $L$, and $\lim\limits_{u \r \infty} \rho_*(u) = 0$}.
\end{split}
\end{equation}

\n
We will now specify the choice of $L$ as a function of $N$ (that enters Theorem \ref{theo4.1}). We do this in a different fashion from \cite{Szni17} (we use here a slightly bigger scale). First, we choose a positive sequence $\gamma_N$, $N \ge 1$, such that
\begin{equation}\label{4.18}
\left\{ \begin{array}{rl}
{\rm i)} & \gamma_N \le 1,  
\\[1ex]
{\rm ii)} & \gamma_N^{2d}/ \rho_*\big((N \log N)^{\frac{1}{d-1}}\big) \underset{N}{\longrightarrow} \infty,
\\[2ex]
{\rm iii)} & \gamma_N^{\frac{d+1}{2}} / (\log N \, N^{-(d-2)}) \underset{N}{\longrightarrow} \infty,
\\[1.5ex]
 {\rm iv)} & \gamma_N \r 0
\end{array}\right.
\end{equation}

\medskip\n
(such a choice is possible since $\lim_{u \r \infty} \rho_*(u) = 0)$.

\medskip
We then define ($L_0$ will stand for our choice of $L$ as a function of $N$)
\begin{equation}\label{4.19}
L_0 = \big[(\gamma_N^{-1} \,N \log N)^{\frac{1}{d-1}}\big] \;\;\mbox{and} \;\; \wh{L}_0 = 100d [\sqrt{\gamma}_N \,N\big].
\end{equation}

\medskip\n
By (\ref{4.18}) iii) for large $N$, $L_0$ is smaller than $\wh{L}_0$. Heuristically, $\wh{L}_0$ can be thought of as ``nearly macroscopic'' (i.e.~of size $N$). As an aside, $\wh{L}_0/N$ (which is comparable to $\sqrt{\gamma}_N$ for large $N$) will in essence correspond to the parameter $\ve$ of Section 3. Together with these choices, we introduce the lattices
\begin{equation}\label{4.20}
\IL_0 = L_0 \, \IZ^d \;\;\mbox{and} \;\; \wh{\IL}_0 = \mbox{\f $\dis\frac{1}{100d}$} \; \wh{L}_0 \, \IZ^d = [\sqrt{\gamma}_N \,N] \, \IZ^d.
\end{equation}

\n
As a last preparation for the proof of Theorem \ref{theo4.1}, we have

\begin{lem}\label{lem4.2}
As $N \r \infty$, one has in the notation of (\ref{4.15}), (\ref{4.19})
\begin{equation}\label{4.21}
\rho(L_0) \;\Big(\mbox{\f $\dis\frac{N_{L_0}}{L_0}$}\Big)^{d-1} / (\wh{L}_0 / L_0)^{d-1} \underset{N}{\longrightarrow} 0.
\end{equation}

\medskip\n
In addition, if one defines the event (where $e$ runs over the canonical basis of $\IR^d$),
\begin{equation}\label{4.22}
\begin{split}
\cB_N = \bigcup\limits_e \big\{& \mbox{there are at least $\rho(L_0) \;\Big(\mbox{\f $\dis\frac{N_{L_0}}{L_0}$}\Big)^{d-1}$ columns of $L_0$-boxes}
\\[-0.5ex]
&\mbox{in the direction $e$ in $B_{\IZ^d} (0,10(M+1) N)$ that contain}
\\
&\mbox{a bad$(\alpha,\beta,\gamma) \,L_0$-box$\big\}$},
\end{split}
\end{equation}
then one has
\begin{equation}\label{4.23}
\lim\limits_N \; \mbox{\f $\dis\frac{1}{N^{d-2}}$} \;\log \IP[\cB_N] = - \infty \;\mbox{(super-exponential bound)}.
\end{equation}
\end{lem}

\begin{proof}
We first prove (\ref{4.21}). For large $N$, on the one hand, we have
\begin{equation}\label{4.24}
(\wh{L}_0 / L_0)^{d-1} \stackrel{(\ref{4.19})}{\ge} c \,\gamma_N^{\frac{d-1}{2} + 1} \;\mbox{\f $\dis\frac{N^{d-2}}{\log N}$} \; \mbox{(which tends to $\infty$ by (\ref{4.18}) iii))},
\end{equation}
and on the other hand
\begin{equation}\label{4.25}
\begin{array}{l}
\rho(L_0) \,(N_{L_0}/L_0)^{d-1} \stackrel{(\ref{4.15})}{=} \rho(L_0) (L^{d-2}_0 / \log L_0)^{d-1} \stackrel{(\ref{4.19})}{\le}
\\[1ex]
\rho(L_0) \, \gamma^{-(d-2)}_N (N \log N)^{d-2} (\log L_0)^{-(d-1)} \stackrel{(\ref{4.19})}{\le}
\\[1ex]
c \,\rho(L_0) \,\gamma^{-(d-2)}_N (N \log N)^{d-2} (\log N)^{-(d-1)} \stackrel{(\ref{4.17}) - (\ref{4.19})}{\le} 
\\[1ex]
c \, \rho_* \big((N \log N)^{\frac{1}{d-1}}\big) \, \gamma_N^{-(d-2)} \; \mbox{\f $\dis\frac{N^{d-2}}{\log N}$} \stackrel{(\ref{4.18})\, ii)}{\le} c \,\gamma_N^{d+2} \; \mbox{\f $\dis\frac{N^{d-2}}{\log N}$}\;.
\end{array}
\end{equation}

\medskip\n
Since $\gamma_N$ tends to zero, cf.~(\ref{4.18}) iv) and $d+2 > \frac{d-1}{2} + 1$, the claim (\ref{4.21}) follows from (\ref{4.24}), (\ref{4.25}).

\medskip
We now turn to the proof of (\ref{4.23}). We note that for large $N$, one has
\begin{equation}\label{4.26}
N_{L_0} \stackrel{(\ref{4.15})}{=} \; \mbox{\f $\dis\frac{L^{d-1}_0}{\log L_0}$} \stackrel{(\ref{4.19})}{\ge} c\,\gamma_N^{-1} \; \mbox{\f $\dis\frac{N \log N}{\log L_0}$}  \ge c' \, \gamma_N^{-1} \,N \ge 10(M + 1) \,N.
\end{equation}

\medskip\n
Hence, for large $N$ the event $\cB_N$ is contained in the union over $e$ of the events under the probability in (\ref{4.16}) where $L_0$ replaces $L$. In addition, by (\ref{4.26}), we know that $N_{L_0} \ge N$ for large $N$ and (\ref{4.23}) follows from (\ref{4.16}). This completes the proof of Lemma \ref{lem4.2}.
\end{proof}

We are now ready to begin the proof of Theorem \ref{theo4.1}. Here is an outline of the proof. We use a coarse graining procedure. In essence, for large $N$, on the disconnection event $\cD^u_N$, cf.~(\ref{4.9}), there will be an interface of $L_0$-boxes, either bad$(\alpha,\beta,\gamma)$ or with $N_u(D_z) \ge \beta \,{\rm cap}_{\IZ^d}(D_z)$, blocking the way between $A_N$ and the complement of $B_{\IZ^d}(0,(M+1)N)$. We will track this interface through boxes of size $\wh{L}_0$ (much larger than $L_0$, but small compared to the macroscopic size $N$), where the interface will have a substantial presence.  This step will involve the inspection of a certain local density in scale $\wh{L}_0$, cf.~(\ref{4.28}), and selecting a region where it is non-degenerate, cf.~(\ref{4.32}). After discarding the bad event $\cB_N$, which is negligible for our purpose, thanks to the super-exponential bound (\ref{4.23}), there will be few bad$(\alpha,\beta,\gamma)$ $L_0$-boxes in each $\wh{L}_0$-boxes, and we will be able to extract a ``porous interface'' made of boxes of size $\wh{L}_0$, so that in each such box there will be a substantial presence of boxes $B_z$ of size $L_0$, all good$(\alpha,\beta,\gamma)$, with $N_u(D_z) \ge \beta {\rm cap}_{\IZ^d}(D_z)$, at mutual distance at least $\ov{K} \,L_0$ (with $\ov{K} = 2 K + 3$ as required in (\ref{4.14})). This selection of $L_0$-boxes will use isoperimetric controls of Deuschel and Pisztora \cite{DeusPisz96} based on an elementary inequality of Loomis and Whitney \cite{LoomWhit49},  in a similar spirit as in Section 2 of \cite{DembSzni06}. These choices will have a small combinatorial complexity, namely $\exp\{o(N^{d-2})\}$, and will produce a coarse graining of the event $\cD_N^u \backslash \cB_N$. By the exponential bound (\ref{4.14}), we will then be reduced to the derivation of a uniform lower bound on ${\rm cap}_{\IZ^d}(C)$, for $C$ the union of the selected $L_0$-boxes. With the help of Proposition \ref{propA.1} of the Appendix, we will be able to replace discrete boxes and random walk capacity with $\IR^d$-boxes and Brownian capacity (first letting $N$ tend to infinity, and then choosing $K$ large). The desired uniform lower bound on the capacity will be provided by Corollary \ref{cor3.4}, where in essence $\ve$ corresponds to $\frac{\wh{L}_0}{N}$ ($\sim {\rm const.} \,\sqrt{\gamma}_N$). The bound in Theorem \ref{theo4.1} will then follow by letting $K$ tend to infinity, and then letting successively $\wt{\ve}$ go to zero and $\gamma$ (together with $\alpha, \beta$) go to $\ov{u}$.

\bigskip\n
{\it Proof of Theorem \ref{theo4.1}:} We recall the definitions of $L_0$ and $\wh{L}_0$ in (\ref{4.19}). Without loss of generality we assume that $\mathring{A} \not= \emptyset$ (otherwise (\ref{4.10}) is immediate) and that $N \ge N_0(A,M)$ so that (\ref{4.8}) holds. We are going to introduce a local density function $\wh{\sigma}(\cdot)$, see (\ref{4.28}) below, in order to track in scale $\wh{L}_0$ an interface of ``blocking $L_0$-boxes'', when $\cD^u_N$ occurs. More precisely, we introduce the random subset
\begin{equation}\label{4.27}
\begin{split}
\cU^1 = &\; \mbox{the union of all $L_0$-boxes $B_z$ that are either contained in}
\\
&\; \mbox{$B_{\IZ^d}(0,(M+1)N)^c$ or linked to an $L_0$-box contained in}
\\
&\; \mbox{$B_{\IZ^d}(0,(M+1)N)^c$ by a path of $L_0$-boxes $B_{z_i}$, $0 \le i \le n$,}
\\
&\;\mbox{which are all, except maybe for the last one, good$(\alpha, \beta,\gamma)$}
\\
&\;\mbox{and such that $N_u(D_{z_i}) < \beta {\rm cap}_{\IZ^d}(D_{z_i})$}.
\end{split}
\end{equation}
We then define the local density function
\begin{equation}\label{4.28}
\wh{\sigma}(x) = |\cU^1 \cap B_{\IZ^d}(x,\wh{L}_0) | / |B_{\IZ^d} (x,\wh{L}_0)|, \;\mbox{for $x \in \IZ^d$}.
\end{equation}
We note that $\wh{\sigma}(\cdot)$ has slow variation in the sense that
\begin{equation}\label{4.29}
|\wh{\sigma}(x+e) - \wh{\sigma}(x)| \le \mbox{\f $\dis\frac{2}{2 \wh{L}_0 + 1}$} \le \mbox{\f $\dis\frac{1}{\wh{L}_0}$} , \;\mbox{for all $x,e$ in $\IZ^d$ with $|e|_1 = 1$}.
\end{equation}

\n
When $B_{\IZ^d} (x,\wh{L}_0 + L_0) \subseteq B_{\IZ^d}(0,(M+1)N)^c$, any $L_0$-box intersecting $B_{\IZ^d}(x,\wh{L}_0)$ is contained in $B_{\IZ^d}(0,(M+1)N)^c$ and hence in $\cU^1$, so that
\begin{equation}\label{4.30}
\wh{\sigma}(x) = 1, \;\mbox{when} \; B_{\IZ^d} (x, \wh{L}_0 + L_0) \subseteq B_{\IZ^d} (0,(M+1)N)^c.
\end{equation}

\n
On the other hand, when $B_{\IZ^d} (x,\wh{L}_0 + L_0) \subseteq A_N (\stackrel{(\ref{4.6})}{=} (NA) \cap \IZ^d)$, any $L_0$-box intersecting $B_{\IZ^d}(x,\wh{L}_0)$ is contained in $A_N$. When $N$ is large, if such a box is contained in $\cU^1$, then by (\ref{4.27}) and the connectivity property (\ref{4.13}), there is a path in $\cV^u$ between $A_N$ und $B_{\IZ^d}(0,MN)^c$ and $\cD_N^u$ does not occur. So, we find that
\begin{equation}\label{4.31}
\mbox{for large $N$ on $\cD^u_N$, $\wh{\sigma}(x) = 0$ when $B_{\IZ^d}(x,\wh{L}_0 + L_0) \subseteq A_N$}.
\end{equation}

\n
Loosely speaking, the random set $\wh{\cS}_N$ that we will now introduce, provides a ``segmentation'' in (nearly macroscopic) scale $\wh{L}_0$ of the interface of blocking (and much smaller) $L_0$-boxes we are interested in. Specifically, we consider the random subset of $\wh{\IL}_0$ (see (\ref{4.20}) for notation)
\begin{equation}\label{4.32}
\wh{\cS}_N = \Big\{x \in \wh{\IL}_0; \; \mbox{\f $\dis\frac{1}{4}$} \le \wh{\sigma}(x) \le \mbox{\f $\dis\frac{3}{4}$} \Big\},
\end{equation}
as well as the compact subset of $\IR^d$
\begin{equation}\label{4.33}
\Delta_N = \bigcup\limits_{x \in \wh{\cS}_N} B\Big(\mbox{\f $\dis\frac{x}{N}$} , \; \mbox{\f $\dis\frac{1}{50d}$} \; \mbox{\f $\dis\frac{\wh{L}_0}{N}\Big)$}
\end{equation}

\medskip\n
(in essence, the unbounded component of the complement of $\Delta_N$ will play the role of $U_1$ in the first three sections of this article, see (\ref{4.49}) below).

\medskip
The Reader may possibly wish to read the statement of the next lemma, first skip its proof, and proceed above (\ref{4.39}) to see how the proof of Theorem \ref{theo4.1} unfolds.

\begin{lem}\label{lem4.3} (Insulation property of $\Delta_N$)

\medskip\n
For large $N$,
\begin{equation}\label{4.34}
\wh{\cS}_N \subseteq B_{\IZ^d}(0,(M+2)N) \cap \wh{\IL}_0,
\end{equation}
and on $\cD_N^u$,
\begin{equation}\label{4.35}
\begin{array}{l}
\mbox{the compact set $\big\{z \in A; d(z,\partial A) \ge \mbox{\f $\dis\frac{\wh{L}_0 + L_0 + 1}{N}$}\big\}$ is contained}
\\
\mbox{in the union of the bounded components of the open set $\IR^d \backslash \Delta_N$}.
\end{array}
\end{equation}
\end{lem}

\begin{proof}
By (\ref{4.30}), we see that for large $N$, $\wh{\sigma}(x) = 1$, when $|x|_\infty > (M+2)N$, and (\ref{4.34}) follows. To prove (\ref{4.35}), we will first show that
\begin{equation}\label{4.36}
\mbox{for large $N$, on $\cD^u_N$ the compact set in (\ref{4.35}) does not intersect $\Delta_N$}.
\end{equation}

\n
Indeed, otherwise scaling up by $N$ and choosing some point in $\IZ^d$ within $| \cdot |_\infty$-distance $1$ of the intersection, we could find $y \in \IZ^d$ with $B_{\IZ^d}(y,\wh{L}_0 + L_0) \subseteq A_N$ and within $|\cdot|_\infty$-distance $(\frac{1}{50d} \;\wh{L}_0 + 1)$ of $\wh{\cS}_N$. This would imply $\wh{\sigma}(y) = 0$ by (\ref{4.31}). But also by (\ref{4.32}) and (\ref{4.29}) that $\wh{\sigma}(y) \ge \frac{1}{4} - d(\frac{1}{50d} \,\wh{L}_0 + 1) \,\frac{1}{\wh{L}_0} \ge \frac{1}{8}$, a contradiction. This proves the claim (\ref{4.36}).

\medskip
As a next step in the proof of (\ref{4.35}), we will show that
\begin{equation}\label{4.37}
\begin{array}{l}
\mbox{for large $N$, on $\cD_N^u$, any continuous path $\psi(\cdot)$: $[0,1] \r \IR^d$, such that}
\\
\mbox{$\psi(0)$ is within $|\cdot |_\infty$-distance $1$ of $\{x \in A_N$: $B_{\IZ^d} (x,\wh{L}_0 + L_0) \subseteq A_N\}$ and}
\\
\mbox{$|\psi(1)|_\infty \ge (M+2)N$, comes within distance  $\frac{1}{50d} \,\wh{L}_0$ of $\wh{\cS}_N$}.
\end{array}
\end{equation}

\medskip\n
This will imply that any continuous path in $\IR^d$ starting in the compact set in (\ref{4.35}) and with end point of $|\cdot |_\infty$-norm at least $M+2$ necessarily meets $\Delta_N$. Together with (\ref{4.36}), this will complete the proof of (\ref{4.35}). There remains to prove (\ref{4.37}).

\medskip
Given $\psi(\cdot)$ as in (\ref{4.37}), we can construct a $\IZ^d$-valued $*$-path $y_i, 0 \le i \le \ell$ (i.e. $|y_{i+1} - y_i|_\infty = 1$, for $0 \le i < \ell$) such that $y_0 \in \{x \in A_N$; $B_{\IZ^d} (x,\wh{L}_0 + L_0) \subseteq A_N\}$ and $|y_\ell|_\infty > (M+1) N + \wh{L}_0 + L_0$ and $\{y_0,\dots,y_\ell\}$ is contained in the closed $1$-neighborhood of $\psi ([0,1]$) for the $|\cdot |_\infty$-distance. Indeed, one constructs by induction a non-decreasing sequence of times $t_i \in [0,1]$ and points $y_i, 1 \le i \le \ell$, so that $t_i$ is the first time after $t_{i-1}$ (with $t_0 = 0$) when the continuous path $\psi$ moves at $| \cdot |_\infty$-distance $1$ from the current point $y_{i-1}$ of the sequence to choose the next $y_i$ in $\IZ^d$ so that $|\psi(t_i) - y_i|_\infty \le \frac{1}{2}$ and $|y_i - y_{i-1}|_\infty = 1$. The procedure stops after finitely many steps (by the continuity of $\psi$) with a point $y_\ell$ that satisfies $|y_\ell - \psi(1)|_\infty \le 1$ (whence $|y_\ell |_\infty > (M+1) \,N + \wh{L}_0 + L_0$).

\medskip
It now follows from (\ref{4.30}) and (\ref{4.31}) that $\wh{\sigma}(y_0) = 0$ and $\wh{\sigma}(y_\ell) = 1$. In addition, $|y_{i+1} - y_i|_1 \le d$, for $0 \le i < \ell$, and by (\ref{4.29}), we find that necessarily for some $0 \le j \le \ell$, one has $|\wh{\sigma}(y_j) - \frac{1}{2}| \le \frac{d}{\wh{L}_0}$.

\medskip
For large $N$, if we now choose $\wh{y} \in \wh{\IL}_0$ such that $|\wh{y} - y_j|_\infty \le \frac{1}{100d} \, \wh{L}_0$, cf.~(\ref{4.20}), we find that
\begin{equation}\label{4.38}
\Big|\wh{\sigma}(\wh{y}) - \mbox{\f $\dis\frac{1}{2}$} \Big| \stackrel{(\ref{4.29})}{\le}  \mbox{\f $\dis\frac{1}{\wh{L}_0}$} \;|\wh{y} - y_j |_1 +  \mbox{\f $\dis\frac{d}{\wh{L}_0}$} \le \mbox{\f $\dis\frac{1}{100}$}  + \mbox{\f $\dis\frac{d}{\wh{L}_0}$} < \mbox{\f $\dis\frac{1}{4}$}.
\end{equation}

\n
This shows that $\wh{y}$ belongs to $\wh{\cS}_N$, cf.~(\ref{4.32}). In addition, $\psi(\cdot)$ comes within $|\cdot|_\infty$-distance $\frac{1}{100d} \,\wh{L}_0 + 1 \le \frac{1}{50d} \, \wh{L}_0$ of $\wh{y}$, and thus of $\wh{\cS}_N$. This yields (\ref{4.37}) and concludes the proof of Lemma \ref{lem4.3}.
\end{proof}

We are now on our way to introduce a coarse graining of the event $\cD^u_N \backslash \cB_N$ (we refer to (\ref{4.22}) for the definition of the ``bad'' event $\cB_N$), see (\ref{4.42}), (\ref{4.43}) below. As a next step, we extract a (measurable random) subset $\wt{\cS}_N$ of the (measurable random) finite set $\wh{\cS}_N$ such that:
\begin{equation}\label{4.39}
\begin{array}{l}
\mbox{$\wt{\cS}_N$ is a maximal subset of $\wh{\cS}_N$ such that the $B_{\IZ^d} (x, 2 \wh{L}_0)$, $x \in \wt{\cS}_N$,}
\\
\mbox{are pairwise disjoint}. 
\end{array}
\end{equation}

\n
For any $x \in \wt{\cS}_N$, we have $\wh{\sigma}(x) \in [\frac{1}{4}, \frac{3}{4}]$ by (\ref{4.32}). By the isoperimetric controls (A.3) - (A.6), p.~480-481 of \cite{DeusPisz96}, we have a projection $\wt{\pi}_x$ on the hyperplane of points with vanishing $\wt{i}_x$-coordinate, such that the $\wt{\pi}_x$-image of the points in $B_{\IZ^d}(x,\wh{L}_0) \backslash \cU^1$ having a neighbor in $B_{\IZ^d}(x,\wh{L}_0) \cap \cU^1$ has cardinality at least $c \, \wh{L}_0^{d-1}$ (see (\ref{4.27}) for notation). Any such point in $B_{\IZ^d} (x, \wh{L}_0) \backslash \cU^1$, which is a neighbor of a point in $B_{\IZ^d}(x, \wh{L}_0) \cap \cU^1$, belongs to a (uniquely defined) $L_0$-box $B_z$, $z \in \IL_0$, which is not contained in $B_{\IZ^d}(0,(M+1)N)^c$ (otherwise the point would belong to $\cU^1$), but also cannot be both good$(\alpha,\beta,\gamma)$ and with $N_u(D_z) < \beta\, {\rm cap}(D_z)$ (otherwise the point would again belong to $\cU^1$, since a neighboring $L_0$-box is contained in $\cU^1$). 

\medskip
As a result, we see that for large $N$,
\begin{equation}\label{4.40}
\begin{array}{l}
\mbox{for each $x \in \wt{\cS}_N$, there is a collection of $L_0$-boxes intersecting $B_{\IZ^d}(x,\wh{L}_0)$}
\\
\mbox{with $\wt{\pi}_x$-projection at mutual distance $\ge \ov{K}\,L_0$ (where $\ov{K} = 2 K + 3)$,}
\\
\mbox{and cardinality at least $(\frac{c}{K} \; \frac{\wh{L}_0}{L_0})^{d-1}$, such that each $B_z$ in the collection}
\\
\mbox{is either bad$(\alpha,\beta,\gamma)$ or $N_u(D_z) \ge \beta\, {\rm cap}_{\IZ^d}(D_z)$}.
\end{array}
\end{equation}

\n
Now for large $N$, on the complement of the event $\cB_N$ in (\ref{4.22}), in each coordinate direction, the number of columns of $L_0$-boxes in $B_{\IZ^d} (0,10(M+1)N)$ that contain a bad$(\alpha,\beta,\gamma)$ $L_0$-box is at most $\rho(L_0) (\frac{N_{L_0}}{L_0})^{d-1} \le (\frac{c}{2K} \; \frac{\wh{L}_0}{L_0})^{d-1}$ by (\ref{4.21}). This observation combined with (\ref{4.40}) shows that
\begin{equation}\label{4.41}
\begin{array}{l}
\mbox{for large $N$, on the event $\wt{\cD}^u_N = \cD_N^u \backslash \cB_N$, for each $x \in \wt{\cS}_N$ there is a}
\\
\mbox{collection $\wt{\cC}_x$ of $L_0$-boxes intersecting $B_{\IZ^d} (x,\wh{L}_0)$ with $\wt{\pi}_x$-projection}
\\
\mbox{at mutual distance at least $\ov{K} \,L_0$ and cardinality $[(\frac{c'}{K} \;\frac{\wh{L}_0}{L_0})^{d-1}]$, such}
\\
\mbox{that for each $z \in \wt{\cC}_x$, $B_z$ is good$(\alpha,\beta,\gamma)$ and $N_u(D_z) \ge \beta \, {\rm cap}_{\IZ^d} (D_z)$.}
\end{array}
\end{equation}

\n
Thus, for large $N$, we can define a random variable
\begin{equation}\label{4.42}
\kappa_N = (\wh{\cS}_N, \wt{\cS}_N, (\wt{\pi}_x, \wt{\cC}_x)_{x \in \wt{\cS}_N}) \;\;\mbox{on} \;\; \wt{\cD}^u_N (= \cD^u_N \backslash \cB_N)
\end{equation}
with the above mentioned properties. As we now explain
\begin{equation}\label{4.43}
\mbox{the set $\cK_N$ of possible values of $\kappa_N$ has cardinality $|\cK_N| = \exp\{o(N^{d-2})\}$}.
\end{equation}

\medskip\n
Indeed, for large $N$, by (\ref{4.34}), (\ref{4.39}), there are at most $2^{2| \wh{\IL}_0 \cap B_{\IZ^d}(0,(M+2)N)|}$ possible choices for the couple $(\wh{\cS}_N, \wt{\cS}_N)$, and this quantity is at most 
\begin{equation*}
\exp\Big\{c(M^d + 1) \Big(\mbox{\f $\dis\frac{N}{\wh{L}_0}$}\Big)^d\Big\} \le \exp\{c(M^d + 1) \,\gamma_N^{-\frac{d}{2}}\}.
\end{equation*}
Then, for any such choice of $\wh{\cS}_N$ and $\wt{\cS}_N$, for any $x \in \wt{\cS}_N$, we can choose $(\wt{\pi}_x, \wt{\cC}_x)$ in at most $d(c \,\frac{\wh{L}_0}{L_0})^{(c \frac{\wh{L}_0}{L_0})^{d-1}}$ ways, and hence the number of choices for $(\wt{\pi}_x,\wt{\cC}_x)_{x \in \wt{\cS}_N}$ is at most
\begin{equation}\label{4.44}
\exp\Big\{c'(M^d + 1) \Big(\mbox{\f $\dis\frac{N}{\wh{L}_0}$}\Big)^d \Big(\log d + \Big(\mbox{\f $\dis\frac{\wh{L}_0}{L_0}$}\Big)^{d-1} \log \Big(c \mbox{\f $\dis\frac{\wh{L}_0}{L_0}$}\Big)\Big)\Big\}.
\end{equation}
Note that for large $N$, one has
\begin{equation*}
\begin{array}{l}
\Big(\mbox{\f $\dis\frac{N}{\wh{L}_0}$}\Big)^d \;  \Big(\mbox{\f $\dis\frac{\wh{L}_0}{L_0}$}\Big)^{d-1} \log \Big(\mbox{\f $\dis\frac{\wh{L}_0}{L_0}$}\Big) = \Big(\mbox{\f $\dis\frac{N}{\wh{L}_0}$}\Big)  \; \Big(\mbox{\f $\dis\frac{N}{L_0}$}\Big)^{d-1} \log  \Big(\mbox{\f $\dis\frac{\wh{L}_0}{L_0}$}\Big) \stackrel{(\ref{4.19})}{\le}
\\[2ex]
c \, \gamma_N^{-\frac{1}{2}} \gamma_N \; \mbox{\f $\dis\frac{N^{d-2}}{\log N}$} \;\log N = c \,\gamma_N^{\frac{1}{2}} \,N^{d-2}.
\end{array}
\end{equation*}

\medskip\n
We thus find that for large $N$, the number $|\cK_N|$ of possible values of $\kappa_N$ is at most $\exp\{c(M^d + 1)(\gamma_N^{-d/2} + \gamma_N^{1/2} \,N^{d-2})\} \stackrel{(\ref{4.18}) i), iii), iv)}{=} \exp\{o(N^{d-2})\}$, whence (\ref{4.43}).

\medskip
For large $N$, the coarse graining of $\wt{\cD}^u_N$ will specifically correspond to the partition:
\begin{align}
\wt{\cD}^u_N =  & \mbox{\f $\dis\bigcup\limits_{\kappa \in \cK_N}$} \cD_{N,\kappa}, \;\mbox{where $\cD_{N,\kappa} = \wt{\cD}^u_N \cap \{ \kappa_N = \kappa\}$, for $\kappa \in \cK_N$} \label{4.45}
\\
&\mbox{(and $|\cK_N| = \exp\{ o(N^{d-2})\}$)}. \nonumber
\end{align}
We can now combine the super-exponential bound in (\ref{4.23}) and the above coarse graining to find that
\begin{equation}\label{4.46}
\begin{split}
\limsup\limits_N \; \dis\frac{1}{N^{d-2}} \log \IP[\cD_N^u] \stackrel{(\ref{4.23})}{\le} & \limsup\limits_N \; \mbox{\f $\dis\frac{1}{N^{d-2}}$} \; \log \IP [\wt{\cD}^u_N]
\\
\stackrel{(\ref{4.45})}{\le} & \limsup\limits_N \;\sup\limits_{\kappa \in \cK_N} \; \mbox{\f $\dis\frac{1}{N^{d-2}}$}\; \log \IP[\cD_{N,\kappa}].
\end{split}
\end{equation}

\medskip\n
To each $\kappa \in \cK_N$, we will associate a ``segmentation'' (corresponding to $U_0$ or $S$ in (\ref{4.49}) below) and a ``porous interface'' (corresponding to $\Sigma$  in (\ref{4.48}) below), as follows. If $\kappa = (\wh{\cS}, \wt{\cS}, (\wt{\pi}_x, \wt{\cC}_x)_{x \in \wt{\cS}}) \in \cK_N$, we define
\begin{equation}\label{4.47}
C = \mbox{\f $\dis\bigcup\limits_{x \in \wt{\cS}}$} \; \mbox{\f $\dis\bigcup\limits_{z \in \wt{\cC}_x}$} \; B_z \; (\subseteq \IZ^d),
\end{equation}
its ``scaled $\IR^d$-filling'':
\begin{equation}\label{4.48}
\Sigma = \mbox{\f $\dis\frac{1}{N}$} \; \Big(\mbox{\f $\dis\bigcup\limits_{x \in \wt{\cS}}$} \; \mbox{\f $\dis\bigcup\limits_{z \in \wt{\cC}_x}$}  \;z + [0,L_0]^d\Big) \; (\subseteq \IR^d),
\end{equation}

\vspace{-3ex}\n
as well as
\begin{equation}\label{4.49}
\begin{split}
U_1 & = \; \mbox{the unbounded component of $\IR^d \backslash \Big\{\mbox{\f $\dis\frac{1}{N}$} \Big(\bigcup\limits_{x \in \wh{\cS}} \;B\Big(x , \mbox{\f $\dis\frac{1}{50d}$} \;\wh{L}_0\Big)\Big)\Big\}$}
\\[-1ex]
U_0 & = \IR^d \backslash U_1, \; S = \partial U_0 = \partial U_1.
\end{split}
\end{equation}

\n
Note that on $\cD_{N,\kappa}$ the open set $U_1$ coincides with the unbounded component of $\IR^d \backslash \Delta_N$ in the notation of (\ref{4.33}). Thus, by Lemma \ref{lem4.3}, for large $N$ and all $\kappa$ in $\cK_N$, on $\cD_{N,\kappa}$, the compact set $\{z \in A; d(z,\partial A) \ge \frac{\wh{L}_0 + L_0 + 1}{N}\}$ in (\ref{4.35}) of Lemma \ref{lem4.3} does not intersect $U_1$. We then consider a compact subset $A'$ of $\mathring{A}$ and some $\ell_*(A,A') \ge 0$ such that, for large $N$ and all $\kappa \in \cK_N$, 

\begin{equation}\label{4.50}
d(A',U_1) \ge 2^{-\ell_*} \;\mbox{(and hence $U_0 \in \cU_{\ell_*, A'}$, see (\ref{2.6}))}.
\end{equation}

\n
In addition, by (\ref{4.39}), (\ref{4.41}), and the definition of $\Sigma$ in (\ref{4.48}), that for large $N$ and all $\kappa \in \cK_N$ and $x \in \wt{\cS}$, ${\rm cap}( \mbox{\f $\dis\bigcup\,\!\!_{z \in \wt{\cC}_x}$} \; z + [0,L_0]^d ) \ge c(K) \wh{L}_0^{d-2}$ (using a projection argument, see \cite{Matt95}, p.~126), and
\begin{equation}\label{4.51}
P_x[H_\Sigma < \tau_{10 \,\frac{\wh{L}_0}{N}}] \ge c(K), \;\mbox{for all $x \in S$ ($=\partial U_0 = \partial U_1$)}.
\end{equation}

\n
This will enable us to apply Corollary \ref{cor3.4} of the previous section. We now come back to (\ref{4.46}), and note that by the exponential bound (\ref{4.14}), as well as our choice of $\wt{\cC}_x$ in (\ref{4.41}) and the notation (\ref{4.47}),
\begin{equation}\label{4.52}
\begin{array}{l}
\limsup\limits_N \; \mbox{\f $\dis\frac{1}{N^{d-2}}$} \;\log \IP[\cD^u_N] \le 
\\
-\Big(\sqrt{\gamma} - \dis\frac{\sqrt{u}}{1 - \wt{\ve}(\sqrt{\frac{\ov{u}}{u}} - 1)}\Big) (\sqrt{\gamma} - \sqrt{u}) \;\liminf\limits_N \;  \inf\limits_{\kappa \in \cK_N} \; \mbox{\f $\dis\frac{1}{N^{d-2}}$} {\rm cap}_{\IZ^d} \;(C).
\end{array}
\end{equation}

\n
Taking a liminf over $K$ and using Proposition \ref{propA.1}, we find with $\Sigma$ as in (\ref{4.48}):
\begin{equation}\label{4.53}
\begin{array}{l}
\limsup\limits_N \; \mbox{\f $\dis\frac{1}{N^{d-2}}$} \;\log \IP[\cD^u_N] \le 
\\
- \Big(\sqrt{\gamma} - \dis\frac{\sqrt{u}}{1 - \wt{\ve}(\sqrt{\frac{\ov{u}}{u}} - 1)}\Big) (\sqrt{\gamma} - \sqrt{u}) \;\underset{K}{\underline{\lim}} \; \underset{N}{\underline{\lim}}  \; \inf\limits_{\kappa \in \cK_N} \;\mbox{\f $\dis\frac{1}{d}$} \;{\rm cap}(\Sigma).
\end{array}
\end{equation}

\n
We can now take (\ref{4.50}), (\ref{4.51}) into account, and by Corollary \ref{cor3.4}, we find that
\begin{equation}\label{4.54}
\liminf\limits_N \;  \;\inf\limits_{\kappa \in \cK_N} \;  {\rm cap}(\Sigma) \ge {\rm cap}(A').
\end{equation}

\n
Inserting this inequality in the last expression of (\ref{4.53}), and letting successively $\wt{\ve}$ tend to zero and $\alpha,\beta,\gamma$ tend to $\ov{u}$ yields
\begin{equation}\label{4.55}
\limsup\limits_N \; \mbox{\f $\dis\frac{1}{N^{d-2}}$} \;  \log \IP[\cD^u_N] \le - \mbox{\f $\dis\frac{1}{d}$} \;(\sqrt{\ov{u}} - \sqrt{u})^2\,  {\rm cap}(A').
\end{equation}

\n
Letting $A'$ increase to $\mathring{A}$ now yields (\ref{4.10}), by Proposition 1.13, p.60 of \cite{PortSton78}. This completes the proof of Theorem \ref{theo4.1}.

 \hfill $\square$

\medskip
In the proof of the above Theorem \ref{theo4.1}, we applied Corollary \ref{cor3.4} with a fixed $\ell_*$, which does not vary with $N$, see (\ref{4.50}). We refer to Remark \ref{rem4.5} 3) below for an application where $\ell_*$ depends on $N$.

\medskip
Theorem \ref{theo4.1} has an immediate application to a similar asymptotic upper bound in the case of disconnection by a simple random walk. We first introduce some notation. We denote by $Z_n$, $n \ge 0$, the canonical simple random walk on $\IZ^d$, (recall $d \ge 3$), and by $P_x^{\IZ^d}$ the canonical law of the walk starting at $x \in \IZ^d$. We write $\cI = \{Z_n; n \ge 0\} \subseteq \IZ^d$ for the set of points visited by the walk and $\cV = \IZ^d \backslash \cI$ for its complement. With $A_N$ and $S_N$ as in (\ref{4.6}), (\ref{4.7}), and $N \ge N_0(A,M)$ such that $A_N \subseteq B_{\IZ^d}(0,MN) \backslash S_N$, we consider the disconnection event where there is no path in $\cV$ between $A_N$ and $S_N$,
\begin{equation}\label{4.56}
\cD_N = \{A_N \stackrel{\cV}{\longleftrightarrow} \hspace{-3ex} \mbox{\f $/$} \quad S_N\}.
\end{equation}
We have
\begin{cor}\label{cor4.4}
Assume that $A$ is a compact subset of $\IR^d$ and $M > 0$ satisfies (\ref{4.5}), then for any $x \in \IZ^d$,
\begin{equation}\label{4.57}
\limsup\limits_N \; \mbox{\f $\dis\frac{1}{N^{d-2}}$} \;\log P_x^{\IZ^d} [\cD_N] \le - \mbox{\f $\dis\frac{1}{d}$} \;\ov{u} \,{\rm cap}(\mathring{A}).
\end{equation}
\end{cor}

\begin{proof}
The proof is similar to that of Corollary 6.4  of \cite{Szni17} (or that of Corollary 7.3 of \cite{Szni15}) and relies on the fact that one can find a coupling $\ov{P}$ of $\cI^u$ under $\IP[\cdot | x \in \cI^u]$ and of $\cI$ under $P_x$, so that $\ov{P}$-a.s., $\cI \subseteq \cI^u$. The claim (\ref{4.57}) then follows by the application of Theorem \ref{theo4.1} to $u \in (0,\ov{u})$, and then letting $u \r 0$ (see the above mentioned references for details).
\end{proof}

\begin{remark}\label{rem4.5} \rm 1) The approach developed in this section should remain pertinent in the context of level-set percolation of the Gaussian free field on $\IZ^d$, $d \ge 3$. Plausibly, one should be able to adapt the strategy of the proof of Theorem \ref{theo4.1} to instead derive asymptotic upper bounds on the probability that the excursion-set of the Gaussian free field {\it below} level $\alpha$ disconnects the macroscopic body $A_N$ from $S_N$, when $\alpha$ is such that the excursion-set of the Gaussian free field {\it above} $\alpha$ is in a strongly percolative regime (i.e.~$\alpha < \ov{h}$ in the notation of \cite{Szni15}). The case when $A_N$ is the discrete blow-up of a box centered at the origin was treated in \cite{Szni15}.

\bigskip\n
2) As mentioned in the introduction, it is plausible but presently open that the critical levels $\ov{u} \le u_* \le u_{**}$ actually coincide (incidentally, some progress towards a possible proof of the equality $u_* = u_{**}$ has been made in \cite{DumiRaouTass17}). If this is the case and $\ov{u} = u_* = u_{**}$, then when the compact set $A$ is regular in the sense that
\begin{equation}\label{4.58}
{\rm cap}(A) = {\rm cap}(\mathring{A}),
\end{equation}

\n
the asymptotic upper bounds in Theorem \ref{theo4.1} and Corollary \ref{cor4.4} respectively match the asymptotic lower bounds from \cite{LiSzni14} in the case of random interlacements, and from \cite{Li17} in the case of simple random walk and under (\ref{4.58})
\begin{align}
&\lim\limits_N \; \mbox{\f $\dis\frac{1}{N^{d-2}}$} \; \log \IP[\cD^u_N] = - \mbox{\f $\dis\frac{1}{d}$} \;(\sqrt{u}_* - \sqrt{u})^2 \,{\rm cap}(A), \;\mbox{when $0 < u < u_*$}, \label{4.59}
\intertext{and}
&\lim\limits_N \; \mbox{\f $\dis\frac{1}{N^{d-2}}$} \; \log \IP[\cD_N] = - \mbox{\f $\dis\frac{1}{d}$}  \; u_* \,{\rm cap}(A). \label{4.60}
\end{align}

\medskip\n
3) Assume that for each rationals $\alpha>\beta>\gamma$ in $(0,\ov{u})$, $\wt{\ve} \in (0,1)$  and integer $K \ge c(\alpha,\beta,\gamma,\wt{\ve})$ (in the notation below (\ref{4.11})) one chooses a positive sequence $\gamma_N$ as in (\ref{4.18}). Then, the proof of Theorem \ref{theo4.1} can straightforwardly be adapted to the situation where in place of $A_N$ in (\ref{4.6}) one considers the discrete blow-up $\wh{A}_N = (NA^{\delta_N}) \cap \IZ^d$ with $A^{\delta_N}$ the closed $\delta_N$-neighborhood in $|\cdot|_\infty$-distance of the compact set $A$, for a sequence $\delta_{N} \underset{N}{\longrightarrow} 0$ in such a fashion that $\delta_N / (\frac{\wh{L}_0}{N}) 
\underset{N}{\longrightarrow} \infty$ (i.e.~$\delta_N/\sqrt{\gamma}_N  \underset{N}{\longrightarrow} \infty$), for all above choices of $\gamma_N$. Such a sequence $\delta_N$ can for instance be constructed by a diagonal procedure. In essence, one only needs to replace $A$ by $A^{\delta_N}$ and $\partial A$ by $\partial (A^{\delta_N})$ in (\ref{4.35}), and then $A'$ by $A$ and $\ell_*$ by the smallest non-negative integer such that $2^{-\ell_*} \le \frac{1}{2} \, \delta_N$ in (\ref{4.50}) (so $\frac{\wh{L}_0}{N} / 2^{-\ell_*}$ tends to zero with $N$). If one denotes by $\wh{\cD}^u_N$ and $\wh{\cD}_N$ the disconnection events respectively corresponding to (\ref{4.9}) and (\ref{4.56}) with $\wh{A}_N$ in place of $A_N$, one obtains that for any compact set $A$ in $\IR^d$ and $M > 0$ such that (\ref{4.5}) holds
\begin{align}
&\limsup\limits_N \; \mbox{\f $\dis\frac{1}{N^{d-2}}$} \; \log \IP[\wh{\cD}^u_N] \le - \mbox{\f $\dis\frac{1}{d}$} \;(\sqrt{\ov{u}} - \sqrt{u})^2\, {\rm cap}(A), \;\mbox{when $0 < u < \ov{u}$}, \label{4.61}
\intertext{and}
&\limsup\limits_N \; \mbox{\f $\dis\frac{1}{N^{d-2}}$} \; \log P_x^{\IZ^d} [\wh{\cD}_N] \le - \mbox{\f $\dis\frac{1}{d}$}  \; \ov{u} \,{\rm cap}(A), \;\mbox{for $x \in \IZ^d$}. \label{4.62}
\end{align}

\hfill $\square$
\end{remark}

\appendix
\section{Appendix}
In this appendix we state and prove Proposition \ref{propA.1} that provides a uniform comparison between the discrete capacity of arbitrary finite unions of discrete $L$-boxes at mutual distance $KL$ with the Brownian capacity of their $\IR^d$-filling, when both $K$ and $L$ tend to infinity.

\medskip
We first introduce some notation and recall some facts. We write $\wt{g}(x,y)$, with $x,y \in \IZ^d$, for the Green function of the simple random walk on $\IZ^d$, $d \ge 3$. If we set $\wt{g}(x) = \wt{g}(0,x)$, $x \in \IZ^d$, and $g(x) = g(0,x)$, $x \in \IR^d$ (with $g(\cdot,\cdot)$ the Green function of Brownian motion), so that $\wt{g}(x,y) = \wt{g}(y-x)$, for $x,y \in \IZ^d$, and $g(x,y) = g(y-x)$, for $x,y \in \IR^d$, then one knows by  \cite{Lawl91}, p.~31, that
\begin{equation}\label{A.1}
\mbox{as $x \r \infty$ in $\IZ^d$,  $\wt{g}(x) \sim dg(x)$ \quad (and $g(x) = \frac{1}{2 \pi^{d/2}} \;\Gamma \big(\frac{d}{2} - 1\big) \,|x|^{2-d}$)}.
\end{equation}

\medskip\n
We consider integers $L \ge 1$ and $K \ge 100$ and recall the notation from (\ref{4.12}) concerning $L$-boxes: $B_z = z + [0,L)^d \cap \IZ^d$, for $z \in \IL (= L \IZ^d)$. For $\cC \subseteq \IL$, a non-empty finite subset of $\IL$, we write
\begin{equation}\label{A.2}
C = \mbox{\f $\dis\bigcup\limits_{z \in \cC}$} B_z \quad \mbox{(and as a shorthand $C = \bigcup\limits_{B \in \cC} B$)}.
\end{equation}
We denote by $\wh{B}_z$ the $\IR^d$-filling of $B_z$:
\begin{equation}\label{A.3}
\wh{B}_z = z + [0,L]^d (\subseteq \IR^d),
\end{equation}
and set
\begin{equation}\label{A.4}
\Gamma = \mbox{\f $\dis\bigcup\limits_{z \in \cC}$} \wh{B}_z \quad \mbox{(and write as a shorthand $\Gamma = \bigcup\limits_{B \in \cC} \wh{B}$)}.
\end{equation}

\n
Of special interest for this appendix is the case when the boxes in $\cC$ are at mutual distance at least $KL$, i.e.
\begin{equation}\label{A.5}
\mbox{$z \not= z'$ in $\cC$ implies that $|z - z'|_\infty \ge KL$.}
\end{equation}

\medskip\n
For $F$ finite subset of $\IZ^d$, we let $\wt{e}_F(\cdot)$ and $\wt{\rm cap}(F)$ stand for the respective equilibrium measure and capacity of $F$ attached to the discrete Green function $\wt{g}(\cdot,\cdot)$ (so $\wt{\rm cap}(F) = {\rm cap}_{\IZ^d}(F)$ in the notation of Section 4). One knows that when $K$ is large and (\ref{A.5}) holds, the equilibrium measure of $C$ ``conditioned on being in the $L$-box $B \in \cC$'', it is very close to the normalized equilibrium measure $\ov{\wt{e}}_B = \wt{e}_B / \wt{{\rm cap}}(B)$.

\medskip
More precisely, by Proposition 1.5 of \cite{Szni17}, for any $\delta \in (0,1)$, when $K \ge c(\delta)$, 
\begin{equation}\label{A.6}
\left\{ \begin{array}{l}
\mbox{for any $L \ge 1$, and any $\cC \subseteq \IL$ non-empty finite satisfying (\ref{A.5}),}
\\
\mbox{one has $(1 - \delta) \, \wt{\mu} \le \wt{e}_C \le (1 + \delta) \, \wt{\mu}$, where}
\\[1ex]
\mbox{$\wt{\mu}(y) = \dsl_{B \in \cC} \wt{e}_C(B) \,\ov{\wt{e}}_B(y)$ for $y \in \IZ^d$}.
\end{array}\right.
\end{equation}

\medskip\n
Finally, we recall that in the case of an $L$-box $B$ and its $\IR^d$-filling $\wh{B}$, one has the large $L$ equivalence of capacities, cf.~Lemma 2.2 of \cite{BoltDeus93} and \cite{Spit01}, p.~301, namely 
\begin{equation}\label{A.7}
a_L = d \; \dis\frac{\wt{{\rm cap}}(B)}{{\rm cap}(\wh{B})} \; \underset{L \r \infty}{\longrightarrow} 1.
\end{equation}

\n
The main object of this Appendix is the following strengthening of this asymptotic equivalence:

\begin{prop}\label{propA.1}
\begin{align}
& \liminf\limits_{K,L \r \infty} \;\;  \inf\limits_\cC \; d \; \dis\frac{\wt{{\rm cap}}(C)}{{\rm cap}(\Gamma)} \ge 1,\label{A.8}
\\[2ex]
& \limsup\limits_{K,L \r \infty} \;\;  \sup\limits_\cC \; d \; \dis\frac{\wt{{\rm cap}}(C)}{{\rm cap}(\Gamma)} \le 1,\label{A.9}
\end{align}

\medskip\n
where in (\ref{A.8}) and (\ref{A.9}), $\cC$ varies over the collection of non-empty finite subsets of $\IL (= L\IZ^d)$ satisfying (\ref{A.5}).
\end{prop}

\begin{proof}
We consider $\delta \in (0,1)$ and assume $K \ge c(\delta)$ so that (\ref{A.6}) holds. We also introduce the quantity (recall the notation (\ref{A.3}))
\begin{equation}\label{A.10}
\eta_K = \sup\limits_{L \ge 1} \;\; \sup\limits_{z_1, z_2 \in \IL \atop |z_1-z_2|_\infty \ge KL} \;\; \sup\limits_{x_1\in B_{z_1}, x_2 \in B_{z_2} \atop x'_1 \in \wh{B}_{z_1}, x'_2 \in \wh{B}_{z_2}} \;\;\max\Big\{ \dis\frac{\wt{g}(x_1,x_2)}{dg(x'_1,x'_2)} , \; \dis\frac{dg(x'_1,x'_2)}{\wt{g}(x_1,x_2)}\Big\},
\end{equation}
and note that by (\ref{A.1})
\begin{equation}\label{A.11}
\lim\limits_{K \r \infty} \; \eta_K = 1.
\end{equation}

\medskip\n
We now consider $\cC$, a non-empty subset of $\IL$ satisfying (\ref{A.5}), as well as $\wt{\mu}$ in (\ref{A.6}) and the finite measure $\mu$ on $\IR^d$:
\begin{equation}\label{A.12}
\mu(dy) = d \dsl_{B \in \cC} \wt{e}_C(B) \,\ov{e}_{\wh{B}}(dy), \;\mbox{so} \; \mu(\IR^d) = d \,\wt{e}_C(\IZ^d) = d \, \wt{\rm cap} (C)
\end{equation}

\medskip\n
(and $\ov{e}_{\wh{B}}(dy)$ stands for the normalized equilibrium measure of the $\IR^d$-filling $\wh{B}$ of $B$, i.e.~$\ov{e}_{\wh{B}}(dy) = e_{\wh{B}}(dy) / {\rm cap}(\wh{B})$).

\medskip
For $x \in \Gamma$ (see (\ref{A.4})), so that $x \in \wh{B}_*$ with $B_* \in \cC$, we choose $x_* \in B_*$, such that $|x - x_*|_\infty \le 1$. We then have
\begin{equation}\label{A.13}
\begin{array}{l}
\dis\int g(x,y) \, \mu(dy) = d \, \wt{e}_C(B_*) \, \dis\int g(x,y) \,\ov{e}_{\wh{B}_*}(dy)\;+ 
\\
\dsl_{B \not= B_*}  d \;\wt{e}_C(B) \dis\int g(x,y) \, \ov{e}_{\wh{B}}(dy).
\end{array}
\end{equation}
The first term in the right-hand side of (\ref{A.13}) equals (see (\ref{A.7}) for notation) $d \,\wt{e}_C(B_*) / {\rm cap}(\wh{B}_*) = \wt{e}_C(B_*) / \wt{\rm cap}(B_*) \,a_L$.

\medskip
By (\ref{A.10}), (\ref{A.11}), we see that on the one hand
\begin{equation}\label{A.14}
\begin{array}{lcl}
\dis\int g(x,y) \,\mu(dy) & \stackrel{(\ref{A.10})}{\le} & \dis\frac{\wt{e}_C(B_*)}{\wt{\rm cap}(B_*)} \;a_L + \eta_K \dsl_{B \not= B_*} \wt{e}_C(B) \dsl_y \wt{g}(x_*,y) \,\ov{\wt{e}}_B(y)
\\
\\[-1ex]
& \le & \max (a_L,\eta_K) \dsl_{B \in \cC} \wt{e}_C(B) \dsl_y \wt{g}(x_*,y) \, \ov{\wt{e}}_B(y)
\\
\\[-1ex]
& =& \max(a_L,\eta_K)  \dsl_y \wt{g}(x_*,y) \,\wt{\mu}(y)
\\
& \stackrel{(\ref{A.6})}{\le} & \max(a_L,\eta_K)  \;(1-\delta)^{-1} \dsl_y \wt{g}(x_*,y) \, \wt{e}_C(y) 
\\
\\[-2ex]
& =& \max(a_L,\eta_K)  \;(1-\delta)^{-1}  \;\; \mbox{(since $x_* \in B_* \subseteq C$)}.
\end{array}
\end{equation}
In a similar fashion, we see that
\begin{equation}\label{A.15}
\begin{array}{lcl}
\dis\int g(x,y) \,\mu(dy) &\ge  & \min (a_L,\eta_K^{-1})  \dsl_y \wt{g}(x_*,y) \, \wt{\mu}(y) 
\\
\\[-2ex]
& \stackrel{(\ref{A.6})}{\ge} & \min(a_L,\eta_K^{-1})  \;(1+\delta)^{-1} \dsl_y \wt{g}(x_*,y) \, \wt{e}_C(y) \qquad \;\;
\\
\\[-2ex]
& =& \min(a_L,\eta_K^{-1})  \;(1+\delta)^{-1} .
\end{array}
\end{equation}

\medskip\n
Note that $\mu$ is supported by $\Gamma$, see (\ref{A.12}), and (\ref{A.15}) holds for arbitrary $x$ in $\Gamma$. It then follows by integration of (\ref{A.14}) and (\ref{A.15}) with respect to the equilibrium measure of $\Gamma$ that
\begin{equation}\label{A.16}
\begin{array}{l}
(1-\delta)^{-1} \max(a_L, \eta_K)\,{\rm cap}(\Gamma) \ge\mu(\IR^d) \stackrel{(\ref{A.12})}{=}  
\\
d\; \wt{\rm cap}(C) \ge (1 + \delta)^{-1} \min(a_L, \eta_K^{-1}) \,{\rm cap}(\Gamma).
\end{array}
\end{equation}

\n
Since $\delta \in (0,1)$ is arbitrary, the claims (\ref{A.8}) and (\ref{A.9}) readily follow by (\ref{A.7}) and (\ref{A.11}).
\end{proof}

\end{document}